\documentclass[final,leqno,showlabe]{siamltex}
\usepackage[bookmarksopen,bookmarksopenlevel=0,bookmarksdepth=2]{hyperref}
\usepackage{appendix}
\usepackage{amsmath}
\usepackage{amssymb}
\usepackage{graphicx}
\usepackage{doi}
\usepackage{cleveref}
\usepackage[usenames]{color}
\usepackage{wrapfig}
\usepackage{stmaryrd}
\usepackage{array} 
\usepackage{mathrsfs}
\usepackage{multirow}
 \usepackage{txfonts}
\usepackage{epstopdf}
\allowdisplaybreaks[4]

\SetSymbolFont{stmry}{bold}{U}{stmry}{m}{n}

\usepackage{enumitem}
\setenumerate[1]{itemsep=0pt,partopsep=0pt,parsep=\parskip,topsep=5pt}
\setitemize[1]{itemsep=0pt,partopsep=0pt,parsep=\parskip,topsep=5pt}
\setdescription{itemsep=0pt,partopsep=0pt,parsep=\parskip,topsep=5pt}

\title{Structure-preserving, weighted implicit-explicit schemes for
multi-phase incompressible Navier-Stokes/Darcy coupled nonlocal Allen-Cahn model}

\author{Meng Li, Ke Wang, Nan Wang\thanks{School of Mathematics and Statistics, Zhengzhou University, Zhengzhou 450001, China. 
Corresponding author (*).
Email addresses: limeng@zzu.edu.cn (M. Li), nwang@zzu.edu.cn (N. Wang$^*$).}}

\date{}
\begin{document}
\sloppy
\maketitle

\begin{abstract}
A multitude of substances exist as mixtures comprising multiple chemical components in the natural world. These substances undergo morphological changes under external influences.
In the phase field model coupled with fluid flow, the dynamic movement and evolution of the phase interface intricately interact with the fluid motion.
This article focuses on the N-component models that couple the conservative Allen-Cahn equation with two types of incompressible fluid flow systems: the Navier-Stokes equation and the Darcy equation.
By utilizing the scalar auxiliary variable method and the projection method,
we innovatively construct two types of structure-preserving weighted implicit-explicit schemes for the coupled models, resulting in fully decoupled linear systems and second-order accuracy in time. The schemes are proved to be mass-conservative. In addition, with the application of $G$-norm inspired by the idea of $G$-stability, we rigorously establish its unconditional energy stability. Finally, the performance of the proposed scheme is verified by some numerical simulations.
\end{abstract}




\section{Introduction}

The phase-field method \cite{caginalp1986phase,karma1996phase} is a renowned modeling approach extensively utilized across diverse disciplines, such as materials science \cite{warren1995prediction}, biophysics \cite{shen2019novel} and geophysics \cite{alpak2016phase}, to simulate a wide range of interface dynamics. It can effectively capture a variety of interface phenomena, including morphological evolution \cite{provatas2007using}, phase transitions \cite{chen2008phase}, and phase separation \cite{zhu2016modelling}. However, when the motion of the fluid is also taken into account, the interaction between fluid dynamics and the evolution of the phase interface makes the interface phenomenon complex.
This kind of coupling problem has been paid more and more attention by many researchers in recent years (see \cite{cai2017error,jeong2017conservative,biswas2020maximum}). Compared to single-component systems, multi-component systems have a wider range of applications
due to the prevalence of mixtures composing of two or even multiple components. In multi-phase simulations, a phase field variable is sufficient to describe the behavior of each component in space \cite{kim2012phase,nestler2011phase}.
This phase field variable represents the spatial distribution of the component, including information about density, interface position and shape. Based on the above, we in this work consider two multi-phase incompressible fluid flow systems coupled nonlocal Allen-Cahn (AC) model \cite{allen1979microscopic}, where the fluid motions are described by Navier-Stokes equation \cite{glowinski1992finite} and Darcy equation, respectively.
\par
As the typical  gradient flow models, the AC model and Cahn-Hilliard (CH) model \cite{cahn1958free} have always aroused great interest. While both models adhere to the energy dissipation law, they differ in terms of mass conservation properties. The CH model, serving as an $H^{-1}$-gradient flow of free energy, inherently ensures mass conservation.
On the contrary, the AC model, operating as an $L^2$-gradient flow of free energy, does not conserve mass. A novel modification known as the conservative AC (CAC) model, introduced in \cite{rubinstein1992nonlocal}, addresses this limitation by adding a nonlocal Lagrange multiplier.
The CAC model, upholding the mass conservation while retaining the characteristic of energy dissipation, finds extensive utilization across numerous studies \cite{li2021unconditionally,xu2023high,lee2022high,CHAI2018631}.
\par

In recent years, a large number of numerical methods have been used to solve the AC equation, such as the convex splitting method \cite{guan2014convergent,baskaran2013convergence,shen2012second}, the Invariant Energy Quadratization (IEQ) method \cite{YANG2016294,YANG2019316,ZHANG2020112719} and the scalar auxiliary variable (SAV) method \cite{SHEN2018407,shen2019new,li2020stability}. Additionally, there are a considerable amount of numerical works on the Navier-Stokes equation and the Darcy equation \cite{chu2023optimal,MR3802271,li2021discontinuous}. While the Darcy equation characterizes fluid flow in porous media or the Hele-Shaw cell, both equations share similar forms and fall within the realm of incompressible fluid flow problems in fluid mechanics.
Such problems are often solved using projection method, originally proposed by Chorin in \cite{chorin1968numerical}. In general, the projection method for incompressible flows can be categorized into three types, namely the pressure-correction method, the velocity-correction method, and the consistent splitting method, see \cite{guermond2006overview} and the references therein for details.
Despite the availability of these three projection methods, the pressure-correction method is preferred in more works \cite{han2018second,LI2023107055,HAN2015139,li2022efficient}.
This method separates the velocity field and the pressure field, and couples them through a pressure correction equation, thereby effectively resolving the evolution of the flow field.
\par

The main contributions of this work include:
\begin{itemize}
    \item We innovatively propose a weighted implicit-explicit (IMEX) scheme with arbitrary weighted parameter $ \theta \in [1/2,1] $ for the incompressible fluid flow systems coupled gradient flow model,
    including the Navier-Stokes equations coupled with the CAC model (NS-CAC) and the Darcy equations coupled with the CAC model (D-CAC) \cite{YANG2021113597}. When $ \theta = 1/2$ and $ \theta = 1 $, the scheme degenerates into the CN scheme \cite{han2017numerical,XIA2021113987} and the BDF2 scheme \cite{yang2021novel,yang2022highly}, respectively.
    \item
    To capture generality, our research concentrates on the numerical methods tailored for multi-component problems, encompassing the N-component NS-CAC and N-component D-CAC models. These models play a pivotal role in deciphering intricate fluid dynamics phenomena entailing interactions among multiple components.
    \item
    To numerically solve the two N-component models, we employ the SAV method to deal with the CAC model and the projection method to deal with the two N-component incompressible flow models. In this approach, linear terms are treated implicitly while nonlinear terms are treated explicitly. This yields a completely decoupled structure and an easily implementable algorithm process, as it only necessitates solving several linear equations with constant coefficients at each time step.
    \item
  For any values of the parameter $\theta$, our scheme not only ensures second-order accuracy in time, but also maintains structure-preserving properties, including the mass conservation and energy stability. Different from \cite{yang2021novel,yang2022highly}, due to the involvement of the weighted parameter $\theta$, we will employ $G$-norm to demonstrate the energy stability.
\end{itemize}
\par
The remaining parts of this article are arranged as follows. In the next section, we will briefly introduce N-component NS-CAC model and N-component D-CAC model. In Section \ref{sec-3}, we construct a second-order weighted IMEX scheme and provide detailed implement steps to solve them. Additionally, energy stability and mass conservation are also proved. In Section \ref{sec-4}, several numerical experiments are presented to verify the performance of the proposed scheme, and the conclusions are drawn in Section \ref{sec-5}.

\section{Governing equations}
\label{sec-2}
Before introducing the two models, some notations are given initially.
We assume the models comprise $N$ material components with $N\geq 3$, and
 denote the vector-type concentration of each system by $\phi=(\phi_1,\phi_2,\dots,\phi_N)$ where  $\phi_k=\phi_k({\bf x}, t)$ is the concentration of the $k$-th phase.
Here, ${\bf x} \in \Omega$ and $t \in [0,T]$ represent the spatial and temporal variables, respectively, with $\Omega$ being a smooth, open, and bounded domain in $\mathbb{R}^{d}$ ($d=2,3$), and $T$ the final time.
 We further assume the value of $\phi_k$
 ranges from $0$ to $1$, where $\phi_k=1$, $0<\phi_k<1$, $\phi_k=0$ indicate $\phi_k$ inside the $k$-th phase, at the interface and outside the $k$-th phase, respectively.
 In addition, some positive parameters are always involved, such as the mobility parameter $M$, the fluid viscosity $\nu$, and the hydraulic conductivity $\alpha$. Moreover, we denote the $L^2$-inner product of any two functions $g({\bf x})$ and $h({\bf x})$ by $(g({\bf x}),h({\bf x}))=\int_\Omega g({\bf x})h({\bf x})d{\bf x}$ and the $L^2$-norm of $g({\bf x})$ by $\left \|g({\bf x})\right \|^2=(g({\bf x}),g({\bf x}))$.

\subsection{N-component NS-CAC model}
\label{sec-2.1}
We define the total free energy functional of N-component NS-CAC model as
\begin{equation}
\label{eqn:Energy_NSCAC}
    {\mathcal{E}} ({\bf u},{\bf \phi})=\int_\Omega \frac{1}{2}\left |{\bf u} \right |^2 + \lambda\sum_{k=1}^{N}F(\phi_k)+\frac{\lambda}{2}\sum_{k=1}^{N}|\nabla \phi_k|^2d{\bf x},
\end{equation}
where $F(\phi_k)=\frac{1}{4\epsilon^2}\phi_k^2(1-\phi_k)^2$ is the double-well potential function describing the phase separation, $\epsilon$ stands for the interface thickness, $\lambda$ is a positive constant associated with the surface tension parameter, and ${\bf u}={\bf u}({\bf x}, t)$ represents the velocity field. Moreover, $\phi_k$ holds
\begin{equation}
\label{eqn:N_condition}
    \phi_1 +\phi_2 +\cdots +\phi_N =1.
\end{equation}
By performing the variational derivative of the total energy \eqref{eqn:Energy_NSCAC}, we get the N-component Navier-Stokes-Allen-Cahn model:
\begin{subequations}
\label{eqn:model1}
\begin{align}
  &\frac{\partial\phi_k}{\partial t}+\nabla \cdot ({\bf{u}}\phi_k)+M\mu_k=0, \label{eqn:model1_a}\\
  &\mu_k=\lambda\left(-\Delta \phi_k + f_k( \phi_k ) \right), \ \ k=1,2,\dots,N,\label{eqn:model1_b}\\
  &\frac{\partial \bf u}{\partial t}+{\bf{u}\cdot \nabla \bf u }-\nu\Delta {\bf u} + \nabla p+\sum_{k=1}^{N} \phi _k \nabla \mu_k=0,
  \label{eqn:model1_c}\\
  &\nabla \cdot {\bf u}=0,
  \label{eqn:model1_d}
\end{align}
\end{subequations}
where $\mu_k=\frac{\delta \mathcal{E}}{\delta \phi_k}$ denotes the chemical potential, $p=p({\bf x},t)$ is the pressure field, and $f_k(\phi_k)=\frac{\partial F(\phi_k)}{\partial \phi_k}=\frac{1}{\epsilon^2}\phi_k(\phi_k-\frac{1}{2})(\phi_k-1)$. The boundary conditions are either periodic or the following conditions:
\begin{equation}
\label{eqn:boundary condition}
\left. \nabla \phi _k \cdot {\bf{n}}\right |_{\partial \Omega } =\left. \nabla \mu _k \cdot {\bf{n}}\right |_{\partial \Omega } = \left. \nabla p \cdot {\bf{n}} \right |_{\partial \Omega } = 0,\ \
\left.
\bf{u} \right |_{\partial \Omega }={\bf 0},
\end{equation}
where $\bf{n}$ refers to the unit vector pointing outward from the boundary $\partial \Omega$.

In order to achieve the mass conservation,
the N-component NS-CAC model is obtained by adding a nonlocal Lagrange multiplier term $\frac{1}{|\Omega|}\int_\Omega f_k(\phi_k)d{\bf x}$, i.e.,
\begin{subequations}
\begin{align}
  &\frac{\partial\phi_k}{\partial t}+\nabla \cdot ({\bf{u}}\phi_k)+M\mu_k=0, \\
  &\mu_k=\lambda\left(-\Delta \phi_k + \bar{f}_k( \phi_k )\right), \ \ k=1,2,\dots,N,\\
  &\frac{\partial \bf u}{\partial t}+{\bf{u}\cdot \nabla \bf u }-\nu\Delta {\bf u} + \nabla p+\sum_{k=1}^{N} \phi _k \nabla \mu_k=0,\\
  &\nabla \cdot {\bf u}=0,
\end{align}
\end{subequations}
where $\bar{f}_k( \phi_k )=f_k( \phi_k )-\frac{1}{|\Omega|}\int_\Omega f_k(\phi_k)d{\bf x}$.
Additionally, another Lagrange multiplier term, $\beta =-\frac{1}{N}\sum_{k=1}^{N}\bar{f}_k( \phi_k )$, is introduced to ensure condition \eqref{eqn:N_condition}, thereby resulting in the final N-component NS-CAC model as follows
\begin{subequations}
\label{eqn:rmodel1}
\begin{align}
   &\frac{\partial\phi_k}{\partial t}+\nabla \cdot ({\bf{u}}\phi_k)+M\mu_k=0, \label{eqn:rmodel1_a}\\
   &\mu_k=\lambda\left(-\Delta \phi_k + \bar{f}_k( \phi_k )+\beta \right), \ \ k=1,2,\dots,N, \label{eqn:rmodel1_b}\\
   &\frac{\partial \bf u}{\partial t}+{\bf{u}\cdot \nabla \bf u }-\nu\Delta {\bf u} + \nabla p+\sum_{k=1}^{N} \phi _k \nabla \mu_k=0,
   \label{eqn:rmodel1_c}\\
   &\nabla \cdot {\bf u}=0.
   \label{eqn:rmodel1_d}
\end{align}
\end{subequations}
The boundary conditions are still either periodic or \eqref{eqn:boundary condition} and $\phi_k$ also satisfies the condition \eqref{eqn:N_condition}. By taking the $L^2$-inner products of \eqref{eqn:rmodel1_a} and \eqref{eqn:rmodel1_b} with $1$ and employing the integration by parts, the divergence theorem and the boundary conditions \eqref{eqn:boundary condition}, we readily derive mass conservation property for the system \eqref{eqn:rmodel1}:
\begin{equation}
\label{eqn:mass conservation}
\frac{d}{dt}\int_\Omega \phi_k d{\bf x}=0,
\end{equation}
where $(\bar{f}_k( \phi_k ),1)=0$ and $(\beta,1)=0$ are used because of the definitions of $\bar{f}_k( \phi_k )$ and $\beta$.\par
In addition, the system \eqref{eqn:rmodel1} follows a dissipation law of energy. Indeed, by computing the $L^2$-inner products of \eqref{eqn:rmodel1_a} and \eqref{eqn:rmodel1_b}  with $\mu_k$ and $ \frac{\partial \phi_k}{\partial t} $, respectively, we have
\begin{align}
\label{eqn:rmodel1_a1}
   &\left (\frac{\partial \phi_k}{\partial t},\mu _k\right )+\left(\nabla \cdot ({\bf u}\phi_k),\mu_k\right)+M\left \|\mu_k\right \|^2=0,\\
   \label{eqn:rmodel1_b1}
   &\left (\mu _k,\frac{\partial \phi_k}{\partial t}\right )=\lambda \left(-\Delta \phi_k,\frac{\partial \phi_k}{\partial t}\right )+\lambda\left (\bar{f}_k( \phi_k ),\frac{\partial \phi_k}{\partial t}\right)
   +\lambda\left (\beta,\frac{\partial \phi_k}{\partial t}\right).
\end{align}
Due to the mass conservation property \eqref{eqn:mass conservation}, we can derive that
\begin{align}
    \left(\bar{f}_k( \phi_k ),\frac{\partial \phi_k}{\partial t}\right)
    &=\int_{\Omega}\bar{f}_k( \phi_k )\frac{\partial \phi_k}{\partial t}d{\bf x}
    =\int_{\Omega}\left(f_k( \phi_k )-\frac{1}{|\Omega|}\int_\Omega f_k(\phi_k)d{\bf x}\right)\frac{\partial \phi_k}{\partial t}d{\bf x}
    \nonumber\\
    &=\int_{\Omega}f_k( \phi_k )\frac{\partial \phi_k}{\partial t}d{\bf x}
    -\frac{1}{|\Omega|}\int_\Omega f_k(\phi_k)d{\bf x}\int_{\Omega}\frac{\partial \phi_k}{\partial t}d{\bf x}
    \nonumber\\
    &=\int_{\Omega}f_k( \phi_k )\frac{\partial \phi_k}{\partial t}d{\bf x}.
    \label{eqn:rmodel1_b1_1}
\end{align}
Combining \eqref{eqn:rmodel1_a1}-\eqref{eqn:rmodel1_b1_1} for $k=1,2,\dots,N$ and using the integration by parts, we obtain
\begin{equation}
\label{eqn:rmodel1_a1b1}
    \frac{d}{dt}\int_\Omega\frac{\lambda}{2}\sum_{k=1}^{N}|\nabla \phi_k|^2+\lambda\sum_{k=1}^{N}F(\phi_k)d{\bf x}+\sum_{k=1}^{N}(\nabla \cdot ({\bf u}\phi_k),\mu_k)
    =-M\sum_{k=1}^{N}\left \|\mu_k\right \|^2,
\end{equation}
where $\sum_{k=1}^{N}\left (\beta,\frac{\partial \phi_k}{\partial t}\right)=\left (\beta,\sum_{k=1}^{N}\frac{\partial \phi_k}{\partial t}\right)=0$ holds with the condition \eqref{eqn:N_condition}.
By calculating the $L^2$-inner product of \eqref{eqn:rmodel1_c} with $\bf u$ and utilizing the properties $({\bf{u}\cdot \nabla \bf u },{\bf u})=0$ and $(\nabla p,{\bf u})=0$, we get
\begin{equation}
\label{eqn:rmodel1_c1}
    \frac{d}{dt}\int_\Omega \frac{1}{2}|{\bf u}|^2d{\bf x}-(\nu \Delta{\bf u},{\bf u}) +\sum_{k=1}^{N} (\phi _k \nabla \mu_k,{\bf u})=0.
\end{equation}
Through the application of integration by parts, the advection term and the surface tension term can be cancelled out, i.e.
\begin{equation}
\label{eqn:ad-sur}
    (\nabla \cdot ({\bf u}\phi_k),\mu_k)+(\phi _k \nabla \mu_k,{\bf u})=0.
\end{equation}
Adding \eqref{eqn:rmodel1_a1b1} to \eqref{eqn:rmodel1_c1} and using \eqref{eqn:ad-sur}, we arrive at the following energy dissipation property:
\begin{equation}
    \frac{d}{dt}{\mathcal{E}} ({\bf u,\phi})=-M\sum_{k=1}^{N}\left \|\mu_k\right \|^2-\left \|\sqrt{\nu}\nabla {\bf u}\right \|^2 \leq 0.
\end{equation}
\begin{remark}
    In the proof of the energy dissipation property provided above, it's evident that both the advection term and the surface tension term have no contribution on the total energy \eqref{eqn:Energy_NSCAC} since \eqref{eqn:ad-sur} consistently holds. This characteristic is referred to as the "zero-energy-contribution" feature.
\end{remark}

\begin{remark}
    When dealing with a system consisting of two material components, namely component 1 and component 2, we only need to use a phase field variable $\phi$ to describe it. Here, $\phi=1$ represents $\phi$ inside the component 1 while $\phi=0$ signifies $\phi$ outside component 1, that is, inside the component 2. Consequently, there is no need for the addition of the Lagrange multiplier term $\beta$. In this paper, we don't have to present theoretical analysis of two-component systems owing to the complete similarity of the theories. Nevertheless, we will provide some numerical experiments.
\end{remark}

\subsection{N-component D-CAC model}
\label{sec-2.2}
The total free energy functional of N-component D-CAC model is defined as
\begin{equation}
\label{eqn:Energy_DCAC}
    E ({\bf u,\phi})=\int_\Omega \frac{\tau}{2}\left |{\bf u} \right |^2 + \lambda\sum_{k=1}^{N}F(\phi_k)+\frac{\lambda}{2}\sum_{k=1}^{N}|\nabla \phi_k|^2d{\bf x}.
\end{equation}
By taking the variational derivative of the total energy \eqref{eqn:Energy_DCAC}, and adding the two Lagrange multiplier terms $\frac{1}{|\Omega|}\int_\Omega f_k(\phi_k)d{\bf x}$ and $\beta$,  we obtain the N-component D-CAC model:
\begin{subequations}
\label{eqn:rmodel2}
\begin{align}
   &\frac{\partial\phi_k}{\partial t}+\nabla \cdot ({\bf{u}}\phi_k)+M\mu_k=0, \label{eqn:rmodel2_a}\\
   &\mu_k=\lambda\left(-\Delta \phi_k + \bar{f}_k( \phi_k ) +\beta\right), \ \ k=1,2,\dots,N,  \label{eqn:rmodel2_b}\\
   &\tau\frac{\partial \bf u}{\partial t}+\alpha \nu{\bf u} + \nabla p+\sum_{k=1}^{N} \phi _k \nabla \mu_k=0,
   \label{eqn:rmodel2_c}\\
   &\nabla \cdot {\bf u}=0,
   \label{eqn:rmodel2_d}
\end{align}
\end{subequations}
where $\tau$ is a positive constant,
and $\phi_1,\phi_2,\dots,\phi_N$ are also restricted by the condition \eqref{eqn:N_condition}. Different from N-component NS-CAC model, ${\bf u}={\bf u}({\bf x}, t)$ represents the nondimensionalized seepage velocity in porous medium or a Hele–Shaw cell. The boundary conditions are either periodic or the following:
\begin{equation}
    \left. \nabla \phi _k \cdot {\bf{n}}\right |_{\partial \Omega } =\left. \nabla \mu _k \cdot {\bf{n}}\right |_{\partial \Omega } = \left. \nabla p \cdot {\bf{n}} \right |_{\partial \Omega } = 0,\ \
    \left.{\bf u} \right |_{\partial \Omega }={\bf 0}.
\end{equation}

Obviously, the system \eqref{eqn:rmodel2} satisfies the mass conservation \eqref{eqn:mass conservation}. Moreover, we have the following energy dissipation property:
\begin{equation}
    \frac{d}{dt}E({\bf u,\phi})=-M\sum_{k=1}^{N}\left \|\mu_k\right \|^2-\alpha\left \|\sqrt{\nu}{\bf u}\right \|^2 \leq 0.
\end{equation}

\section{Numerical schemes}
\label{sec-3}
In this section, we construct second-order weighted IMEX schemes based on the SAV and projection methods. Moreover, we prove the mass conservation of the schemes, and verify their energy dissipation properties through $G$-norm inspired by $G$-stability.
\par
\subsection{G-norm for the energy stability}
In this subsection, in order to obtain the energy stability properties of the numerical schemes, we recall the following $G$-norm by the $G$-stability concept \cite{wang2023unconditionally}.
When $\theta \in [1/2,1]$, for any ${\bf w,v}\in L^2(\Omega)$, we define the $G$-norm as
\begin{equation}
\left\|
   \begin{bmatrix}
      {\bf w}\\
      {\bf v}
   \end{bmatrix}
\right\|_{\bf G}^2:=
\left(
\begin{bmatrix}
    {\bf w}\\
      {\bf v}
 \end{bmatrix},G
 \begin{bmatrix}
    {\bf w}\\
      {\bf v}
 \end{bmatrix}
 \right),
\end{equation}
where
\begin{align}
  &G=\begin{bmatrix}
  \frac{\theta(2\theta+3)}{2}&-\frac{(\theta+1)(2\theta-1)}{2} \\
  -\frac{(\theta+1)(2\theta-1)}{2}&\frac{\theta(2\theta-1)}{2}
\end{bmatrix}.\label{G}
\end{align}
Obviously, when $\theta \in (1/2,1]$, $G$ is a symmetric positive matrix, and when $\theta=1/2$, the norm will degrade into the $L^2$-norm, that is,
\begin{equation}
    \left\|
    \begin{bmatrix}
      {\bf w}\\
      {\bf v}
    \end{bmatrix}
    \right\|_{\bf G}^2
    =\left\|\bf w\right\|^2.
\end{equation}
Furthermore, the above norm can be expressed in detail as
\begin{equation}
\left\|
   \begin{bmatrix}
      {\bf w}\\
      {\bf v}
   \end{bmatrix}
\right\|_{\bf G}^2
=\frac{\theta(2\theta+3)}{2}\left\|{\bf w}\right\|^2+\frac{\theta(2\theta-1)}{2}\left\|{\bf v}\right\|^2-(\theta+1)(2\theta-1)({\bf w},{\bf v}).
\end{equation}
\par
In addition, for any $w,v\in L^\infty([0,T])$, we define the following quadratic form:
\begin{equation}
\left |
   \begin{bmatrix}
      w\\
      v
   \end{bmatrix}
\right |_{\bf G}^2:=
\begin{bmatrix}
    w,v
 \end{bmatrix}
 G
 \begin{bmatrix}
    w\\
      v
 \end{bmatrix}.
\end{equation}
i.e.
\begin{align}
&\left |
   \begin{bmatrix}
      w\\
      v
   \end{bmatrix}
\right |_{\bf G}^2
=\frac{\theta(2\theta+3)}{2}w^2+\frac{\theta(2\theta-1)}{2}v^2-(\theta+1)(2\theta-1)wv.
\end{align}
\par
According to the definition of $G$-norm, we have the following lemma.
\begin{lemma}
\label{lemma_GF-norm}
    When $\theta \in [1/2,1]$, for any sequence of function $\{{\bf w}^n\}$, it holds
    \begin{align}
        &\left(\frac{2\theta+1}{2}{\bf w}^{n+1}-2\theta {\bf w}^n+\frac{2\theta-1}{2}{\bf w}^{n-1},\theta {\bf w}^{n+1}+(1-\theta){\bf w}^n\right)
        \nonumber\\
        =&\frac{1}{2}\left\|
        \begin{bmatrix}
            {\bf w}^{n+1}\\
            {\bf w}^n
        \end{bmatrix}
        \right\|_{\bf G}^2
        -\frac{1}{2}\left\|
        \begin{bmatrix}
            {\bf w}^n\\
            {\bf w}^{n-1}
        \end{bmatrix}
        \right\|_{\bf G}^2
        +\frac{\theta(2\theta-1)}{4}\left\|{\bf w}^{n+1}-2{\bf w}^n+{\bf w}^{n-1}\right\|^2.
    \end{align}
\end{lemma}

\begin{remark}
\label{remark_GF}
   Similar to Lemma \ref{lemma_GF-norm}, for any sequence of function $\{w^n\}$, we can get the following identity
    \begin{align}
        &\left(\frac{2\theta+1}{2}w^{n+1}-2\theta w^n+\frac{2\theta-1}{2}w^{n-1}\right)\left(\theta w^{n+1}+(1-\theta)w^n\right)
        \nonumber\\
        =&\frac{1}{2}\left|
        \begin{bmatrix}
            w^{n+1}\\
            w^n
        \end{bmatrix}
        \right|_{\bf G}^2
        -\frac{1}{2}\left|
        \begin{bmatrix}
            w^n\\
            w^{n-1}
        \end{bmatrix}
        \right|_{\bf G}^2
        +\frac{\theta(2\theta-1)}{4}\left|w^{n+1}-2w^n+w^{n-1}\right|^2.
    \end{align}
\end{remark}

\subsection{N-component NS-CAC model}
\label{sec-3.1}
We introduce two scalar auxiliary variables
\begin{equation}
    r(t)=\sqrt{\int_\Omega \sum_{k=1}^{N}F(\phi_k)d{\bf x}+C},\quad\quad q(t)\equiv 1,
\end{equation}
where $C$ is a positive constant to ensure $\int_{\Omega}\sum_{k=1}^{N}F(\phi_k)d{\bf x}+C>0$. Then, the system \eqref{eqn:rmodel1} can be reformulated as
\begin{subequations}
\label{eqn:rrmodel1}
\begin{align}
&\frac{\partial\phi_k}{\partial t}+q\nabla \cdot ({\bf{u}}\phi_k)+M\mu_k=0, \label{eqn:rrmodel1_a}\\
&\mu_k=\lambda\left(-\Delta \phi_k + (\bar{H}_k+\gamma) r\right), \label{eqn:rrmodel1_b}\\
&\frac{dr}{dt}=\frac{1}{2}\sum_{k=1}^{N}\int_\Omega {\bar H}_k\frac{\partial \phi_k}{\partial t}d{\bf x},
\label{eqn:rrmodel1_c}\\
&\frac{\partial \bf u}{\partial t}+q{\bf{u}\cdot \nabla \bf u }-\nu\Delta {\bf u} + \nabla p+q\sum_{k=1}^{N} \phi _k \nabla \mu_k=0,
\label{eqn:rrmodel1_d}\\
&\nabla \cdot {\bf u}=0,
\label{eqn:rrmodel1_e}\\
&\frac{dq}{dt}=\sum_{k=1}^{N}\int_\Omega \nabla \cdot ({\bf u}\phi_k)\mu_kd{\bf x}+\sum_{k=1}^{N}\int_\Omega\phi _k \nabla \mu_k\cdot {\bf u}d{\bf x}+\int_\Omega {\bf{u}\cdot \nabla \bf u }\cdot {\bf u}d{\bf x},
\label{eqn:rrmodel1_f}
\end{align}
\end{subequations}
where
\begin{equation}
    H_k=\frac{f_k(\phi _k)}{\sqrt{\int_\Omega \sum_{k=1}^{N}F(\phi_k)d{\bf x}+C}},\ \
    \bar{H}_k=H_k-\frac{1}{|\Omega|}\int_\Omega H_kd{\bf x},\ \
    \gamma=-\frac{1}{N}\sum_{k=1}^{N}\bar{H}_k.
    \label{eqn:rrmodel1_g}
\end{equation}
To derive \eqref{eqn:rrmodel1_f}, we actually use
\begin{equation}
    \left(\nabla \cdot ({\bf u}\phi_k),\mu_k \right)
    +\left(\phi _k \nabla \mu_k,{\bf u} \right)=0,\quad
    \left({\bf{u}\cdot \nabla \bf u},{\bf u} \right)=0,
\end{equation}
which could be obtained by applying the divergence theorem, the integration by parts, \eqref{eqn:rrmodel1_e} and the boundary conditions \eqref{eqn:boundary condition}. In what follows, we prove that the system \eqref{eqn:rrmodel1} satisfies the conservation
 of mass and stability of energy.
\begin{theorem}\label{rrmodel1_mass}
The solutions of the system \eqref{eqn:rrmodel1} satisfy the mass conservation of each phase.
\end{theorem}

\begin{proof}
  By taking the $L^2$-inner products of \eqref{eqn:rrmodel1_a} and \eqref{eqn:rrmodel1_b} with 1, we get
  \begin{align}
      &\left(\frac{\partial\phi_k}{\partial t},1\right )+\left(q\nabla \cdot ({\bf{u}}\phi_k),1\right)+M(\mu_k,1)=0, \label{eqn:rrmodel1_a1_1}\\
      &(\mu_k,1)=\lambda(-\Delta \phi_k,1) +\lambda( \bar{H}_kr ,1)
      +\lambda( \gamma r,1). \label{eqn:rrmodel1_b1_1}
  \end{align}
  From \eqref{eqn:rrmodel1_g}, we have
  \small{
  \begin{align}\label{eqn:addxxx1}
  (\bar{H}_k,1)=\int_\Omega \left(H_k-\frac{1}{|\Omega|}\int_\Omega H_kd{\bf x}\right)d{\bf x}=0,~~~(\gamma,1)
  =(-\frac{1}{N}\sum_{k=1}^{N}\bar{H}_k,1)
  =-\frac{1}{N}\sum_{k=1}^{N}(\bar{H}_k,1)=0.
  \end{align}
  }
  Obviously, it follows that $(q\nabla \cdot ({\bf{u}}\phi_k),1)=0$ and $(-\Delta \phi_k,1)=0$. Hence, substituting \eqref{eqn:rrmodel1_b1_1} into  \eqref{eqn:rrmodel1_a1_1} and using \eqref{eqn:addxxx1}, we obtain
\begin{align}
\frac{d}{dt} \int_\Omega \phi_k d{\bf x}=\int_\Omega\frac{\partial\phi_k}{\partial t}d{\bf x}=0.
\end{align}
We have completed the proof.
\end{proof}

\begin{theorem}
\label{rrmodel1_energy}
    The system \eqref{eqn:rrmodel1} satisfies  the energy dissipation law.
\end{theorem}

\begin{proof}
    By taking the $L^2$-inner products of \eqref{eqn:rrmodel1_a} with $\mu_k$ and \eqref{eqn:rrmodel1_b} with $ \frac{\partial \phi_k}{\partial t} $, we have
\begin{align}
    &\left (\frac{\partial \phi_k}{\partial t},\mu _k\right )+(q\nabla \cdot ({\bf u}\phi_k),\mu_k)+M\left \|\mu_k\right \|^2=0,\label{eqn:rrmodel1_a1}\\
    &\left (\mu _k,\frac{\partial \phi_k}{\partial t}\right )=\lambda \left(-\Delta \phi_k,\frac{\partial \phi_k}{\partial t}\right )+\lambda\left (\bar{H}_kr ,\frac{\partial \phi_k}{\partial t}\right)
    +\lambda\left (\gamma r ,\frac{\partial \phi_k}{\partial t}\right).\label{eqn:rrmodel1_b1}
\end{align}
Combining \eqref{eqn:rrmodel1_b1} with \eqref{eqn:rrmodel1_a1} for $k=1,2,\dots,N$ and using the integration by parts, we obtain
\begin{equation}
\label{eqn:rrmodel1_a1b1}
    \frac{d}{dt}
    \left(\frac{\lambda}{2}\sum_{k=1}^{N}\left \|\nabla \phi_k\right \|^2\right)
    +\lambda\sum_{k=1}^{N}\left (\bar{H}_kr ,\frac{\partial \phi_k}{\partial t}\right)+q\sum_{k=1}^{N}(\nabla \cdot ({\bf u}\phi_k),\mu_k)=-M\sum_{k=1}^{N}\left \|\mu_k\right \|^2,
\end{equation}
where $\sum_{k=1}^{N}\left (\gamma ,\frac{\partial \phi_k}{\partial t}\right)=\left (\gamma,\sum_{k=1}^{N}\frac{\partial \phi_k}{\partial t}\right)=0$.
By multiplying \eqref{eqn:rrmodel1_c} with $2\lambda r $, we get
\begin{equation}
\label{eqn:rrmodel1_c1}
    \lambda\frac{d}{dt}(|r|^2)=\lambda\sum_{k=1}^{N}\int_\Omega {\bar H}_kr\frac{\partial \phi_k}{\partial t}d{\bf x}.
\end{equation}
By taking the $L^2$-inner product of \eqref{eqn:rrmodel1_d} with $\bf u$ and using \eqref{eqn:rrmodel1_e}, we obtain
\begin{equation}
\label{eqn:rrmodel1_d1}
    \frac{d}{dt} \left(\frac{1}{2}\left\|{\bf u}\right \|^2\right)+q({\bf u\cdot \nabla u},{\bf u})+\left\|\sqrt{\nu} \nabla{\bf u}\right \|^2+(\nabla p,{\bf u})+q\sum_{k=1}^{N} (\phi _k \nabla \mu_k,{\bf u})=0.
\end{equation}
Multiplying \eqref{eqn:rrmodel1_f} with $ q $ obtains
\begin{equation}
\label{eqn:rrmodel1_f1}
    \frac{d}{dt}\left (\frac{|q|^2}{2}\right )=q\sum_{k=1}^{N}\int_\Omega \nabla \cdot ({\bf u}\phi_k)\mu_kd{\bf x}+q\sum_{k=1}^{N}\int_\Omega\phi _k \nabla \mu_k\cdot {\bf u}d{\bf x}+q\int_\Omega {\bf{u}\cdot \nabla \bf u }\cdot {\bf u}d{\bf x}.
\end{equation}
Combining \eqref{eqn:rrmodel1_a1b1}-\eqref{eqn:rrmodel1_f1} and using \eqref{eqn:rrmodel1_e} for the pressure term, we obtain the energy dissipative law as
\begin{align}
    \frac{d}{dt}{\mathcal{E}}({\bf u,\phi})&=\frac{d}{dt}\left(\frac{\lambda}{2}\sum_{k=1}^{N}\left \|\nabla \phi_k\right \|^2\right)+\frac{d}{dt}\left(\frac{1}{2}\left\|{\bf u}\right\|^2\right)+ \frac{d}{dt}\left(\frac{|q|^2}{2}\right)+\lambda\frac{d}{dt}(|r|^2)
    \nonumber\\
    &=-M\sum_{k=1}^{N}\left \|\mu_k\right \|^2-\left \|\sqrt{\nu}\nabla {\bf u}\right \|^2 \leq 0.
\end{align}
Therefore, we complete the proof.
\end{proof}
\par
Let $\Delta t=T/{N_t}$ be a time step size, where $N_t$ is the
total number of time steps. Set $t^n=n\Delta t$ for $0\leq n \leq {N_t}$. We denote $(\cdot)^n$ be the numerical approximation at $n$th time level. Then we propose a temporally second-order weighted IMEX scheme.\par
Assuming \small{$({\bf \tilde u},{\bf u},\phi_k,r,q,\mu,p)^{n}$} and \small{$({\bf \tilde u},{\bf u},\phi_k,r,q,\mu,p)^{n-1}$}  are given, we update $({\bf \tilde u},{\bf u},\phi_k,r,q,\mu,p)^{n+1}$  by
\begin{subequations}
 \label{eqn:rrrmodel1}
 \begin{align}
    &\frac{\frac{2\theta+1}{2}\phi_k^{n+1}-2\theta\phi_k^n+\frac{2\theta-1}{2}\phi_k^{n-1}}{\Delta t}+q^{n+\theta}\nabla\cdot({\bf u}^{*}\phi_k^{*})+M\mu_k^{n+\theta}=0,\label{eqn:rrrmodel1_a}\\
    &\mu_k^{n+\theta}
    =\lambda\left(-\Delta\phi_k^{n+\theta}
    +(\bar{H}_k^{*}+\gamma^{*})
    r^{n+\theta}\right),
    \label{eqn:rrrmodel1_b}\\
    &\frac{\frac{2\theta+1}{2}r^{n+1}-2\theta r^n+\frac{2\theta-1}{2}r^{n-1}}{\Delta t}=\frac{1}{2}\sum_{k=1}^{N}\int_\Omega \bar{H}_k^{*}\frac{\frac{2\theta+1}{2}\phi_k^{n+1}-2\theta\phi_k^n+\frac{2\theta-1}{2}\phi_k^{n-1}}{\Delta t}d{\bf x},
    \label{eqn:rrrmodel1_c}\\
    &\frac{\frac{2\theta+1}{2}{\bf \tilde u}^{n+1}-2\theta {\bf u}^n+\frac{2\theta-1}{2}{\bf u}^{n-1}}{\Delta t}+q^{n+\theta}{\bf u}^{*}\cdot \nabla{\bf u}^{*}-\nu\Delta {\bf \tilde u}^{n+\theta}+\nabla p^n+q^{n+\theta}\sum_{k=1}^{N}\phi_k^{*}\nabla \mu _k^{*}=0,\label{eqn:rrrmodel1_d}\\
    &\frac{\frac{2\theta+1}{2}{\bf u}^{n+1}-\frac{2\theta+1}{2}{\bf \tilde u}^{n+1}}{\Delta t}+(\nabla p^{n+\theta}-\nabla p^{n})=0,\label{eqn:rrrmodel1_e}\\
    &\nabla \cdot {\bf u}^{n+1}=0,
    \label{eqn:rrrmodel1_f}\\
    &\frac{\frac{2\theta+1}{2}q^{n+1}-2\theta q^n+\frac{2\theta-1}{2}q^{n-1}}{\Delta t}
    =\sum_{k=1}^{N}\int_\Omega \nabla \cdot ({\bf u}^{*}\phi_k^{*})\mu_k^{n+\theta}d{\bf x}
    \nonumber\\
    & \qquad \qquad \qquad \qquad  \qquad ~~~+\sum_{k=1}^{N}\int_\Omega\phi _k^{*} \nabla \mu_k^{*}\cdot {\bf \tilde u}^{n+\theta}d{\bf x}
    +\int_\Omega {\bf u}^{*}\cdot \nabla {\bf u }^{*}\cdot {\bf \tilde u}^{n+\theta}d{\bf x},
    \label{eqn:rrrmodel1_g}
 \end{align}
\end{subequations}
where
\begin{align}
  &(\cdot)^{n+\theta}=\theta (\cdot)^{n+1}+(1-\theta) (\cdot)^{n},
  \label{eqn:rrrmodel1_h}\\
  &(\cdot)^{*}=(1+\theta) (\cdot)^{n}-\theta (\cdot)^{n-1},
  \label{eqn:rrrmodel1_i}\\
  &{\bf \tilde u}^{n+\theta}=\theta {\bf \tilde u}^{n+1}+(1-\theta) {\bf u}^{n},
  \label{eqn:rrrmodel1_j}\\
  &H_k^{*}=\frac{f_k(\phi_k^{*})}{\sqrt{\int_\Omega \sum_{k=1}^{N}F(\phi_k^{*})d{\bf x}+C}},
  \label{eqn:rrrmodel1_k}\\
  &\bar{H}_k^{*}=H_k^{*}-\frac{1}{|\Omega|} \int_\Omega H_k^{*}d{\bf x},
  \label{eqn:rrrmodel1_l}\\
  &\gamma^{*}=-\frac{1}{N}\sum_{k=1}^{N}\bar{H}_k^{*}.
  \label{eqn:rrrmodel1_m}
\end{align}
The boundary conditions are periodic or as follows,
\begin{align}
    {\bf u}^{n+1}\cdot {\bf n}=0,\quad {\bf \tilde u}^{n+1}={\bf 0},\quad \nabla\phi_k^{n+1}\cdot{\bf n}=\nabla\mu_k^{n+1}\cdot{\bf n}=\nabla p^{n+1}\cdot{\bf n}=0,\quad \text{on}~~\partial \Omega.
\end{align}
Here, ${\bf \tilde u}^{n+1}$ is an intermediate field that does not normally satisfy the divergence-free condition. We give a detailed implementation of the scheme in Appendix \ref{A.1}.
\begin{remark}
\label{Remark initial steps}
    In order to compute the time discrete system \eqref{eqn:rrrmodel1}, we need the values of all variables at $t=t^1$. To this end, for the given initial condition ${\bf \tilde u}^0$, ${\bf u}^0$, $\phi_k^0$, $r^0$, $q^0$, $\mu_k^0$, $p^0$, we construct the following first-order scheme
    \begin{subequations}
       \begin{align}
          &\frac{\phi_k^1-\phi_k^0}{\Delta t}+q^{\theta}\nabla\cdot({\bf u}^{0}\phi_k^{0})+M\mu_k^{\theta}=0,
          \label{eqn:rrrrmodel1_a}\\
          &\mu_k^{\theta}=\lambda\left(-\Delta\phi_k^{\theta}+(\bar{H}_k^{0}+\gamma^{0})r^{\theta}\right),
          \label{eqn:rrrrmodel1_b}\\
          &\frac{r^{1}-r^0}{\Delta t}=\frac{1}{2}\sum_{k=1}^{N}\int_\Omega \bar{H}_k^{0}\frac{\phi_k^{1}-\phi_k^0}{\Delta t}d{\bf x},
          \label{eqn:rrrrmodel1_c}\\
          &\frac{{\bf \tilde u}^{1}-{\bf u}^0}{\Delta t}+q^{\theta}{\bf u}^{0}\cdot \nabla{\bf u}^{0}-\nu\Delta {\bf \tilde u}^{\theta}+\nabla p^0+q^{\theta}\sum_{k=1}^{N}\phi_k^{0}\nabla \mu _k^{0}=0,
          \label{eqn:rrrrmodel1_d}\\
          &\frac{{\bf u}^{1}-{\bf \tilde u}^{1}}{\Delta t}+(\nabla p^{\theta}-\nabla p^{0})=0,
          \label{eqn:rrrrmodel1_e}\\
          &\nabla \cdot {\bf u}^{1}=0,
          \label{eqn:rrrrmodel1_f}\\
          &\frac{q^{1}-q^0}{\Delta t}=\sum_{k=1}^{N}\int_\Omega \nabla \cdot ({\bf u}^{0}\phi_k^{0})\mu_k^{\theta}d{\bf x}+\sum_{k=1}^{N}\int_\Omega\phi _k^{0} \nabla \mu_k^{0}\cdot {\bf \tilde u}^{\theta}d{\bf x}+\int_\Omega {\bf u}^{0}\cdot \nabla {\bf u }^{0}\cdot {\bf \tilde u}^{\theta}d{\bf x},
          \label{eqn:rrrrmodel1_g}
        \end{align}
    \end{subequations}
    where
    $H_k^0=H_k(\phi_k^0)$, $\bar{H}_k^0=H_k^0-\frac{1}{|\Omega|} \int_\Omega H_k^0d{\bf x}$, $\gamma^0=-\frac{1}{N}\sum_{k=1}^{N}\bar{H}_k^0$.\par
    Taking the $L^2$-inner products of \eqref{eqn:rrrrmodel1_a} and \eqref{eqn:rrrrmodel1_b} with $1$, respectively, we have
    \begin{align}
          &\left(\frac{\phi_k^1-\phi_k^0}{\Delta t}, 1\right)
          +\left(q^{\theta}\nabla\cdot({\bf u}^{0}\phi_k^{0}), 1\right)
          +M(\mu_k^{\theta}, 1)=0,\\
          &(\mu_k^{\theta}, 1)
          =\lambda(-\Delta\phi_k^{\theta}, 1)
          +\lambda \left(\bar{H}_k^{0}r^{\theta}, 1\right)
          +\lambda \left(\gamma^{0}r^{\theta}, 1\right).
    \end{align}
    Applying the integration by parts and the definitions of $\bar{H}_k^0$ and $\gamma^0$ gets
    \begin{equation}
        (\bar{H}_k^0,1)=0,\ \ (\gamma^0,1)=0,\ \ \left(\nabla\cdot({\bf u}^0\phi_k^0),1\right)=0, \ \ (-\Delta\phi_k^{\theta},1)=0.
    \end{equation}
    Therefore, we obtain $\int_{\Omega}\phi_k^1d{\bf x}=\int_{\Omega}\phi_k^0d{\bf x}$.


\end{remark}
\begin{remark}
\label{remark_Poisson}
    During solving $p^{n+1}$, we need solve a Poisson equation which is obtained by applying the divergence operator to \eqref{eqn:rrrmodel1_e} and using \eqref{eqn:rrrmodel1_f}, i.e.,
    \begin{equation}
    \label{eqn:Poisson}
        \Delta p^{n+1}=\Delta p^n+\frac{2\theta+1}{2\theta\Delta t}\nabla \cdot {\bf \tilde u}^{n+1}.
    \end{equation}
    However, in fact, when the boundary condition is the homogeneous Neumann boundary condition or periodic boundary condition, Poisson equation does not have a unique solution. Instead, its solution is unique up to a constant. Generally, we use the following two methods \cite{choi2015numerical} to gain the unique solution, i.e.,
    \begin{itemize}
        \item imposing the Dirichlet condition at a single point;
        \item forcing its summation to be zero.
    \end{itemize}
    After $p^{n+1}$ is uniquely solved, we can naturally update ${\bf u}^{n+1}$ from \eqref{eqn:rrrmodel1_e}, i.e.,
    \begin{equation}
        {\bf u}^{n+1}=-\frac{2\theta\Delta t}{2\theta+1}(\nabla p^{n+1}-\nabla p^{n})+{\bf \tilde u}^{n+1}.
    \end{equation}
\end{remark}
Next, we give two theorems to illustrate the mass conservation and energy stability results of the scheme \eqref{eqn:rrrmodel1_a}-\eqref{eqn:rrrmodel1_g}.

\begin{theorem}
\label{rrrmodel1_mass}The discrete scheme \eqref{eqn:rrrmodel1_a}-\eqref{eqn:rrrmodel1_g} conserves the mass of the each phase.
\end{theorem}
\begin{proof}
    Taking the $L^2$-inner product of \eqref{eqn:rrrmodel1_a} and \eqref{eqn:rrrmodel1_b} with 1, respectively, we have
    \begin{align}
        &\left(\frac{\frac{2\theta+1}{2}\phi_k^{n+1}-2\theta\phi_k^n+\frac{2\theta-1}{2}\phi_k^{n-1}}{\Delta t},1\right)
        +\left(q^{n+\theta}\nabla\cdot({\bf u}^{*}\phi_k^{*}),1\right)+M(\mu_k^{n+\theta},1)=0.\\
        &(\mu_k^{n+\theta},1)=\lambda(-\Delta\phi_k^{n+\theta},1)+\lambda(\bar{H}_k^{*}r^{n+\theta},1)+\lambda (\gamma^{*}r^{n+\theta},1).
    \end{align}
    Using the integration by parts and the definitions of $\bar{H}_k^{*}$ and $\gamma^{*}$, we get
    \begin{equation}
        (\bar{H}_k^{*},1)=0,\ \ (\gamma^{*},1)=0,\ \ \left(\nabla\cdot({\bf u}^{*}\phi_k^{*}),1\right)=0,\ \
        (-\Delta\phi_k^{n+\theta},1)=0.
    \end{equation}
    Combining the above equations, we arrive at the following equation
    \begin{equation*}
        \int_{\Omega}\phi_k^{n+1}d{\bf x}
        =\frac{4\theta}{2\theta+1}\int_{\Omega}\phi_k^nd{\bf x}
        -\frac{2\theta-1}{2\theta+1}\int_{\Omega}\phi_k^{n-1}d{\bf x}.
    \end{equation*}
    From Remark \ref{Remark initial steps}, we know $\int_{\Omega}\phi_k^1d{\bf x}=\int_{\Omega}\phi_k^0d{\bf x}$.
    Applying mathematical induction, we can obtain
    $$\int_{\Omega}\phi_k^{n+1}d{\bf x}=\int_{\Omega}\phi_k^nd{\bf x}=\cdots=\int_{\Omega}\phi_k^1d{\bf x}=\int_{\Omega}\phi_k^0d{\bf x},\qquad k=1,2,\dots,N.$$
    Thus, we complete the proof.
\end{proof}

\begin{theorem}
\label{rrrmodel1_energy}When $\theta \in [1/2,1]$, the discrete scheme \eqref{eqn:rrrmodel1_a}-\eqref{eqn:rrrmodel1_g} satisfies the discrete energy dissipation law in the sense that
    \begin{equation*}
        {\tilde {\mathcal{E}}}^{n+1}\leq {\tilde {\mathcal{E}}}^n,
    \end{equation*}
where ${\tilde {\mathcal{E}}}^{n+1}$ is the modified energy functional defined by
    \begin{equation*}
        {\tilde {\mathcal{E}}}^{n+1}=
        \frac{\lambda}{2}\sum_{k=1}^{N}\left\|
        \begin{bmatrix}
            \nabla\phi_k^{n+1}\\
            \nabla\phi_k^n
        \end{bmatrix}
        \right\|_{\bf G}^2
        +\lambda\left|
        \begin{bmatrix}
            r^{n+1}\\
            r^n
        \end{bmatrix}
        \right|_{\bf G}^2
        +\frac{1}{2}\left|
        \begin{bmatrix}
            q^{n+1}\\
            q^n
        \end{bmatrix}
        \right|_{\bf G}^2
        +\frac{1}{2}\left\|
        \begin{bmatrix}
            {\bf u}^{n+1}\\
            {\bf u}^n
        \end{bmatrix}
        \right\|_{\bf G}^2
        +\frac{\theta^2\Delta t^2}{2\theta+1}\left\|\nabla p^{n+1}\right\|^2.
    \end{equation*}
\end{theorem}
\begin{proof}
    By taking the $L^2$-inner product of \eqref{eqn:rrrmodel1_a} with $\Delta t\mu _k^{n+\theta}$, we obtain
    \begin{align}
        \left(
        \frac{2\theta+1}{2}\phi_k^{n+1}-2\theta\phi_k^n+\frac{2\theta-1}{2}\phi_k^{n-1},\mu_k^{n+\theta}
        \right )
       +\Delta t \left(q^{n+\theta}\nabla\cdot({\bf u}^{*}\phi_k^{*}),\mu_k^{n+\theta}\right)+M\Delta t \left \|\mu_k^{n+\theta}\right \|^2=0.
        \label{eqn:rrrmodel1_a1}
    \end{align}
    By computing the $L^2$-inner product of \eqref{eqn:rrrmodel1_b} with $\frac{2\theta+1}{2}\phi_k^{n+1}-2\theta\phi_k^n+\frac{2\theta-1}{2}\phi_k^{n-1}$ and from Lemma \eqref{lemma_GF-norm}, we get
    \begin{align}
        &\left(\mu_k^{n+\theta},\frac{2\theta+1}{2}\phi_k^{n+1}-2\theta\phi_k^n+\frac{2\theta-1}{2}\phi_k^{n-1}\right)
        \nonumber\\
        =&\lambda\left(-\Delta\phi_k^{n+\theta},\frac{2\theta+1}{2}\phi_k^{n+1}
        -2\theta\phi_k^n+\frac{2\theta-1}{2}\phi_k^{n-1}\right)
        \nonumber\\
        &+\lambda\left((\bar{H}_k^{*}+\gamma^{*})r^{n+\theta},\frac{2\theta+1}{2}\phi_k^{n+1}-2\theta\phi_k^n+\frac{2\theta-1}{2}\phi_k^{n-1}\right)
        \nonumber\\
        =&\lambda\left(\nabla\phi_k^{n+\theta},\frac{2\theta+1}{2}\nabla\phi_k^{n+1}
        -2\theta\nabla\phi_k^n
        +\frac{2\theta-1}{2}
        \nabla\phi_k^{n-1}\right)
        \nonumber\\
        &+\lambda\left((\bar{H}_k^{*}+\gamma^{*})r^{n+\theta},\frac{2\theta+1}{2}\phi_k^{n+1}-2\theta\phi_k^n+\frac{2\theta-1}{2}\phi_k^{n-1}\right)
        \nonumber\\
        =&\lambda\left(
        \frac{1}{2}
        \left\|
        \begin{bmatrix}
            \nabla\phi_k^{n+1}\\
            \nabla\phi_k^n
        \end{bmatrix}
        \right\|_{\bf G}^2
        -\frac{1}{2}
        \left\|
        \begin{bmatrix}
            \nabla\phi_k^n\\
            \nabla\phi_k^{n-1}
        \end{bmatrix}
        \right\|_{\bf G}^2
        +\frac{\theta(2\theta-1)}{4}\left\|\nabla\phi_k^{n+1}-2\nabla\phi_k^n+\nabla\phi_k^{n-1}\right\|^2\right)
        \nonumber\\  &+\lambda\left(\bar{H}_k^{*}r^{n+\theta},\frac{2\theta+1}{2}\phi_k^{n+1}-2\theta\phi_k^n+\frac{2\theta-1}{2}\phi_k^{n-1}\right)
        \nonumber\\
        &+\lambda\left(\gamma^{*}r^{n+\theta},\frac{2\theta+1}{2}\phi_k^{n+1}-2\theta\phi_k^n+\frac{2\theta-1}{2}\phi_k^{n-1}\right).
        \label{eqn:rrrmodel1_b1}
    \end{align}
    Combining the above equations and taking the summation for $k=1,2,\dots,N$, we have
    \begin{align}
        &\lambda\sum_{k=1}^{N}\left(
        \frac{1}{2}
        \left\|
        \begin{bmatrix}
            \nabla\phi_k^{n+1}\\
                \nabla\phi_k^n
        \end{bmatrix}
             \right\|_{\bf G}^2
        -\frac{1}{2}
        \left\|
        \begin{bmatrix}
            \nabla\phi_k^n\\
            \nabla\phi_k^{n-1}
        \end{bmatrix}
        \right\|_{\bf G}^2
        +\frac{\theta(2\theta-1)}{4}\left\|\nabla\phi_k^{n+1}-2\nabla\phi_k^n+\nabla\phi_k^{n-1}\right\|^2
        \right)
        \nonumber \\
        &+\lambda\sum_{k=1}^{N}\left(\bar{H}_k^{*}r^{n+\theta},\frac{2\theta+1}{2}\phi_k^{n+1}-2\theta\phi_k^n+\frac{2\theta-1}{2}\phi_k^{n-1}\right)
        \nonumber\\
        &+\Delta t \sum_{k=1}^{N}\left(q^{n+\theta}\nabla\cdot({\bf u}^{*}\phi_k^{*}),\mu_k^{n+\theta}\right )
        =-M\Delta t \sum_{k=1}^{N}\left \|\mu_k^{n+\theta}\right \|^2,
        \label{eqn:rrrmodel1_a1b1}
    \end{align}
    where the following equation is used
    \begin{align}
        &\sum_{k=1}^{N}\left(\gamma^{*}r^{n+\theta},\frac{2\theta+1}{2}\phi_k^{n+1}-2\theta\phi_k^n+\frac{2\theta-1}{2}\phi_k^{n-1}\right)
        \nonumber\\
        =&\left(\gamma^{*}r^{n+\theta},\sum_{k=1}^{N} \left(\frac{2\theta+1}{2}\phi_k^{n+1}-2\theta\phi_k^n+\frac{2\theta-1}{2}\phi_k^{n-1}\right )\right)=0.
    \end{align}
    Multiplying \eqref{eqn:rrrmodel1_c} with $2\lambda\Delta tr^{n+\theta}$ and from Remark \ref{remark_GF}, we get
    \begin{align}
        &2\lambda\left(
        \frac{1}{2}
        \left|
        \begin{bmatrix}
            r^{n+1}\\
            r^n-
        \end{bmatrix}
        \right|_{\bf G}^2
        -\frac{1}{2}
        \left|
        \begin{bmatrix}
            r^n\\
            r^{n-1}
        \end{bmatrix}
        \right|_{\bf G}^2
        +\frac{\theta(2\theta-1)}{4}\left|r^{n+1}-2r^n+r^{n-1}\right|^2\right)
        \nonumber\\
        &=\lambda\sum_{k=1}^{N}\int_\Omega \bar{H}_k^{*}r^{n+\theta}\left(\frac{2\theta+1}{2}\phi_k^{n+1}-2\theta\phi_k^n+\frac{2\theta-1}{2}\phi_k^{n-1}\right)d{\bf x}.
        \label{eqn:rrrmodel1_c1}
    \end{align}
    By taking the $L^2$-inner product of \eqref{eqn:rrrmodel1_d} with $\Delta t{\bf \tilde u}^{n+\theta}$, we have
    \begin{align}
        &\left(\frac{2\theta+1}{2}{\bf \tilde u}^{n+1}-2\theta {\bf u}^n+\frac{2\theta-1}{2}{\bf u}^{n-1},{\bf \tilde u}^{n+\theta}\right)
        +\Delta t(q^{n+\theta}{\bf u}^{*}\cdot \nabla{\bf u}^{*},{\bf \tilde u}^{n+\theta})
        \nonumber\\
        &+\Delta t\left\|\sqrt{\nu}\nabla {\bf \tilde u}^{n+\theta}\right\|^2
        +\Delta t(\nabla p^n,{\bf \tilde u}^{n+\theta})+\Delta t(q^{n+\theta}\sum_{k=1}^{N}\phi_k^{*}\nabla \mu _k^{*},{\bf \tilde u}^{n+\theta})=0.
        \label{eqn:rrrmodel1_d1}
    \end{align}
    From \eqref{eqn:rrrmodel1_e}, for any variable $\bf v$ with $\nabla \cdot {\bf v}=0$ and ${\bf v}\cdot {\bf n}\left|_{\partial \Omega}\right. =0$, we obtain
    \begin{equation}
    \label{eq:utilde_u}
        ({\bf u}^{n+1},{\bf v})=({\bf \tilde u}^{n+1},{\bf v}).
    \end{equation}
    From \eqref{eq:utilde_u} and Lemma \ref{lemma_GF-norm}, we derive the following equality
    \begin{align}
        &\left(\frac{2\theta+1}{2}{\bf \tilde u}^{n+1}-2\theta {\bf u}^n+\frac{2\theta-1}{2}{\bf u}^{n-1},{\bf \tilde u}^{n+\theta}\right)
        \nonumber\\
        =&\left(\frac{2\theta+1}{2}{\bf \tilde u}^{n+1}-2\theta {\bf u}^n+\frac{2\theta-1}{2}{\bf u}^{n-1},\theta {\bf \tilde u}^{n+1}+(1-\theta){\bf u}^n+\theta {\bf u}^{n+1}-\theta {\bf u}^{n+1}\right)
        \nonumber\\
        =&\left(\frac{2\theta+1}{2}{\bf \tilde u}^{n+1}-2\theta {\bf u}^n+\frac{2\theta-1}{2}{\bf u}^{n-1},\theta {\bf u}^{n+1}+(1-\theta){\bf u}^n+\theta ({\bf \tilde u}^{n+1}-{\bf  u}^{n+1})\right)
        \nonumber\\
        =&\left(\frac{2\theta+1}{2}{\bf \tilde u}^{n+1}-2\theta {\bf u}^n+\frac{2\theta-1}{2}{\bf u}^{n-1},\theta {\bf u}^{n+1}+(1-\theta){\bf u}^n\right)
        \nonumber\\
        &+\theta\left(\frac{2\theta+1}{2}{\bf \tilde u}^{n+1}-2\theta {\bf u}^n+\frac{2\theta-1}{2}{\bf u}^{n-1},{\bf \tilde u}^{n+1}-{\bf  u}^{n+1}\right)
        \nonumber\\
        =&\left(\frac{2\theta+1}{2}{\bf u}^{n+1}-2\theta {\bf u}^n+\frac{2\theta-1}{2}{\bf u}^{n-1},\theta {\bf u}^{n+1}+(1-\theta){\bf u}^n\right)
        \nonumber\\
        &+\frac{(2\theta+1)\theta}{2}\left({\bf \tilde u}^{n+1},{\bf \tilde u}^{n+1}-{\bf  u}^{n+1}\right)
        \nonumber\\
        =&\left(\frac{2\theta+1}{2}{\bf u}^{n+1}-2\theta {\bf u}^n+\frac{2\theta-1}{2}{\bf u}^{n-1},\theta {\bf u}^{n+1}+(1-\theta){\bf u}^n\right)
        \nonumber\\
        &+\frac{(2\theta+1)\theta}{2}\left({\bf \tilde u}^{n+1}+{\bf  u}^{n+1},{\bf \tilde u}^{n+1}-{\bf  u}^{n+1}\right)
        \nonumber\\
        =&\frac{1}{2}
        \left\|
        \begin{bmatrix}
            {\bf u}^{n+1}\\
            {\bf u}^n
        \end{bmatrix}
        \right\|_{\bf G}^2
        -\frac{1}{2}
        \left\|
        \begin{bmatrix}
            {\bf u}^n\\
            {\bf u}^{n-1}
        \end{bmatrix}
        \right\|_{\bf G}^2
        +\frac{\theta(2\theta-1)}{4}\left\|{\bf  u}^{n+1}-2{\bf u}^n+{\bf u}^{n-1}\right\|^2
        \nonumber\\
        &+\frac{(2\theta+1)\theta}{2}\left(\left\|{\bf \tilde u}^{n+1}\right\|^2-\left\|{\bf  u}^{n+1}\right\|^2\right).
        \label{eqn:rrrmodel1_d2}
    \end{align}
    By computing the $L^2$-inner product of \eqref{eqn:rrrmodel1_e} with ${\bf u}^{n+1}$ and using the identity $2(a-b)a=a^2-b^2+(a-b)^2$ and \eqref{eqn:rrrmodel1_f}, we get
    \begin{equation}
        \left\|{\bf u}^{n+1}\right\|^2-\left\|{\bf \tilde u}^{n+1}\right\|^2
        +\left\|{\bf u}^{n+1}-{\bf \tilde u}^{n+1}\right\|^2
        =0.
    \end{equation}
    Obviously, the equation \eqref{eqn:rrrmodel1_e} can be written as
    \begin{equation}
        \frac{2\theta+1}{2\Delta t}({\bf u}^{n+1}-{\bf \tilde u}^{n+1})=-\theta(\nabla p^{n+1}-\nabla p^{n}).
    \end{equation}
    Taking the square of both sides of the above equation and integrating over $\Omega$, we obtain
    \begin{equation}
        \left\|{\bf u}^{n+1}-{\bf \tilde u}^{n+1}\right\|^2
        =\frac{4\Delta t^2\theta^2}{(2\theta+1)^2}\left\|\nabla p^{n+1}-\nabla p^n\right\|^2.
    \end{equation}
    Thus, we have
    \begin{equation}
        \left\|{\bf u}^{n+1}\right\|^2-\left\|{\bf \tilde u}^{n+1}\right\|^2
        =-\left\|{\bf u}^{n+1}-{\bf \tilde u}^{n+1}\right\|^2
        =-\frac{4\Delta t^2\theta^2}{(2\theta+1)^2}\left\|\nabla p^{n+1}-\nabla p^n\right\|^2.
        \label{eqn:rrrmodel1_e1}
    \end{equation}
    We reformulate the projection step \eqref{eqn:rrrmodel1_e} as
    \begin{equation}
        \frac{\frac{2\theta+1}{2\theta}\left(\theta{\bf u}^{n+1}+(1-\theta){\bf u}^n-{\bf \tilde u}^{n+\theta}\right)}{\Delta t}+\theta(\nabla p^{n+1}-\nabla p^n)=0,
        \label{eqn:rrrmodel1_e2}
    \end{equation}
    where we apply the equality derived below
    \begin{align}
        &\frac{2\theta+1}{2}{\bf  u}^{n+1}-\frac{2\theta+1}{2}{\bf \tilde u}^{n+1}
        =\frac{2\theta+1}{2\theta}({\theta\bf u}^{n+1}-{\theta\bf \tilde u}^{n+1})
        \nonumber\\
        & =\frac{2\theta+1}{2\theta}\left(\theta{\bf u}^{n+1}+(1-\theta){\bf u}^n-\theta{\bf \tilde u}^{n+1}-(1-\theta){\bf u}^n\right)
        \nonumber\\
        & =\frac{2\theta+1}{2\theta}\left(\theta{\bf u}^{n+1}+(1-\theta){\bf u}^n-{\bf \tilde u}^{n+\theta}\right).
    \end{align}
    By taking the $L^2$-inner product of \eqref{eqn:rrrmodel1_e2} with $\nabla p^n$ and using the identity $2(a-b)b=a^2-b^2-(a-b)^2$, we obtain
    \begin{equation}
        \Delta t(\nabla p^n,{\bf \tilde u}^{n+\theta})=\frac{\theta^2\Delta t^2}{2\theta+1}\left(\left\|\nabla p^{n+1}\right\|^2-\left\|\nabla p^n\right\|^2-\left\|\nabla p^{n+1}-\nabla p^n\right\|^2\right).
        \label{eqn:rrrmodel1_e3}
    \end{equation}
    Combining \eqref{eqn:rrrmodel1_d1}, \eqref{eqn:rrrmodel1_d2}, \eqref{eqn:rrrmodel1_e1} and  \eqref{eqn:rrrmodel1_e3} gets
    \begin{align}
        &\frac{1}{2}
        \left\|
        \begin{bmatrix}
            {\bf u}^{n+1}\\
            {\bf u}^n
        \end{bmatrix}
        \right\|_{\bf G}^2
        -\frac{1}{2}
        \left\|
        \begin{bmatrix}
            {\bf u}^n\\
            {\bf u}^{n-1}
        \end{bmatrix}
        \right\|_{\bf G}^2
        +\frac{\theta(2\theta-1)}{4}\left\|{\bf  u}^{n+1}-2{\bf u}^n+{\bf u}^{n-1}\right\|^2
        \nonumber\\
        &+\Delta tq^{n+\theta}({\bf u}^{*}\cdot \nabla{\bf u}^{*},{\bf \tilde u}^{n+\theta})
        +\frac{\theta^2\Delta t^2}{2\theta+1}\left(\left\|\nabla p^{n+1}\right\|^2-\left\|\nabla p^n\right\|^2\right)
        \nonumber\\
        &+\Delta t \left(q^{n+\theta}\sum_{k=1}^{N}\phi_k^{*}\nabla \mu _k^{*},{\bf \tilde u}^{n+\theta}\right)
        \nonumber\\
        &=-\frac{\theta^2\Delta t^2}{2\theta+1}(2\theta-1)\left\|\nabla p^{n+1}-\nabla p^n\right\|^2-\Delta t\left\|\sqrt{\nu}\nabla {\bf \tilde u}^{n+\theta}\right\|^2.
        \label{eqn:rrrmodel1_d1e1}
    \end{align}
    By multiplying \eqref{eqn:rrrmodel1_g} with $\Delta tq^{n+\theta}$ and from Remark \ref{remark_GF}, we have
    \begin{align}
        &\frac{1}{2}
        \left |
        \begin{bmatrix}
            q^{n+1}\\
            q^n
        \end{bmatrix}
        \right |_{\bf G}^2
        -\frac{1}{2}
        \left |
        \begin{bmatrix}
            q^n\\
            q^{n-1}
        \end{bmatrix}
        \right |_{\bf G}^2
        +\frac{\theta(2\theta-1)}{4}\left|q^{n+1}-2q^n+q^{n-1}\right|^2
        \nonumber\\
        & =\Delta t \sum_{k=1}^{N}\int_\Omega q^{n+\theta}\nabla \cdot ({\bf u}^{*}\phi_k^{*})\mu_k^{n+\theta}d{\bf x}+\Delta t \sum_{k=1}^{N}\int_\Omega q^{n+\theta}\phi _k^{*} \nabla \mu_k^{*}\cdot {\bf \tilde u}^{n+\theta}d{\bf x}
        \nonumber\\
        &+\Delta t\int_\Omega q^{n+\theta}{\bf u}^{*}\cdot
        \nabla {\bf u }^{*}\cdot {\bf \tilde u}^{n+\theta}d{\bf x}.
        \label{eqn:rrrmodel1_g1}
    \end{align}
    Combining \eqref{eqn:rrrmodel1_a1b1}, \eqref{eqn:rrrmodel1_c1}, \eqref{eqn:rrrmodel1_d1e1} and \eqref{eqn:rrrmodel1_g1}, we obtain
    \begin{align}
        &\frac{\lambda}{2}
        \sum_{k=1}^{N} \left(
        \left\|
        \begin{bmatrix}
            \nabla \phi_k^{n+1}\\
            \nabla \phi_k^n
        \end{bmatrix}
        \right\|_{\bf G}^2
        -\left\|
        \begin{bmatrix}
            \nabla \phi_k^n\\
            \nabla \phi_k^{n-1}
        \end{bmatrix}
        \right\|_{\bf G}^2
        \right)
        +\lambda\left(
        \left|
        \begin{bmatrix}
            r^{n+1}\\
            r^n
        \end{bmatrix}
        \right|_{\bf G}^2
        -\left|
        \begin{bmatrix}
            r^n\\
            r^{n-1}
        \end{bmatrix}
        \right|_{\bf G}^2
        \right)
        +\frac{1}{2}\left(
        \left|
        \begin{bmatrix}
            q^{n+1}\\
            q^n
        \end{bmatrix}
        \right|_{\bf G}^2
        -\left|
        \begin{bmatrix}
            q^n\\
            q^{n-1}
        \end{bmatrix}
        \right|_{\bf G}^2
        \right)
        \nonumber\\
        &+\frac{1}{2}
        \left(
        \left\|
        \begin{bmatrix}
            {\bf u}^{n+1}\\
            {\bf u}^n
        \end{bmatrix}
        \right\|_{\bf G}^2
        -\left\|
        \begin{bmatrix}
            {\bf u}^n\\
            {\bf u}^{n-1}
        \end{bmatrix}
        \right\|_{\bf G}^2
        \right)
        +\frac{\theta^2\Delta t^2}{2\theta+1}\left(\left\|\nabla p^{n+1}\right\|^2-\left\|\nabla p^n\right\|^2\right)
        \nonumber
        \nonumber\\
        &=-M\Delta t \sum_{k=1}^{N}\left \|\mu_k^{n+\theta}\right \|^2
        -\frac{\lambda\theta(2\theta-1)}{4}
        \sum_{k=1}^{N}\left\|\nabla\phi_k^{n+1}-2\nabla\phi_k^n+\nabla\phi_k^{n-1}\right\|^2
        -\frac{\lambda\theta(2\theta-1)}{2}\left|r^{n+1}-2r^n+r^{n-1}\right|^2
        \nonumber\\
        &-\frac{\theta(2\theta-1)}{4}\left\|{\bf  u}^{n+1}-2{\bf u}^n+{\bf u}^{n-1}\right\|^2
        -\frac{\theta^2\Delta t^2}{2\theta+1}(2\theta-1)\left\|\nabla p^{n+1}-\nabla p^n\right\|^2
        \nonumber\\
        &-\Delta t\left\|\sqrt{\nu}\nabla {\bf \tilde u}^{n+\theta}\right\|^2
        -\frac{\theta(2\theta-1)}{4}\left|q^{n+1}-2q^n+q^{n-1}\right|^2\leq 0.
        \nonumber
    \end{align}
We have completed the proof for the energy dissipation law of \eqref{eqn:rrrmodel1}.
\end{proof}

\subsection{N-component D-CAC model}
\label{sec-3.2}
We still introduce two scalar auxiliary variables
\begin{equation}
    r(t)=\sqrt{\int_\Omega \sum_{k=1}^{N}F(\phi_k)d{\bf x}+C},\quad\quad q(t)\equiv 1,
\end{equation}
then the system \eqref{eqn:rmodel2} can be written as
\begin{subequations}
\label{eqn:rrmodel2}
\begin{align}
&\frac{\partial\phi_k}{\partial t}+q\nabla \cdot ({\bf{u}}\phi_k)+M\mu_k=0, \label{eqn:rrmodel2_a}\\
&\mu_k=\lambda\left(-\Delta \phi_k + (\bar{H}_k+\gamma)r \right), \label{eqn:rrmodel2_b}\\
&\frac{dr}{dt}=\frac{1}{2}\sum_{k=1}^{N}\int_\Omega {\bar H}_k\frac{\partial \phi_k}{\partial t}d{\bf x},
\label{eqn:rrmodel2_c}\\
&\tau\frac{\partial \bf u}{\partial t} +\alpha \nu {\bf u} +\nabla p+q\sum_{k=1}^{N} \phi _k \nabla \mu_k=0,
\label{eqn:rrmodel2_d}\\
&\nabla \cdot \bf u=0,
\label{eqn:rrmodel2_e}\\
&\frac{dq}{dt}=\sum_{k=1}^{N}\int_\Omega \nabla \cdot ({\bf u}\phi_k)\mu_kd{\bf x}+\sum_{k=1}^{N}\int_\Omega\phi _k \nabla \mu_k\cdot {\bf u}d{\bf x},
\label{eqn:rrmodel2_f}
\end{align}
\end{subequations}
where
\begin{equation}
    H_k=\frac{f_k(\phi _k)}{\sqrt{\int_\Omega \sum_{k=1}^{N}F(\phi_k)d{\bf x}+C}},\ \
    \bar{H}_k=H_k-\frac{1}{|\Omega|}\int_\Omega H_kd{\bf x},\ \
    \gamma = -\frac{1}{N}\sum_{k=1}^{N}\bar{H}_k.
    \label{eqn:rrmodel2_g}
\end{equation}
\par
Similar to the approaches of the N-component NS-CAC model, we could give the mass conservation and energy stability properties of the system \eqref{eqn:rrmodel2}.
\begin{theorem}
    The solutions of the system \eqref{eqn:rrmodel2} satisfy the mass conservation of each phase.
\end{theorem}
\begin{proof}
    We can easily obtain $\frac{d}{dt} \int_\Omega \phi_k d{\bf x}=0$, which is similar to Theorem \ref{rrmodel1_mass}.
\end{proof}
\par
\begin{theorem}
    The system \eqref{eqn:rrmodel2} satisfies the energy dissipation law.
\end{theorem}
\begin{proof}
    By taking the $L^2$-inner  product of \eqref{eqn:rrmodel2_d} with ${\bf u}$, it follows that
    \begin{equation}
    \label{eqn:rrmodel2_d1}
          \frac{d}{dt} \left(\frac{\tau}{2}\left\|{\bf u}\right \|^2\right)
          +\alpha\left\|\sqrt{\nu}{\bf u}\right \|^2+q\sum_{k=1}^{N} (\phi _k \nabla \mu_k,{\bf u})=0.
    \end{equation}
    Multiplying \eqref{eqn:rrmodel2_f} with $q$ obtains
    \begin{equation}
    \label{eqn:rrmodel2_f1}
           \frac{d}{dt}\left (\frac{|q|^2}{2}\right )=q\sum_{k=1}^{N}\int_\Omega \nabla \cdot ({\bf u}\phi_k)\mu_kd{\bf x}+q\sum_{k=1}^{N}\int_\Omega\phi _k \nabla \mu_k\cdot {\bf u}d{\bf x}.
    \end{equation}
Following same technique as Theorem \ref{rrmodel1_energy},
we can also get \eqref{eqn:rrmodel1_a1b1} and \eqref{eqn:rrmodel1_c1}. Combining them with \eqref{eqn:rrmodel2_d1}-\eqref{eqn:rrmodel2_f1}, we obtain the energy dissipative law as
    \begin{align*}
        \frac{d}{dt}E({\bf u,\phi})
        &=\frac{d}{dt}\left(\frac{\lambda}{2}\sum_{k=1}^{N}\left \|\nabla \phi_k\right \|^2\right)
        +\frac{d}{dt}\left(\frac{\tau}{2}\left\|{\bf u}\right\|^2\right)+ \frac{d}{dt}\left(\frac{|q|^2}{2}\right)+\lambda\frac{d}{dt}(|r|^2)\\
        &=-M\sum_{k=1}^{N}\left \|\mu_k\right \|^2-\alpha\left \|\sqrt{\nu}{\bf u}\right \|^2 \leq 0.
    \end{align*}
    Therefore, we finish the proof.
\end{proof}
\par
Similar to the scheme \eqref{eqn:rrrmodel1_a}-\eqref{eqn:rrrmodel1_g} for the N-component NS-CAC model, we also develop the following scheme to solve the system \eqref{eqn:rrmodel2}:
\begin{subequations}
 \label{eqn:rrrmodel2}
 \begin{align}
    &\frac{\frac{2\theta+1}{2}\phi_k^{n+1}-2\theta\phi_k^n+\frac{2\theta-1}{2}\phi_k^{n-1}}{\Delta t}+q^{n+\theta}\nabla\cdot({\bf u}^{*}\phi_k^{*})+M\mu_k^{n+\theta}=0,\label{eqn:rrrmodel2_a}\\
    &\mu_k^{n+\theta}=\lambda\left(-\Delta\phi_k^{n+\theta}+(\bar{H}_k^{*}+\gamma^{*})r^{n+\theta}\right),\label{eqn:rrrmodel2_b}\\
    &\frac{\frac{2\theta+1}{2}r^{n+1}-2\theta r^n+\frac{2\theta-1}{2}r^{n-1}}{\Delta t}=\frac{1}{2}\sum_{k=1}^{N}\int_\Omega \bar{H}_k^{*}\frac{\frac{2\theta+1}{2}\phi_k^{n+1}-2\theta\phi_k^n+\frac{2\theta-1}{2}\phi_k^{n-1}}{\Delta t}d{\bf x},
    \label{eqn:rrrmodel2_c}\\
    &\tau\frac{\frac{2\theta+1}{2}{\bf \tilde u}^{n+1}-2\theta {\bf u}^n+\frac{2\theta-1}{2}{\bf u}^{n-1}}{\Delta t}+\alpha\nu{\bf \tilde u}^{n+\theta}+\nabla p^n+q^{n+\theta}\sum_{k=1}^{N}\phi_k^{*}\nabla \mu _k^{*}=0,\label{eqn:rrrmodel2_d}\\
    &\tau\frac{\frac{2\theta+1}{2}{\bf u}^{n+1}-\frac{2\theta+1}{2}{\bf \tilde u}^{n+1}}{\Delta t}+(\nabla p^{n+\theta}-\nabla p^{n})=0,\label{eqn:rrrmodel2_e}\\
    &\nabla \cdot {\bf u}^{n+1}=0,
    \label{eqn:rrrmodel2_f}\\
    &\frac{\frac{2\theta+1}{2}q^{n+1}-2\theta q^n+\frac{2\theta-1}{2}q^{n-1}}{\Delta t}
    =\sum_{k=1}^{N}\int_\Omega \nabla \cdot ({\bf u}^{*}\phi_k^{*})\mu_k^{n+\theta}d{\bf x}+\sum_{k=1}^{N}\int_\Omega\phi _k^{*} \nabla \mu_k^{*}\cdot {\bf \tilde u}^{n+\theta}d{\bf x},
    \label{eqn:rrrmodel2_g}
 \end{align}
\end{subequations}
where the boundary conditions are  periodic or as follows
\begin{equation}
    \left.{\bf u}^{n+1}\cdot {\bf n}\right |_{\partial \Omega}=0,
    \left.{\bf \tilde u}^{n+1}\right |_{\partial \Omega}={\bf 0},
    \left.\nabla\phi_k^{n+1}\cdot {\bf n}\right |_{\partial \Omega}=\left.\nabla\mu_k^{n+1}\cdot {\bf n}\right |_{\partial \Omega}=\left.\nabla p^{n+1}\cdot {\bf n}\right |_{\partial \Omega}=0.
\end{equation}
A detailed implementation of the scheme \eqref{eqn:rrrmodel2_a}-\eqref{eqn:rrrmodel2_g} is shown  in Appendix \ref{A.2}.
Moreover, for the scheme \eqref{eqn:rrrmodel2}, we have the following mass conservation and energy stability results.
\begin{theorem}
    The discrete scheme \eqref{eqn:rrrmodel2_a}-\eqref{eqn:rrrmodel2_g} conserves the mass of the each phase.
\end{theorem}
\begin{proof}
  We can readily obtain
    $\int_\Omega \phi_k^{n+1}d{\bf x}=\int_\Omega \phi_k^nd{\bf x}=\cdots=\int_\Omega \phi_k^0d{\bf x}$ by the same routine of Theorem \ref{rrrmodel1_mass}  for $k=1,2,\dots,N$.
\end{proof}

\begin{theorem}
   The discrete scheme \eqref{eqn:rrrmodel2_a}-\eqref{eqn:rrrmodel2_g} satisfies the discrete energy dissipation law in the sense that
    \begin{equation}
        {\tilde E}^{n+1}\leq {\tilde E}^n,
    \end{equation}
where ${\tilde E}^{n+1}$ is the modified energy functional defined by
    \begin{equation*}
        {\tilde E}^{n+1}=
        \frac{\lambda}{2}\sum_{k=1}^{N}\left\|
        \begin{bmatrix}
            \nabla\phi_k^{n+1}\\
            \nabla\phi_k^n
        \end{bmatrix}
        \right\|_{\bf G}^2
        +\lambda\left|
        \begin{bmatrix}
            r^{n+1}\\
            r^n
        \end{bmatrix}
        \right|_{\bf G}^2
        +\frac{1}{2}\left|
        \begin{bmatrix}
            q^{n+1}\\
            q^n
        \end{bmatrix}
        \right|_{\bf G}^2
        +\frac{\tau}{2}\left\|
        \begin{bmatrix}
            {\bf u}^{n+1}\\
            {\bf u}^n
        \end{bmatrix}
        \right\|_{\bf G}^2
        +\frac{\theta^2\Delta t^2}{\tau(2\theta+1)}\left\|\nabla p^{n+1}\right\|^2.
    \end{equation*}
\end{theorem}
\begin{proof}
    By taking the $L^2$-inner products of \eqref{eqn:rrrmodel2_a} and \eqref{eqn:rrrmodel2_b} with $\Delta t\mu _k^{n+\theta}$ and $\frac{2\theta+1}{2}\phi_k^{n+1}-2\theta\phi_k^n+\frac{2\theta-1}{2}\phi_k^{n-1}$, respectively, we have
    \begin{align}
    \label{eqn:rrrmodel2_a1}
        \left(
        \frac{2\theta+1}{2}\phi_k^{n+1}-2\theta\phi_k^n+\frac{2\theta-1}{2}\phi_k^{n-1},\mu_k^{n+\theta}
        \right )
       +\Delta t \left(q^{n+\theta}\nabla\cdot({\bf u}^{*}\phi_k^{*}),\mu_k^{n+\theta}\right)+M\Delta t \left \|\mu_k^{n+\theta}\right \|^2=0,
    \end{align}
    and
    \begin{align}
    \label{eqn:rrrmodel2_b1}
        &\left(\mu_k^{n+\theta},\frac{2\theta+1}{2}\phi_k^{n+1}-2\theta\phi_k^n+\frac{2\theta-1}{2}\phi_k^{n-1}\right)
        \nonumber\\
        =&\lambda\left(
        \frac{1}{2}
        \left\|
        \begin{bmatrix}
            \nabla\phi_k^{n+1}\\
            \nabla\phi_k^n
        \end{bmatrix}
        \right\|_{\bf G}^2
        -\frac{1}{2}
        \left\|
        \begin{bmatrix}
            \nabla\phi_k^n\\
            \nabla\phi_k^{n-1}
        \end{bmatrix}
        \right\|_{\bf G}^2
        +\frac{\theta(2\theta-1)}{4}\left\|\nabla\phi_k^{n+1}-2\nabla\phi_k^n+\nabla\phi_k^{n-1}\right\|^2\right)
        \nonumber\\  &+\lambda\left(\bar{H}_k^{*}r^{n+\theta},\frac{2\theta+1}{2}\phi_k^{n+1}-2\theta\phi_k^n+\frac{2\theta-1}{2}\phi_k^{n-1}\right)
        \nonumber\\
        &+\lambda\left(\gamma^{*}r^{n+\theta},\frac{2\theta+1}{2}\phi_k^{n+1}-2\theta\phi_k^n+\frac{2\theta-1}{2}\phi_k^{n-1}\right).
    \end{align}
    One can  multiply \eqref{eqn:rrrmodel2_c} with $2\lambda\Delta t r^{n+\theta}$ to arrive at
    \begin{align}
        &2\lambda\left(
        \frac{1}{2}
        \left|
        \begin{bmatrix}
            r^{n+1}\\
            r^n
        \end{bmatrix}
        \right|_{\bf G}^2
        -\frac{1}{2}
        \left|
        \begin{bmatrix}
            r^n\\
            r^{n-1}
        \end{bmatrix}
        \right|_{\bf G}^2
        +\frac{\theta(2\theta-1)}{4}\left|r^{n+1}-2r^n+r^{n-1}\right|^2
        \right)
        \nonumber\\
        &=\lambda\sum_{k=1}^{N}\int_\Omega \bar{H}_k^{*}r^{n+\theta}\left(\frac{2\theta+1}{2}\phi_k^{n+1}-2\theta\phi_k^n+\frac{2\theta-1}{2}\phi_k^{n-1}\right)d{\bf x}.
        \label{eqn:rrrmodel2_c1}
    \end{align}
    By computing the $L^2$-inner product of \eqref{eqn:rrrmodel2_d} with $\Delta t{\bf\tilde u}^{n+\theta}$, we obtain
    \begin{align}
        &\tau\left(\frac{2\theta+1}{2}{\bf \tilde u}^{n+1}-2\theta {\bf u}^n+\frac{2\theta-1}{2}{\bf u}^{n-1},{\bf \tilde u}^{n+\theta}\right)
        +\alpha \Delta t\left\|\sqrt{\nu}{\bf \tilde u}^{n+\theta}\right\|^2 \nonumber\\
        &+\Delta t(\nabla p^n,{\bf \tilde u}^{n+\theta})+\Delta t(q^{n+\theta}\sum_{k=1}^{N}\phi_k^{*}\nabla \mu _k^{*},{\bf \tilde u}^{n+\theta})=0.
        \label{eqn:rrrmodel2_d1}
    \end{align}
    From \eqref{eqn:rrrmodel2_e}-\eqref{eqn:rrrmodel2_f}, we can also derive \eqref{eqn:rrrmodel1_d2} and the following equations by the similar process in Theorem \ref{rrrmodel1_energy}
    \begin{align}
        &\left\|{\bf u}^{n+1}\right\|^2-\left\|{\bf \tilde u}^{n+1}\right\|^2
        =-\left\|{\bf u}^{n+1}-{\bf \tilde u}^{n+1}\right\|^2
        =-\frac{4\Delta t^2\theta^2}{\tau^2(2\theta+1)^2}\left\|\nabla p^{n+1}-\nabla p^n\right\|^2,
        \label{eqn:rrrmodel2_e1}\\
        &\Delta t(\nabla p^n,{\bf \tilde u}^{n+\theta})=\frac{\theta^2\Delta t^2}{\tau(2\theta+1)}\left(\left\|\nabla p^{n+1}\right\|^2-\left\|\nabla p^n\right\|^2-\left\|\nabla p^{n+1}-\nabla p^n\right\|^2\right).
        \label{eqn:rrrmodel2_e2}
    \end{align}
    Multiplying \eqref{eqn:rrrmodel2_g} with $\Delta t q^{n+\theta}$, we have
    \begin{align}
        &\frac{1}{2}
        \left|
        \begin{bmatrix}
            q^{n+1}\\
            q^n
        \end{bmatrix}
        \right|_{\bf G}^2
        -\frac{1}{2}
        \left|
        \begin{bmatrix}
            q^n\\
            q^{n-1}
        \end{bmatrix}
        \right|_{\bf G}^2
        +\frac{\theta(2\theta-1)}{4}\left|q^{n+1}-2q^n+q^{n-1}\right|^2
        \nonumber\\
        =&\Delta t
        \sum_{k=1}^{N}\int_\Omega q^{n+\theta}\nabla \cdot ({\bf u}^{*}\phi_k^{*})\mu_k^{n+\theta}d{\bf x}
        +\Delta t
        \sum_{k=1}^{N}\int_\Omega q^{n+\theta}\phi _k^{*} \nabla \mu_k^{*}\cdot {\bf \tilde u}^{n+\theta}d{\bf x}.
        \label{eqn:rrrmodel2_g1}
    \end{align}
    By combining \eqref{eqn:rrrmodel2_a1}-\eqref{eqn:rrrmodel2_g1} with \eqref{eqn:rrrmodel1_d2}, we deduce
    \begin{align}
        &\frac{\lambda}{2}
        \sum_{k=1}^{N}\left(
        \left\|
        \begin{bmatrix}
            \nabla\phi_k^{n+1}\\
            \nabla\phi_k^n
        \end{bmatrix}
        \right\|_{\bf G}^2
        -\left\|
        \begin{bmatrix}
            \nabla\phi_k^n\\
            \nabla\phi_k^{n-1}
        \end{bmatrix}
        \right\|_{\bf G}^2
        \right)
        +\lambda\left(
        \left|
        \begin{bmatrix}
            r^{n+1}\\
            r^n
        \end{bmatrix}
        \right|_{\bf G}^2
        -\left|
        \begin{bmatrix}
            r^n\\
            r^{n-1}
        \end{bmatrix}
        \right|_{\bf G}^2
        \right)
        +\frac{1}{2}
        \left(
        \left|
        \begin{bmatrix}
            q^{n+1}\\
            q^n
        \end{bmatrix}
        \right|_{\bf G}^2
        -\left|
        \begin{bmatrix}
            q^n\\
            q^{n-1}
        \end{bmatrix}
        \right|_{\bf G}^2
        \right)
        \nonumber\\
        &+\frac{\tau}{2}
        \left(
        \left\|
        \begin{bmatrix}
            {\bf u}^{n+1}\\
            {\bf u}^n
        \end{bmatrix}
        \right\|_{\bf G}^2
        -\left\|
        \begin{bmatrix}
            {\bf u}^n\\
            {\bf u}^{n-1}
        \end{bmatrix}
        \right\|_{\bf G}^2
        \right)
        +\frac{\theta^2\Delta t^2}{\tau(2\theta+1)}\left(\left\|\nabla p^{n+1}\right\|^2-\left\|\nabla p^n\right\|^2\right)
        \nonumber\\
        &=-M\Delta t \sum_{k=1}^{N}\left \|\mu_k^{n+\theta}\right \|^2
        -\frac{\lambda\theta(2\theta-1)}{4}\sum_{k=1}^{N}\left\|\nabla\phi_k^{n+1}-2\nabla\phi_k^n+\nabla\phi_k^{n-1}\right\|^2
        -\frac{\lambda\theta(2\theta-1)}{2}\left|r^{n+1}-2r^n+r^{n-1}\right|^2
        \nonumber\\
        &-\frac{\tau\theta(2\theta-1)}{4}\left\|{\bf  u}^{n+1}-2{\bf u}^n+{\bf u}^{n-1}\right\|^2
        -\frac{\theta^2\Delta t^2}{\tau(2\theta+1)}(2\theta-1)\left\|\nabla p^{n+1}-\nabla p^n\right\|^2
        \nonumber\\
        &-\alpha\Delta t
        \left\|\sqrt{\nu}{\bf \tilde u}^{n+\theta}\right\|^2
        -\frac{\theta(2\theta-1)}{4}\left|q^{n+1}-2q^n+q^{n-1}\right|^2\leq 0.
        \nonumber
    \end{align}
    Hence, we conclude the proof.
\end{proof}

\section{Numerical experiments}
\label{sec-4}
In this section, we implement the proposed numerical schemes ($\theta$-SAV) for simulating the 2-component and 3-component NS-CAC/D-CAC models. Numerical tests demonstrate the high performance of the schemes, including the accuracy, energy dissipation and mass conservation.
In addition, we exhibit the temporal evolutions of the phase separation for the proposed numerical schemes. Let ${\bf u}=(u,v)$, and the Fourier spectral method \cite{chen2018enriched,Mao2017SIAM} is used in space.\par

\subsection{Convergence tests} \label{sec-4.1}

In this subsection, we present some numerical tests to check the temporal accuracy of the $\theta$-SAV schemes for the 2-component and 3-component NS-CAC/D-CAC models.
For the two models, we set the parameters as $\alpha = 1000$, $\tau=1$, $\nu = 1$, $\lambda = 0.01$, $\epsilon = 0.05$, $M = 10$, $C = 10$.
Then, we perform the numerical simulations using different time steps, i.e.
$\Delta t$ = 1e-3, 5e-4, 2.5e-4, 1.25e-4, 6.25e-5, 3.125e-5, and the spatial mesh size $128\times 128$.
\par
To test the convergence and accuracy of the schemes, we first consider the two models with external forcing terms. For 2-component models, the exact solution is set to be
\begin{align}
    \left \{
    \begin{aligned}
        &\phi=0.5+0.5\cos(t)\sin(\pi x)\sin(\pi y), \\
        &u = \pi \sin(t)\sin(2\pi y)\sin^2(\pi x),\\
        &v = -\pi \sin(t)\sin(2\pi x)\sin^2(\pi y),\\
        &p = \sin(t)\cos(\pi x)\sin(\pi y),
    \end{aligned}
    \right.
    \nonumber
\end{align}
and for 3-component models, the exact solution is set to be
\begin{align}
    \left \{
    \begin{aligned}
        &\phi_1=0.3+0.01\cos(t)\sin(\pi x)\sin(\pi y), \\
        &\phi_2=0.3+0.02\cos(t)\sin(\pi x)\sin(\pi y), \\
        &\phi_3=1-\phi_1-\phi_2, \\
        &u = \pi \sin(t)\sin(2\pi y)\sin^2(\pi x),\\
        &v = -\pi \sin(t)\sin(2\pi x)\sin^2(\pi y),\\
        &p = \sin(t)\cos(\pi x)\sin(\pi y).
    \end{aligned}
    \right.
    \nonumber
\end{align}
In Figs. \ref{figure:error_1_NS_AC}- \ref{figure:error_3_D_CAC}, we show the $L^{\infty}$ error and the convergence rates at $t=0.1$ for the $\theta$-SAV schemes with $\theta =0.5$, $0.75$, $1$.
The results indicate that the temporal accuracy of the $\theta$-SAV scheme is always second-order for arbitrary $\theta\in [1/2,1]$.\par

\begin{figure}[!htp]
\centering
\includegraphics[width=0.32\textwidth]{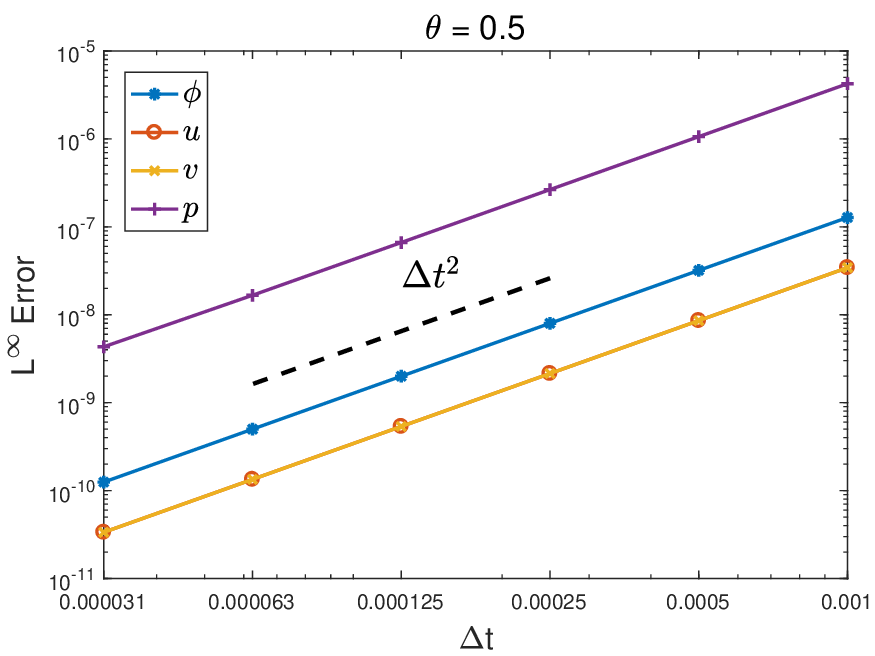}
\includegraphics[width=0.32\textwidth]{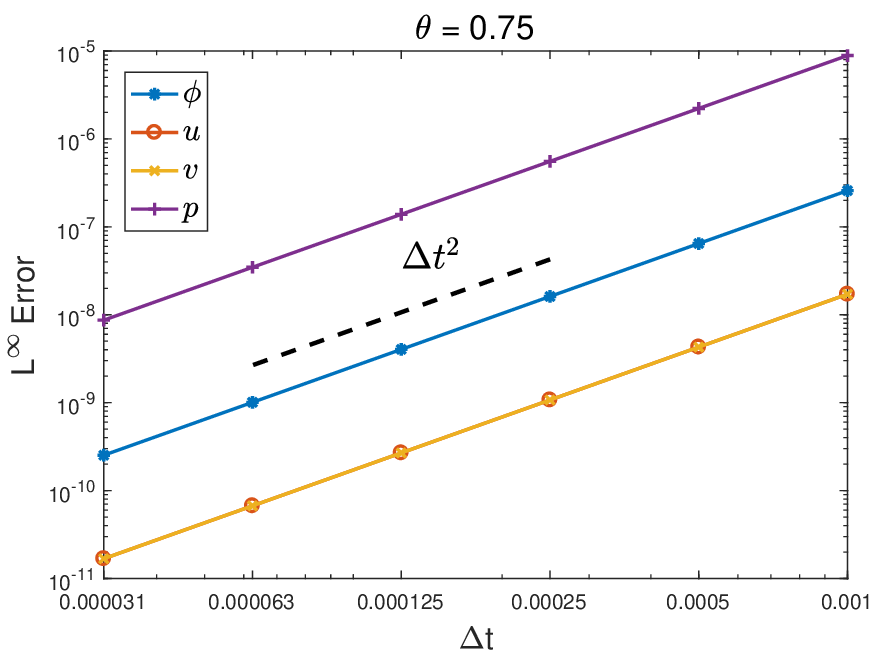}
\includegraphics[width=0.32\textwidth]{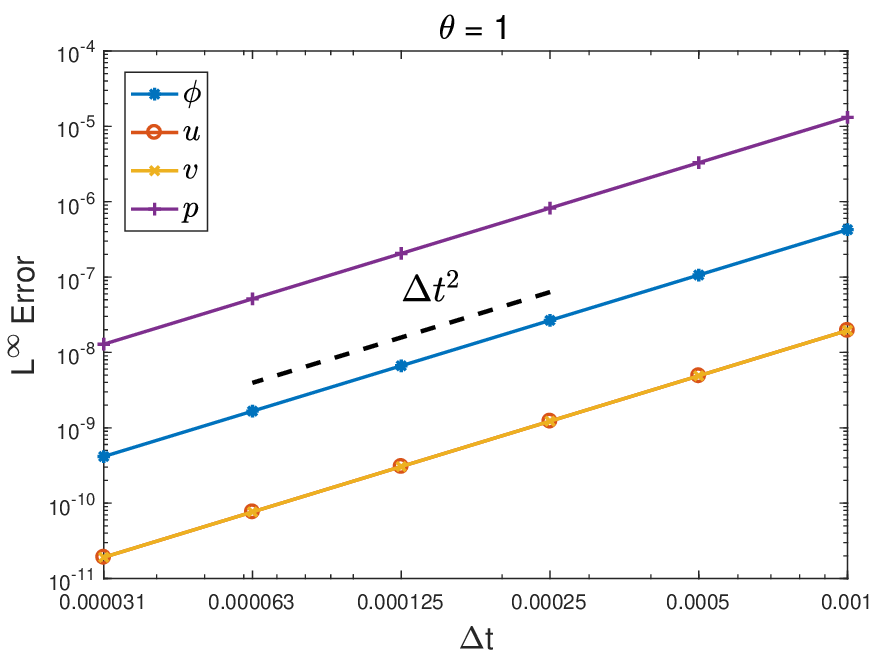}
\caption{Convergence tests of the 2-component NS-CAC model with external forcing terms}
\label{figure:error_1_NS_AC}
\end{figure}

\begin{figure}[!htp]
\centering
\includegraphics[width=0.32\textwidth]{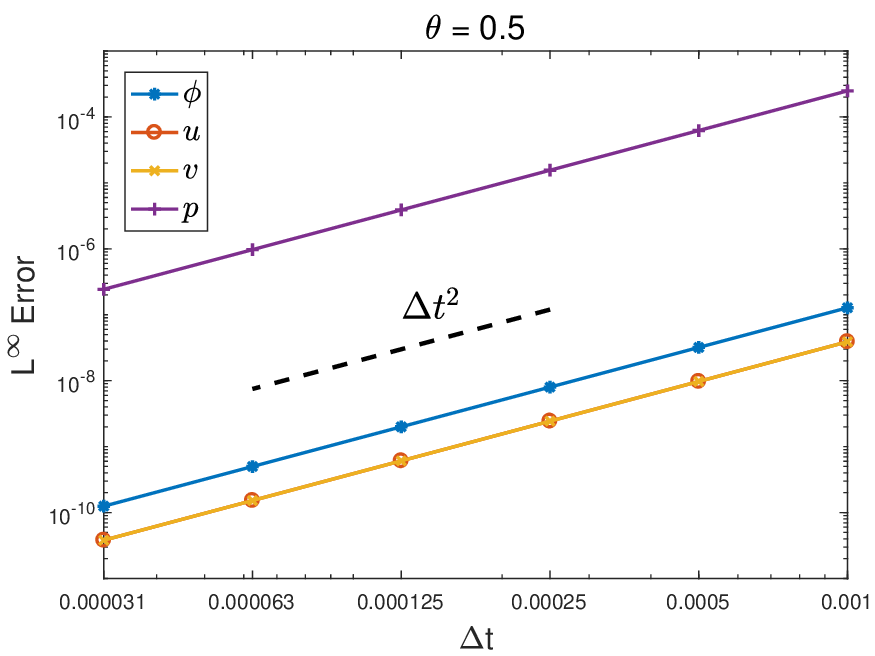}
\includegraphics[width=0.32\textwidth]{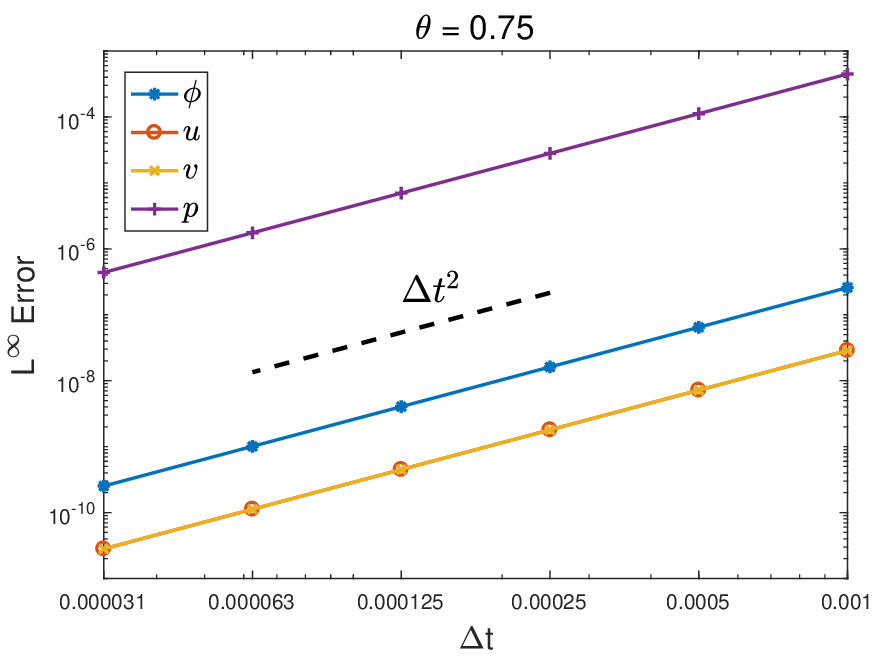}
\includegraphics[width=0.32\textwidth]{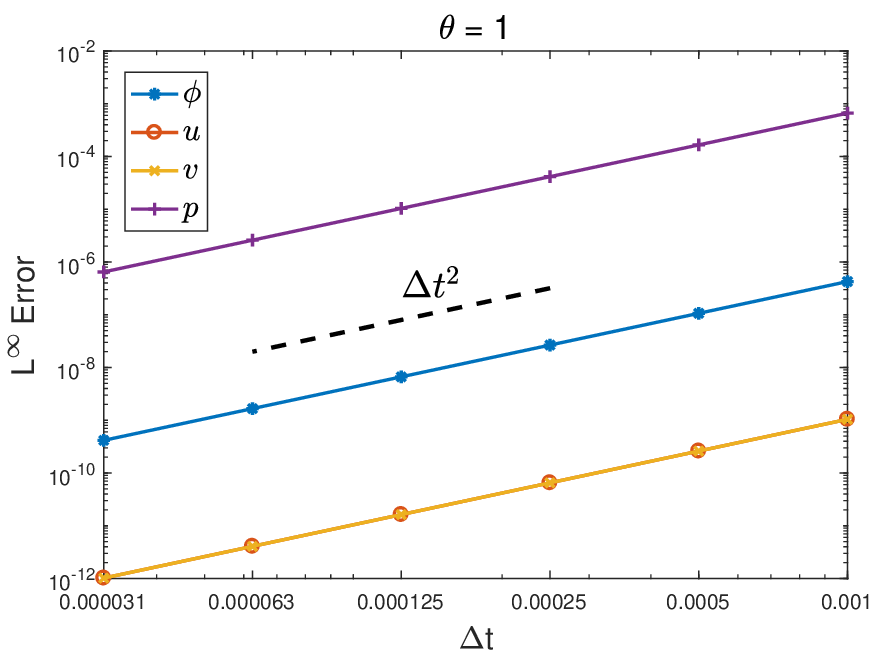}
\caption{Convergence tests of the 2-component D-CAC model with external forcing terms}
\label{figure:error_1_D_AC}
\end{figure}

\begin{figure}[!htp]
\centering
\includegraphics[width=0.32\textwidth]{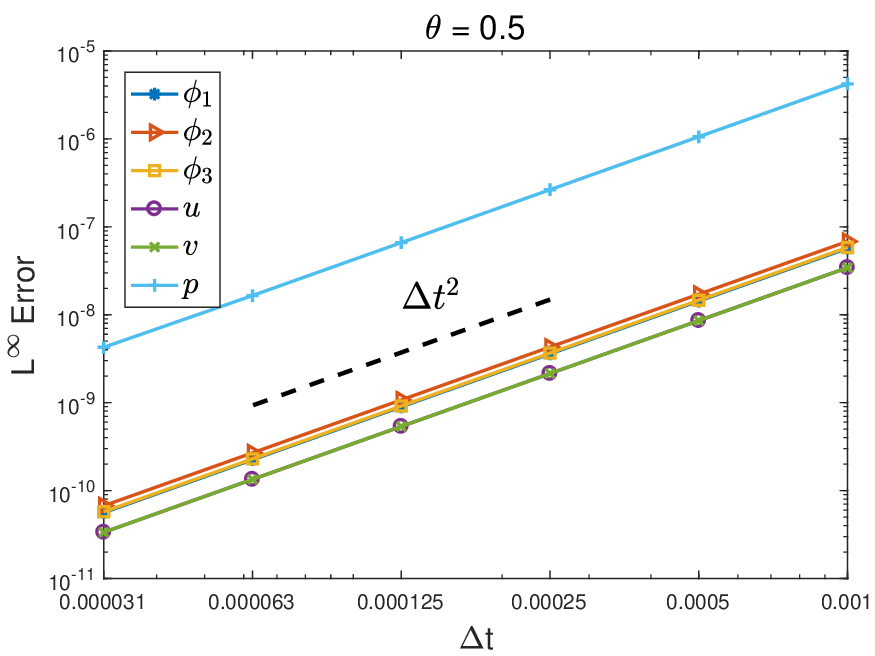}
\includegraphics[width=0.32\textwidth]{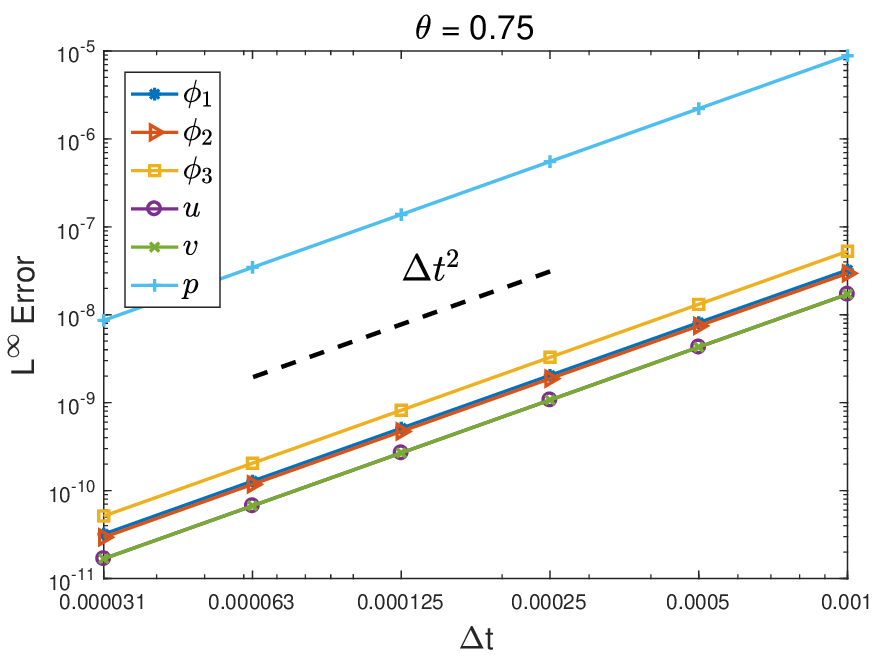}
\includegraphics[width=0.32\textwidth]{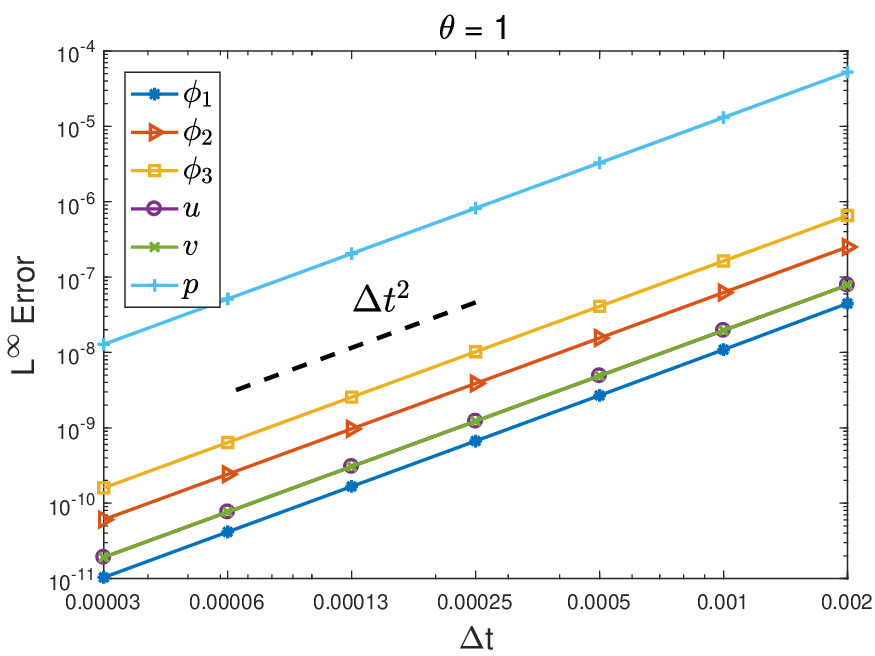}
\caption{Convergence tests of the 3-component NS-CAC model with external forcing terms}
\label{figure:error_3_NS_CAC}
\end{figure}

\begin{figure}[!htp]
\centering
\includegraphics[width=0.32\textwidth]{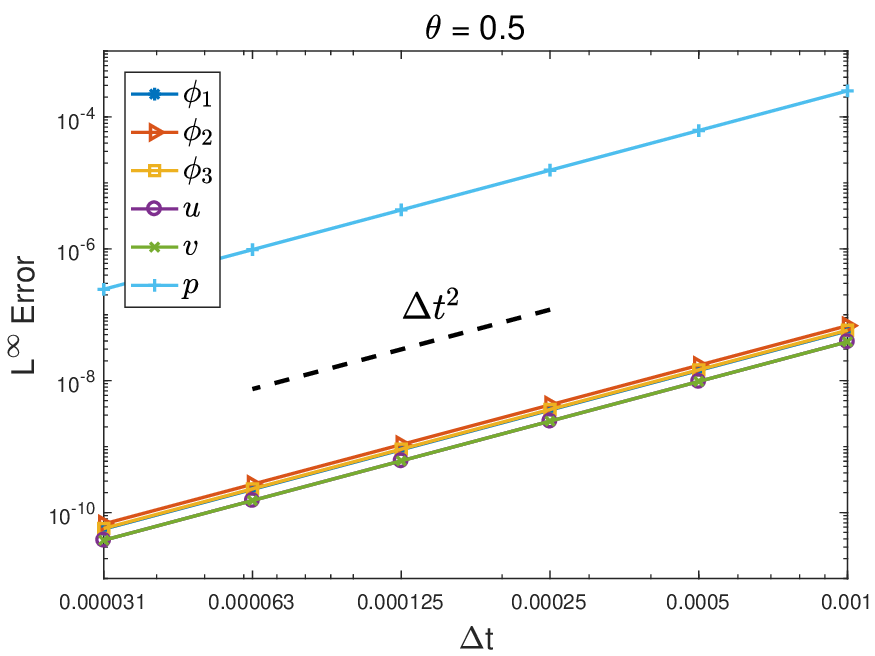}
\includegraphics[width=0.32\textwidth]{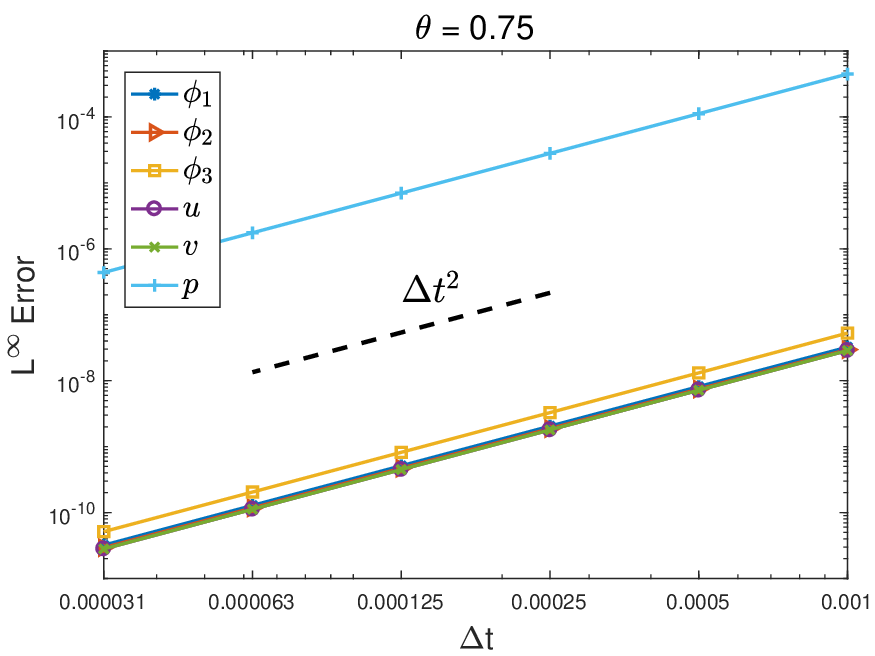}
\includegraphics[width=0.32\textwidth]{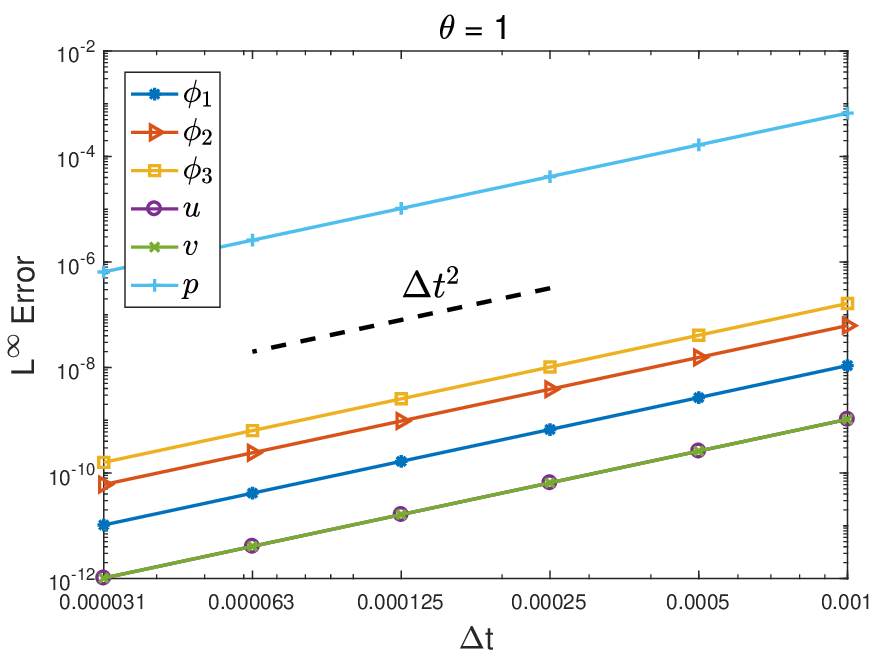}
\caption{Convergence tests of the 3-component D-CAC model with external forcing terms}
\label{figure:error_3_D_CAC}
\end{figure}

\subsection{Energy dissipation and mass conservation} \label{sec-4.2}
In this subsection, in order to verify the energy dissipation and mass conservation of the $\theta$-SAV schemes, we consider the computational domain $\Omega = [0,2]^2$. The spatial mesh size and the parameters are the same as ones in Subsection \ref{sec-4.1}.

\subsubsection{2-component models}
\label{sec-4.2.1}
For 2-component NS-CAC model and 2-component D-CAC model, we assume that the initial conditions are both
\begin{align}
    \left \{
    \begin{aligned}
        &\phi(x,y,0)= \mathrm{rand}(x,y), \\
        &u(x,y,0)=v(x,y,0)=p(x,y,0)=1,
    \end{aligned}
    \right.
    \nonumber
\end{align}
where the $\mathrm{rand}(x,y)$ is the random number between 0 and 1. In Fig. \ref{figure:Energy_1_NS_CAC}, we plot the modified energy until $t=20$ for the different values of $\theta$ with $\Delta t=0.005$, as well as energy evolution until $t=2$ for the different values of $\Delta t$ with $\theta=0.6$, respectively.   It shows that the modified energy of the $\theta$-SAV scheme is no-increasing for 2-component NS-CAC model.
 The relative error of  mass of $\phi$  is denoted as $\Delta M(t)=M(t)-M(0)$, where $M(t)$ is the mass of $\phi$ at time $t$ and $M(0)$ is the initial mass of $\phi$. We draw  $\Delta M(t)$  until $t=5$ using the same values of weight $\theta$ and time steps $\Delta t$ as above in Fig. \ref{figure:Mass_1_NS_CAC}. It illustrates that the $\theta$-SAV scheme is mass conservation  since the $\Delta M(t)$ arrives  at the machine precision. For 2-component D-CAC model, we also plot the similar time evolution of modified energy and relative error of mass in Figs. \ref{figure:Energy_1_D_CAC}-\ref{figure:Mass_1_D_CAC}, which display that the $\theta$-SAV scheme of 2-component D-CAC model also satisfies energy dissipation and mass conservation.


\begin{figure}[!htp]
\centering
\includegraphics[width=0.45\textwidth]{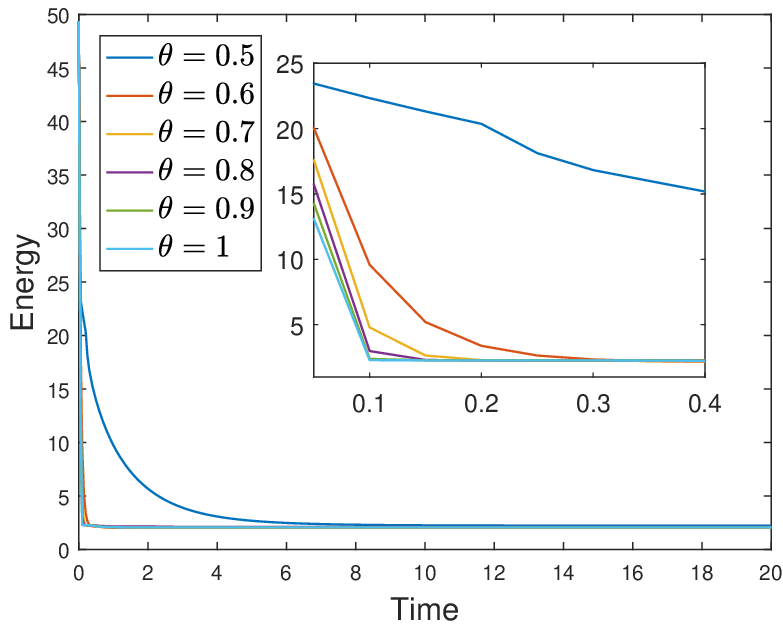}
\includegraphics[width=0.45\textwidth]{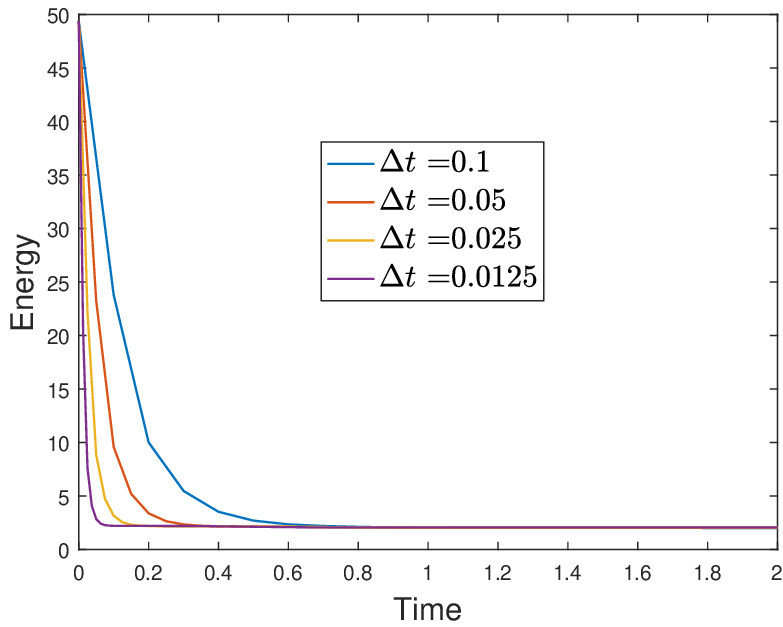}
\caption{The modified energy evolution of the 2-component NS-CAC model}
\label{figure:Energy_1_NS_CAC}
\end{figure}

\begin{figure}[!htp]
\centering
\includegraphics[width=0.45\textwidth]{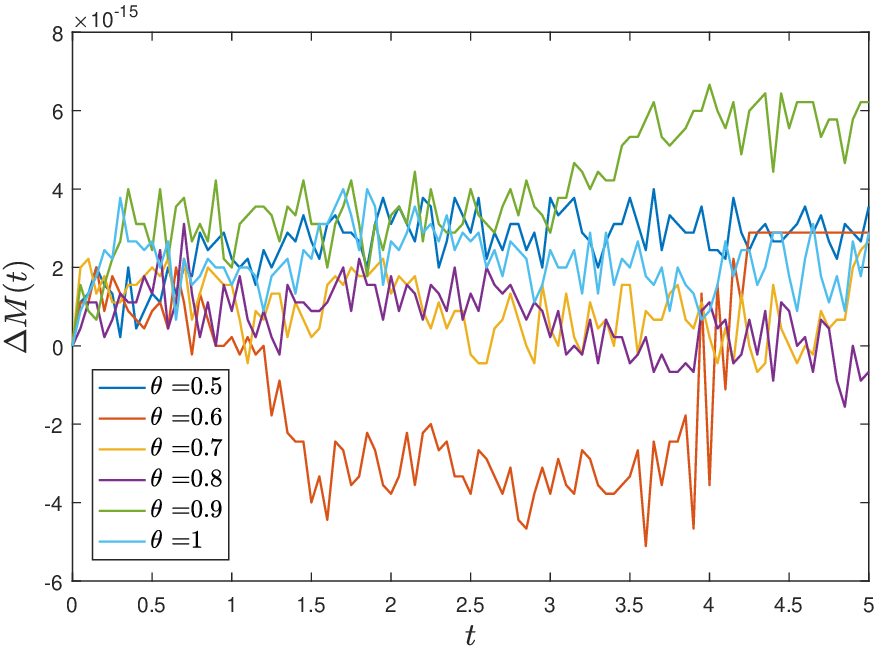}
\includegraphics[width=0.45\textwidth]{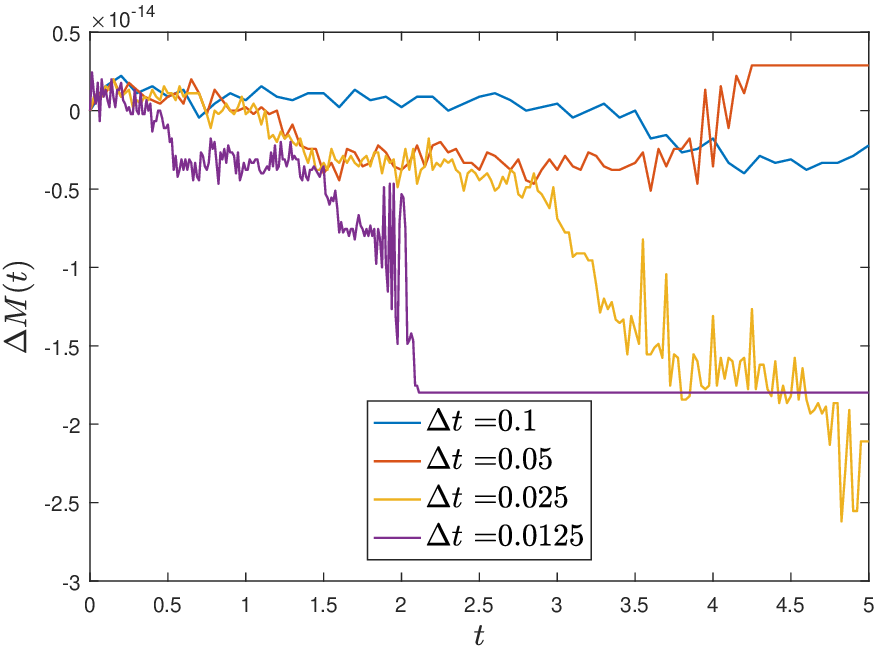}
\caption{The relative error of mass of the 2-component NS-CAC model}
\label{figure:Mass_1_NS_CAC}
\end{figure}

\begin{figure}[!htp]
\centering
\includegraphics[width=0.45\textwidth]{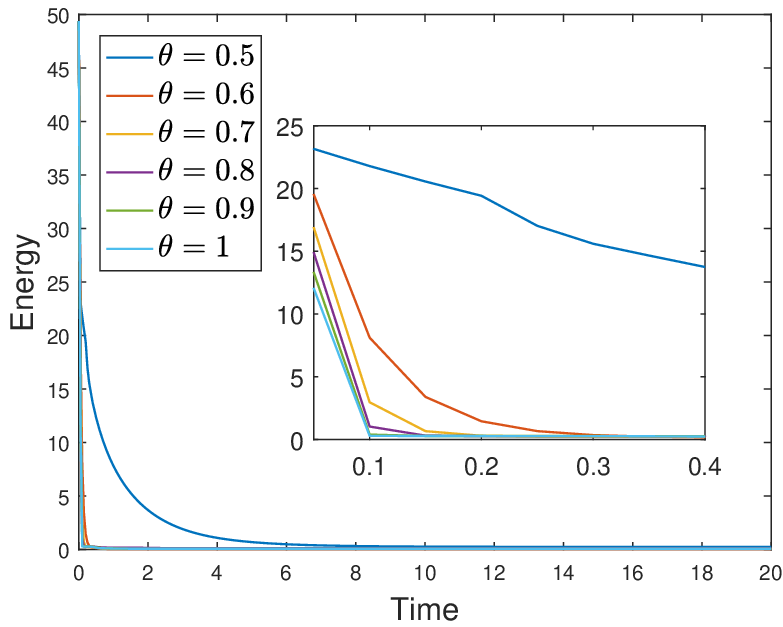}
\includegraphics[width=0.45\textwidth]{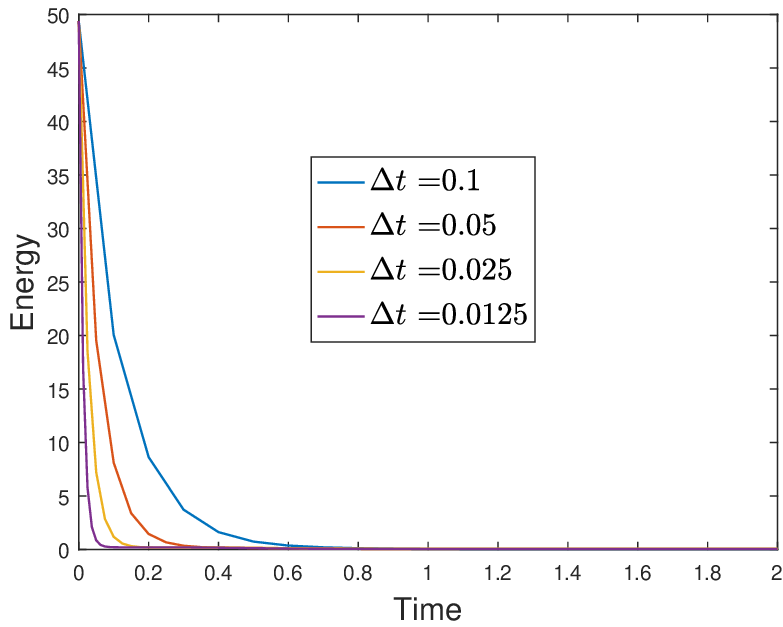}
\caption{The modified energy evolution of the 2-component D-CAC model}
\label{figure:Energy_1_D_CAC}
\end{figure}

\begin{figure}[!htp]
\centering
\includegraphics[width=0.45\textwidth]{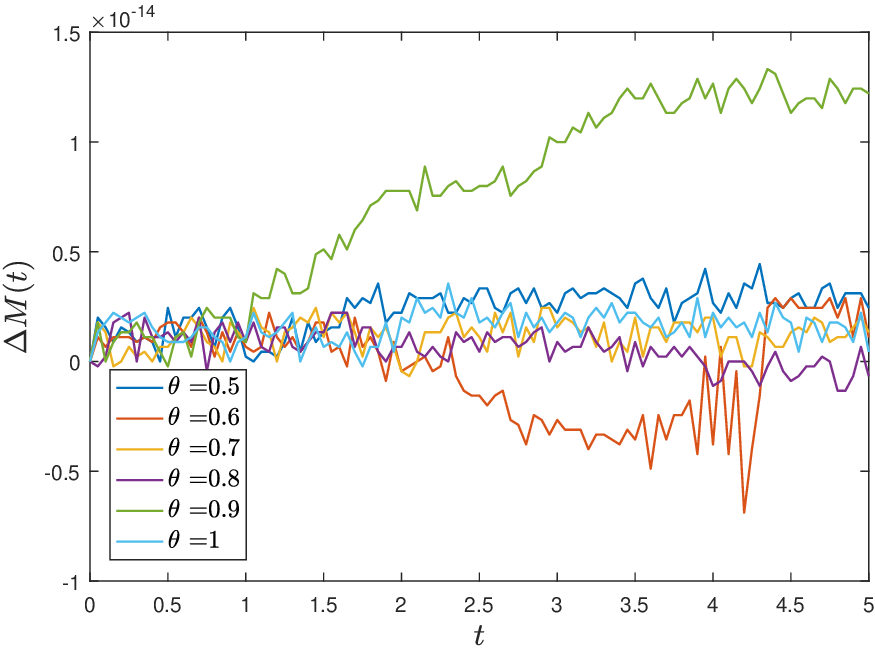}
\includegraphics[width=0.45\textwidth]{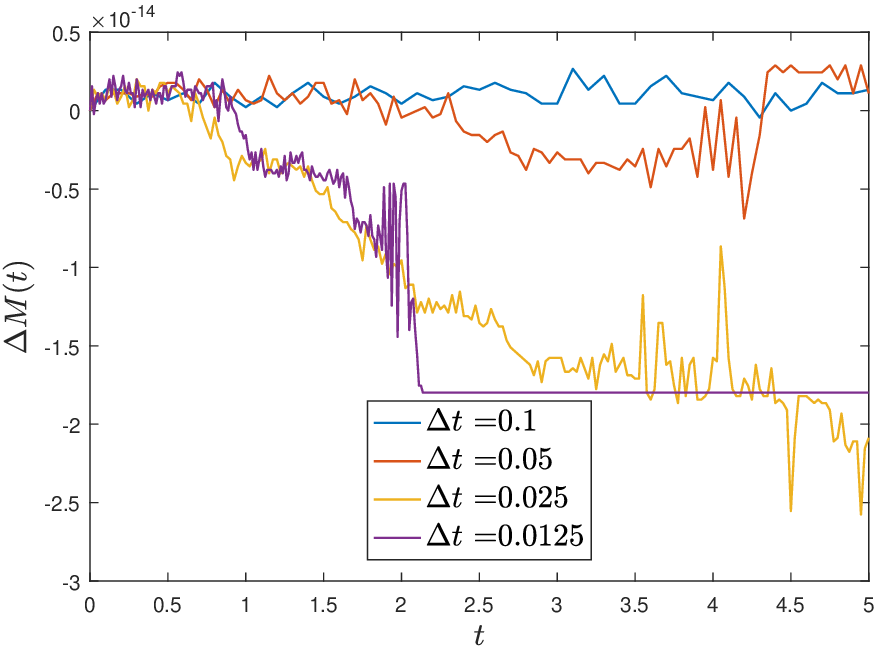}
\caption{The relative errors of mass of the 2-component D-CAC model}
\label{figure:Mass_1_D_CAC}
\end{figure}

\subsubsection{3-component models}
\label{sec-4.2.2}
For 3-component NS-CAC model and 3-component D-CAC model, we suppose that the initial conditions are both
\begin{align}
    \left \{
    \begin{aligned}
        &\phi_1(x,y,0)= \frac{1}{3}+0.01(2\mathrm{rand}(x,y)-1), \\
        &\phi_2(x,y,0)= \frac{1}{3}+0.01(2\mathrm{rand}(x,y)-1),\\
        &\phi_3(x,y,0)= 1 - \phi_1(x,y,0) - \phi_2(x,y,0),\\
        &u(x,y,0)=v(x,y,0)=p(x,y,0)=0.
    \end{aligned}
    \right.
    \nonumber
\end{align}
With the values of $\theta$ and $\Delta t$ as same as ones in Subsubsection \ref{sec-4.2.1}, Figs. \ref{figure:Energy_3_NS_CAC}-
\ref{figure:Mass_theta06_3_D_CAC} are drawn to exhibit the temporal evolution of two aspects: (i) the total modified energy; (ii) the relative error of the mass denoted as $\Delta M_k(t)$  of $\phi_k$ $(k=1,2,3)$. From Figs. \ref{figure:Energy_3_NS_CAC} and \ref{figure:Energy_3_D_CAC}, we can find that the energy dissipation is satisfied. Further, we observe the fact that mass conservation also is satisfied in Figs. \ref{figure:Mass_Delta005_3_NS_CAC}-\ref{figure:Mass_theta06_3_NS_CAC} and \ref{figure:Mass_Delta005_3_D_CAC}-\ref{figure:Mass_theta06_3_D_CAC}.
\par
To sum up, we verify the $\theta$-SAV schemes have the properties of energy dissipation and mass conservation for arbitrary $\theta\in [1/2,1]$, which can reflect the correctness of the previous theoretical results.




\begin{figure}[!htp]
\centering
\includegraphics[width=0.45\textwidth]{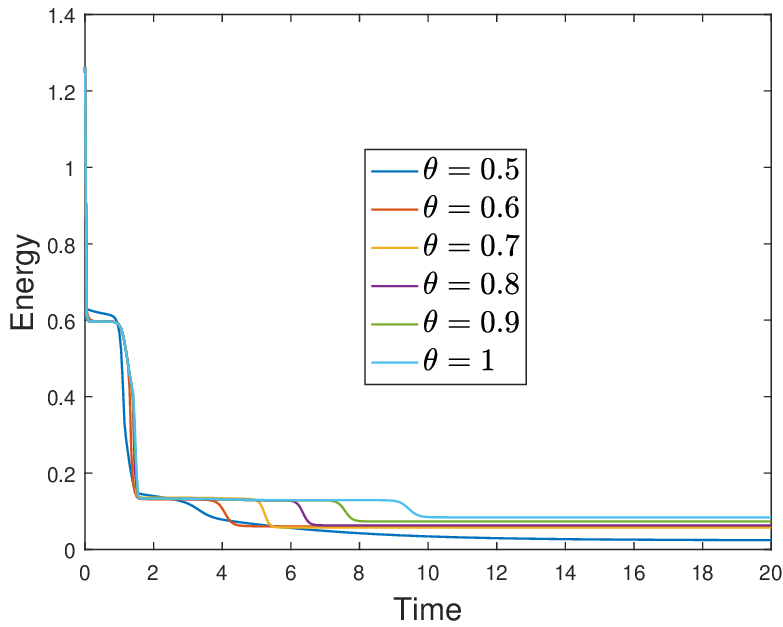}
\includegraphics[width=0.45\textwidth]{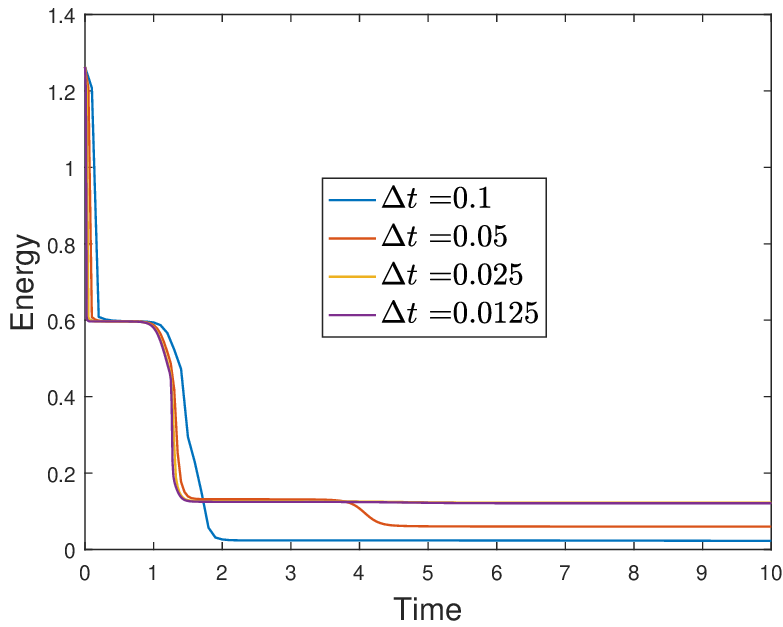}
\caption{The modified energy evolution of the 3-component NS-CAC model}
\label{figure:Energy_3_NS_CAC}
\end{figure}

\begin{figure}[!htp]
\centering
\includegraphics[width=0.32\textwidth]{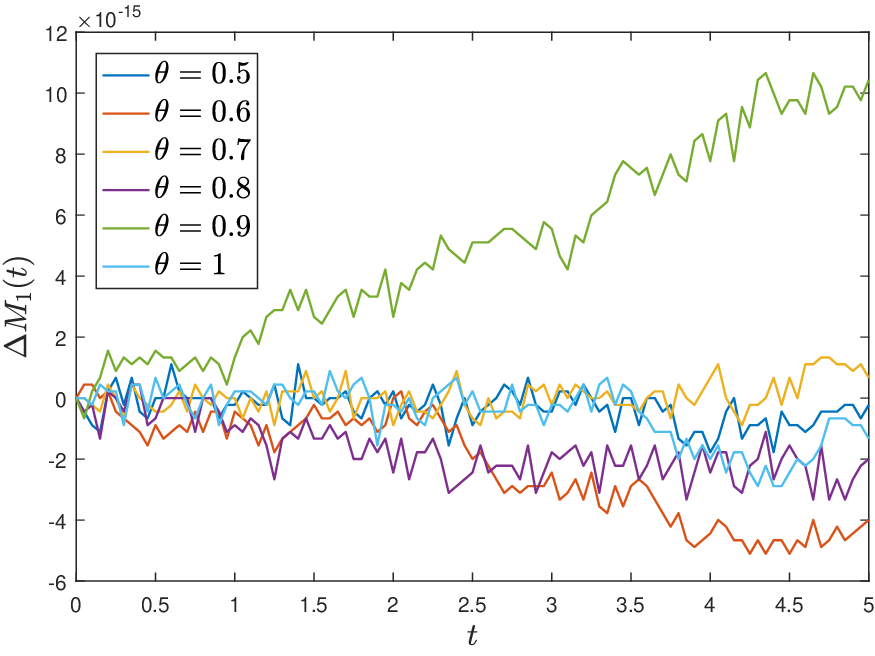}
\includegraphics[width=0.32\textwidth]{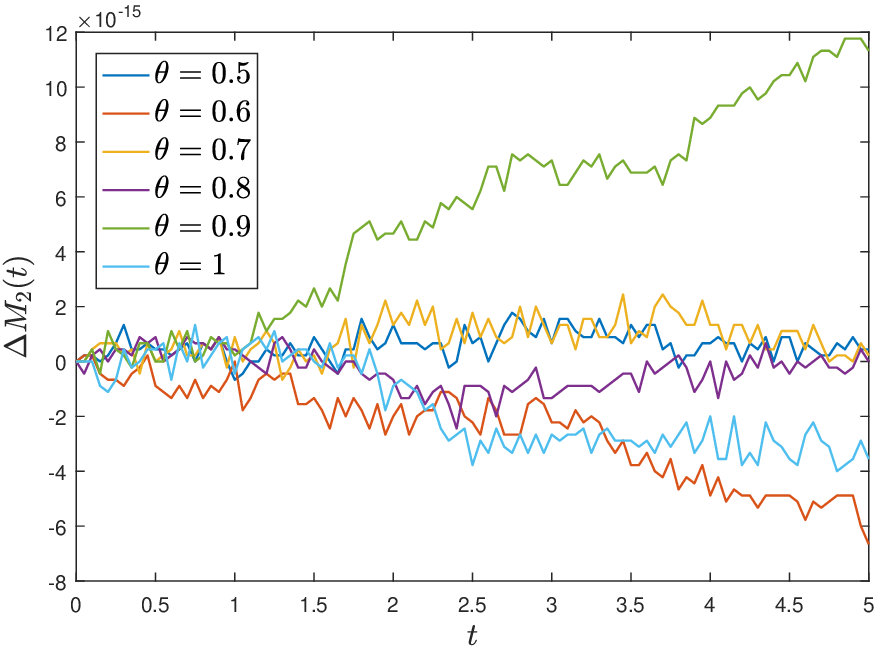}
\includegraphics[width=0.32\textwidth]{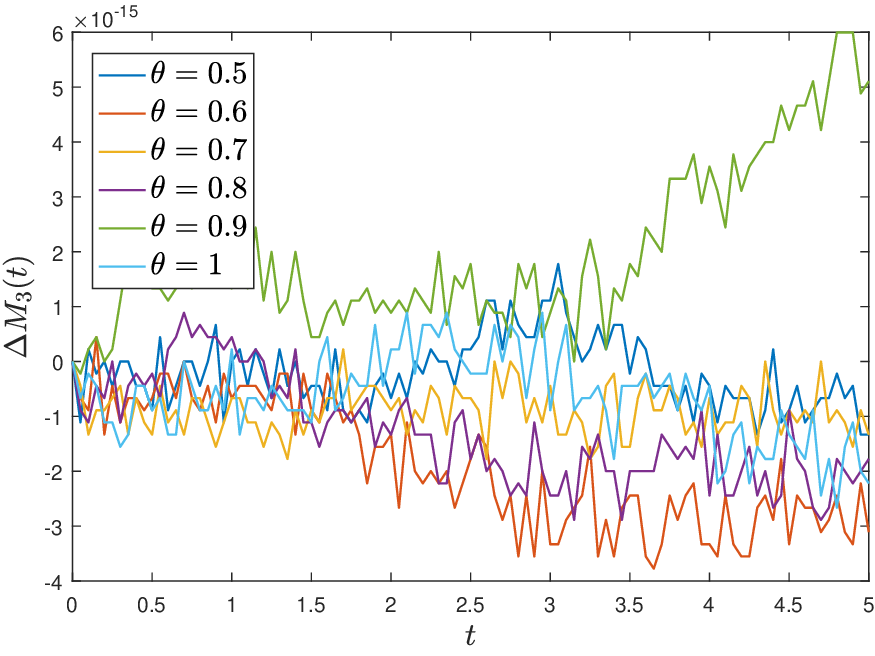}
\caption{The relative errors of mass of the 3-component NS-CAC model with $\Delta t = 0.05$}
\label{figure:Mass_Delta005_3_NS_CAC}
\end{figure}

\begin{figure}[!htp]
\centering
\includegraphics[width=0.32\textwidth]{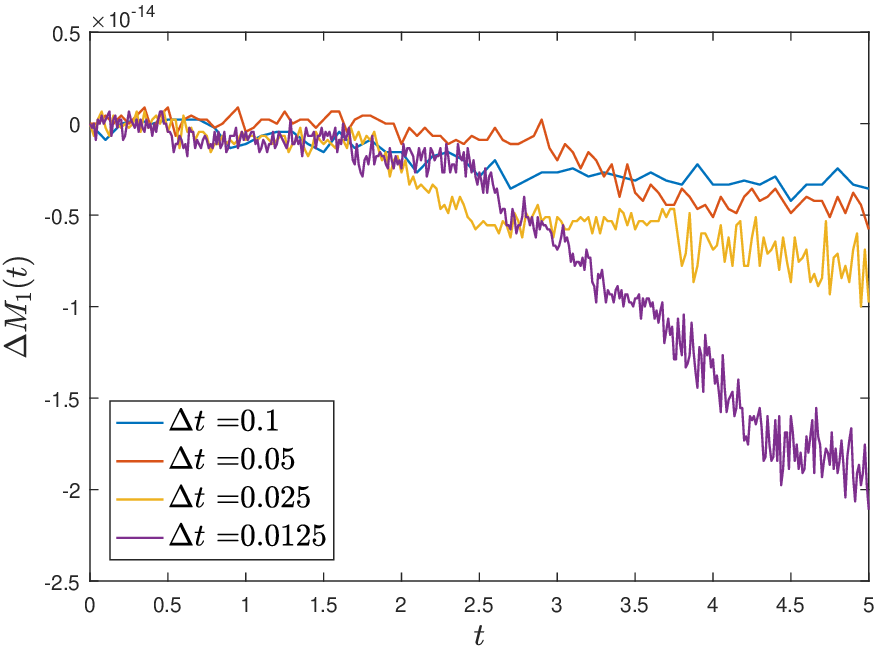}
\includegraphics[width=0.32\textwidth]{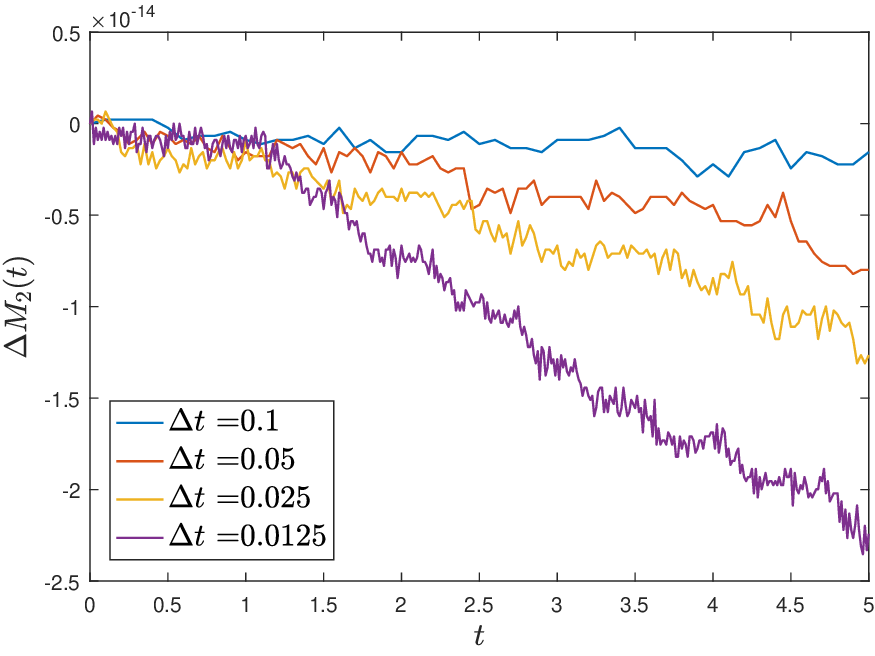}
\includegraphics[width=0.32\textwidth]{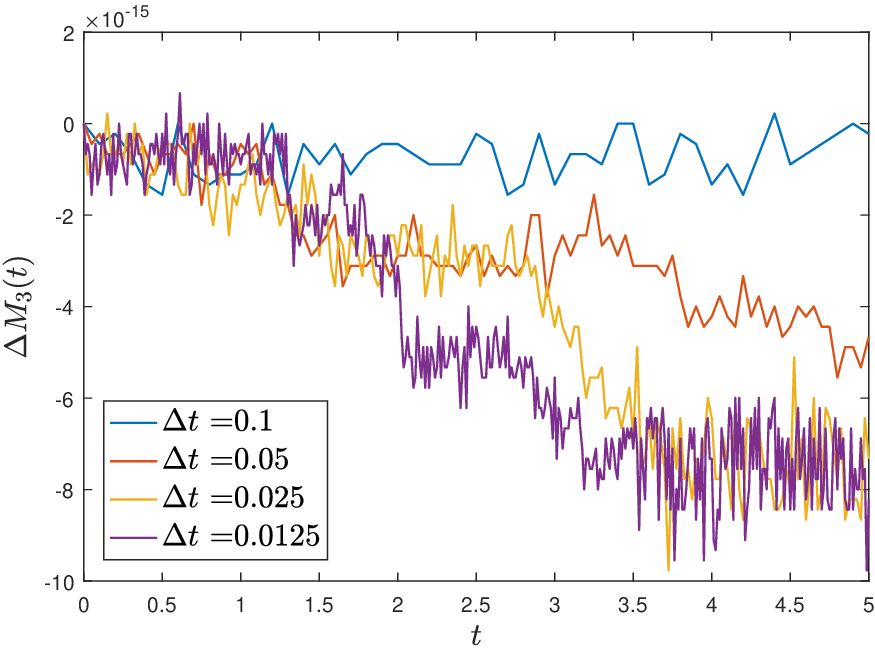}
\caption{The relative errors of mass of the 3-component NS-CAC model with $\theta = 0.6$}
\label{figure:Mass_theta06_3_NS_CAC}
\end{figure}

\begin{figure}[!htp]
\centering
\includegraphics[width=0.45\textwidth]{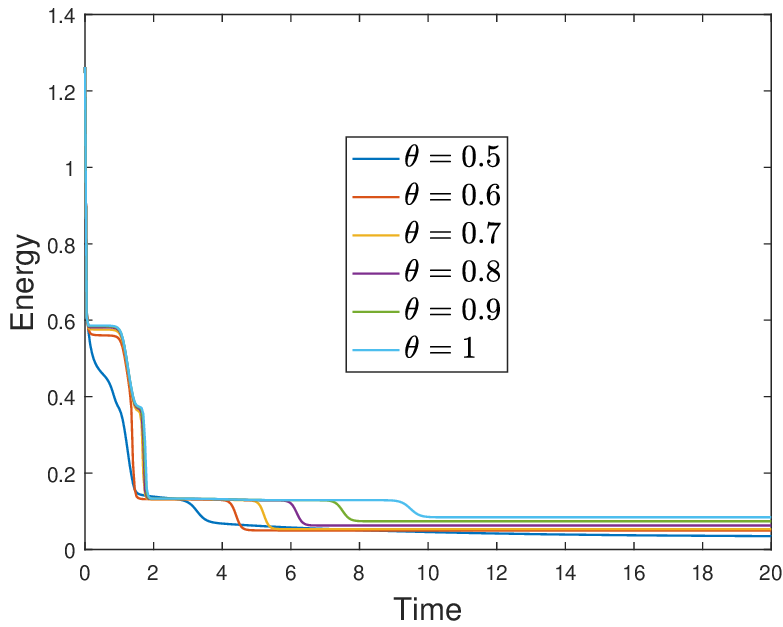}
\includegraphics[width=0.45\textwidth]{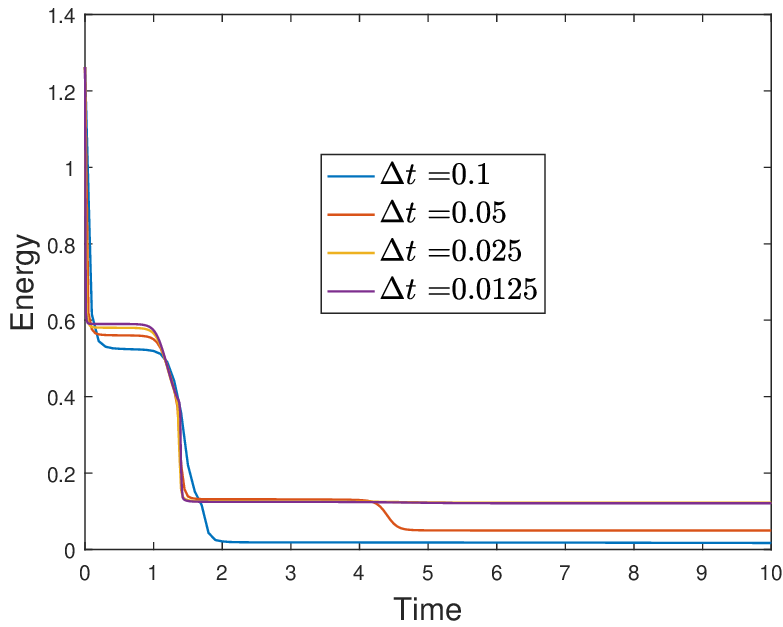}
\caption{The modified energy evolution of the 3-component D-CAC model}
\label{figure:Energy_3_D_CAC}
\end{figure}

\begin{figure}[!htp]
\centering
\includegraphics[width=0.32\textwidth]{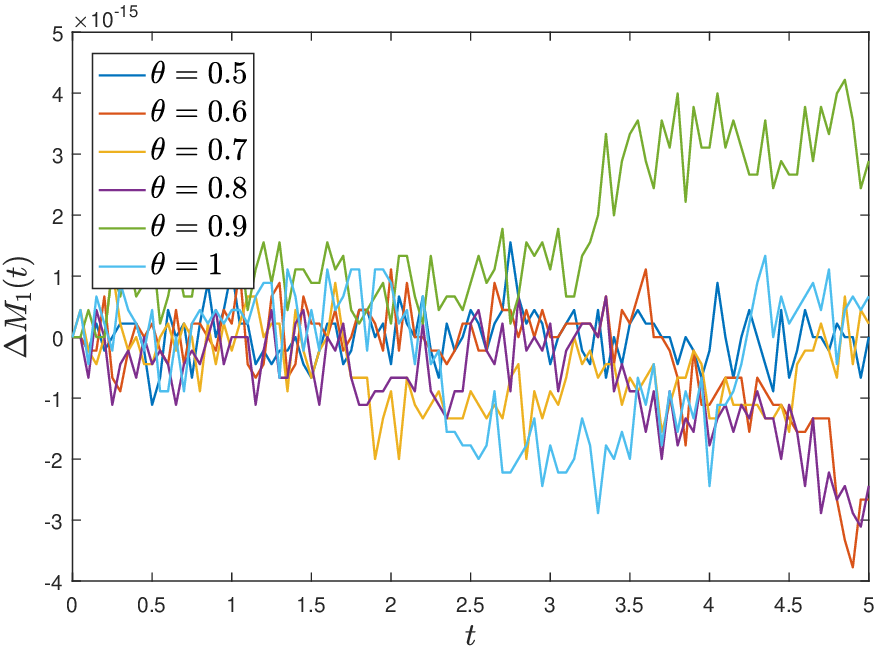}
\includegraphics[width=0.32\textwidth]{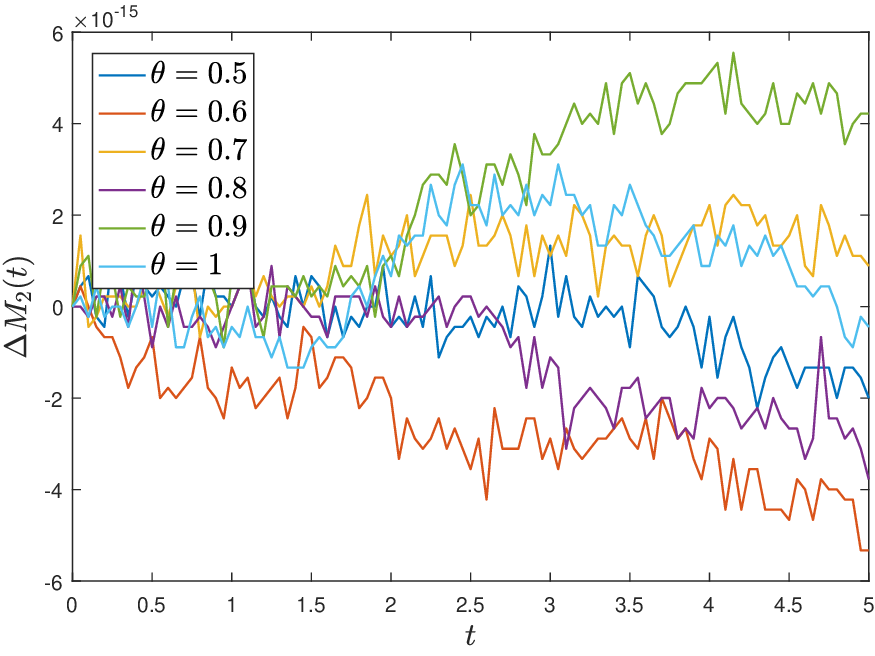}
\includegraphics[width=0.32\textwidth]{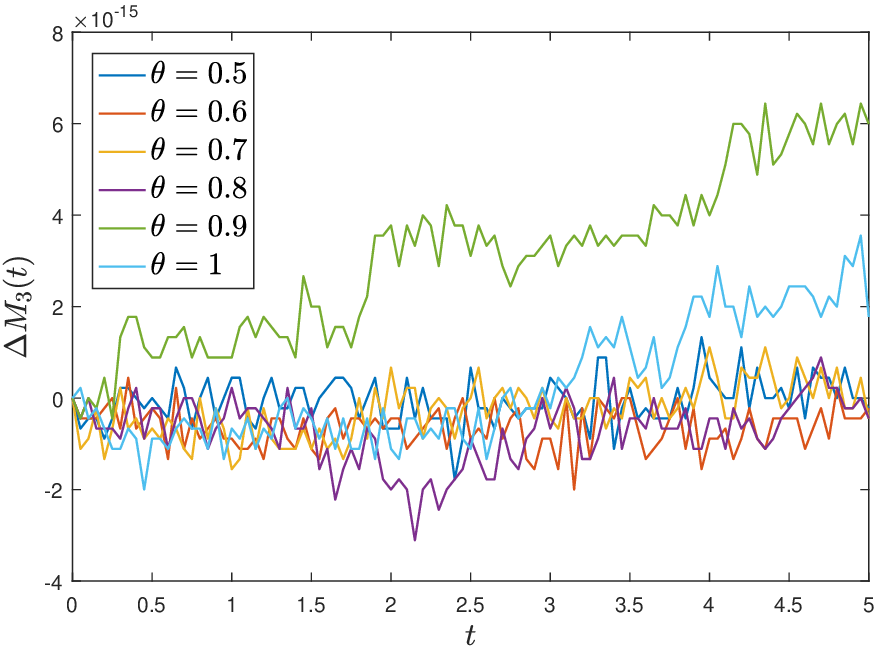}
\caption{The relative errors of mass of the 3-component D-CAC model with $\Delta t = 0.05$}
\label{figure:Mass_Delta005_3_D_CAC}
\end{figure}

\begin{figure}[!htp]
\centering
\includegraphics[width=0.32\textwidth]{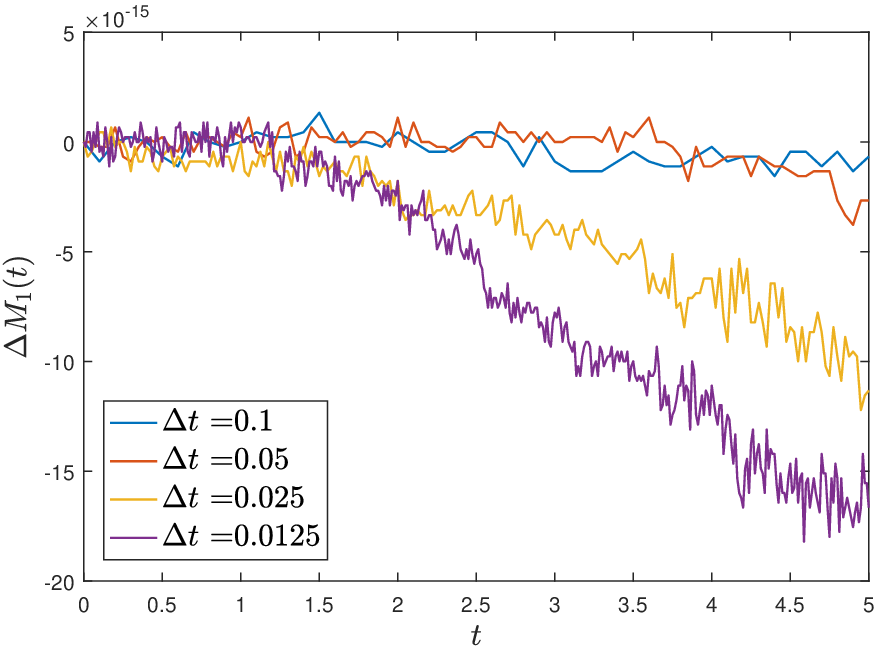}
\includegraphics[width=0.32\textwidth]{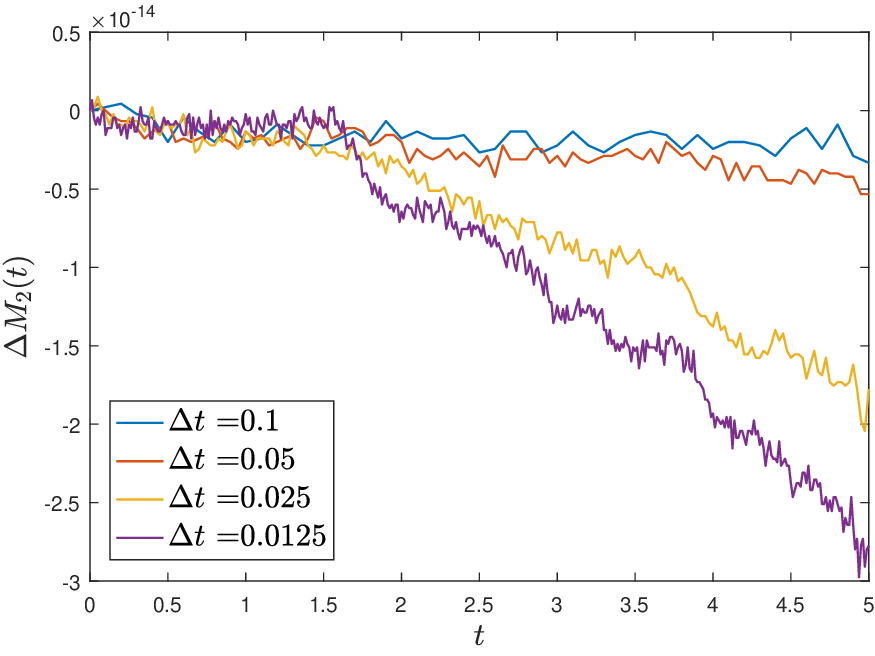}
\includegraphics[width=0.32\textwidth]{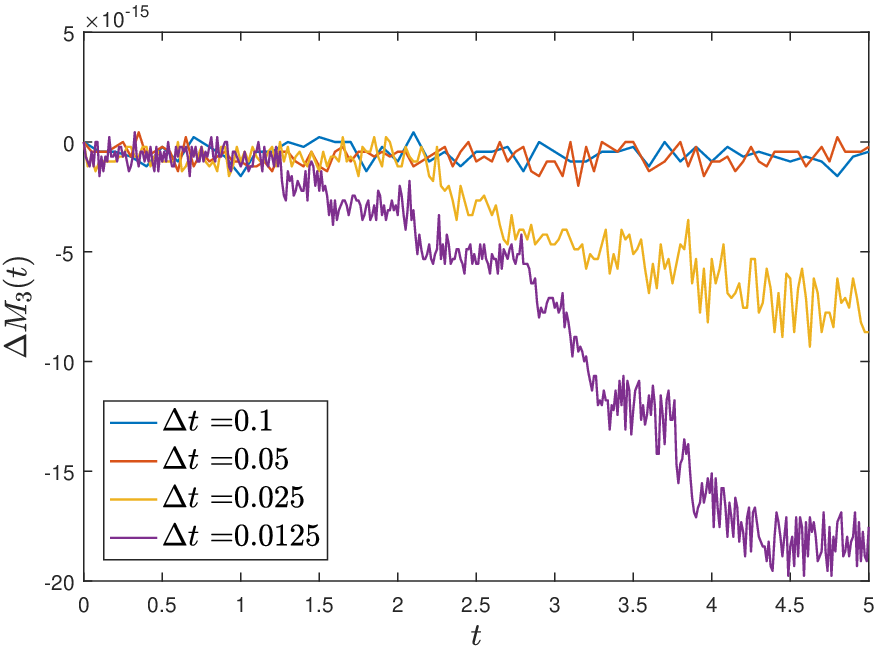}
\caption{The relative errors of mass of the 3-component D-CAC model with $\theta = 0.6$}
\label{figure:Mass_theta06_3_D_CAC}
\end{figure}

\subsection{The phase-separation of 3-component systems}

In this subsection, we study the phase-separation of 3-component NS-CAC model and D-CAC model. For the two models, we take the same initial conditions, i.e.,
\begin{align}
    \left \{
    \begin{aligned}
        &\phi_1(x,y,0)= \frac{1}{3}+0.01(2\mathrm{rand}(x,y)-1), \\
        &\phi_2(x,y,0)= \frac{1}{3}+0.01(2\mathrm{rand}(x,y)-1),\\
        &\phi_3(x,y,0)= \frac{1}{3}+0.01(2\mathrm{rand}(x,y)-1),\\
        &u(x,y,0)=v(x,y,0)=p(x,y,0)=0.
    \end{aligned}
    \right.
    \nonumber
\end{align}
We consider the computational domain $\Omega =[0,1]^2$ and the mesh size $256\times 256$. The parameters $\epsilon = 0.001$, $M = 10$, $\lambda = 0.001$, $C=10$, $\nu = 1$, $\alpha=1000$ are used. In this test, we focus on the case with $\theta=0.6$ and $\Delta t=1e-4$.
Fig. \ref{figure:temporal_evolution_theta06_3_NS_CAC_Phase} displays the temporal evolution of morphologies during phase separation of 3-component NS-CAC model at different
time, where the blue, green and red regions represent $\phi_1$, $\phi_2$, $\phi_3$, respectively.
Fig. \ref{figure:temporal_evolution_theta06_3_NS_CAC_Energy} implies that the total modified energy does not increase in time. For 3-component D-CAC model, we also show the snapshots of phase separation in Fig. \ref{figure:temporal_evolution_theta06_3_D_CAC_Phase} and find that energy dissipation law is also satisfied from Fig. \ref{figure:temporal_evolution_theta06_3_D_CAC_Energy}.

\begin{figure}[!htp]
\centering
\includegraphics[width=0.245\textwidth]{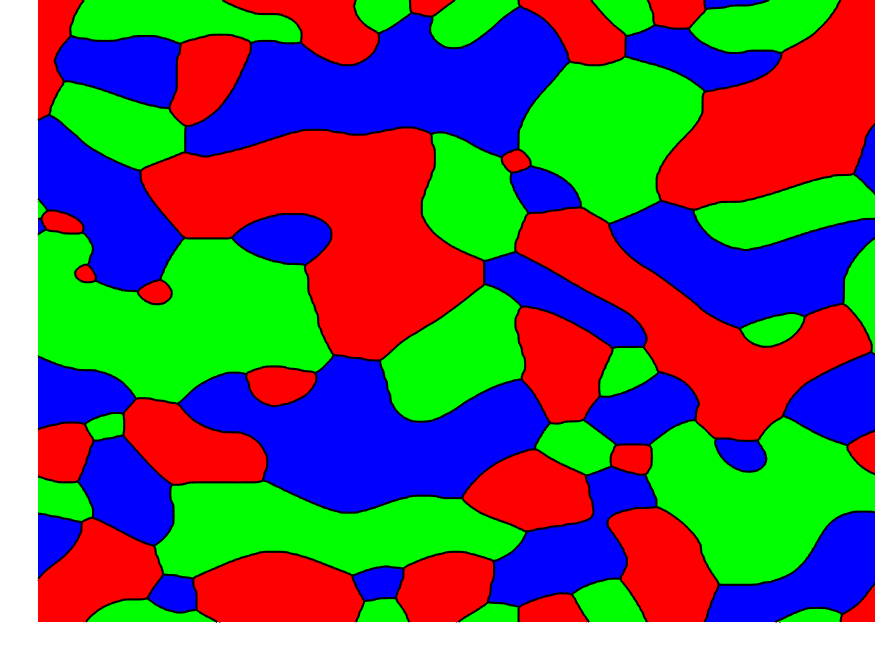}
\includegraphics[width=0.245\textwidth]{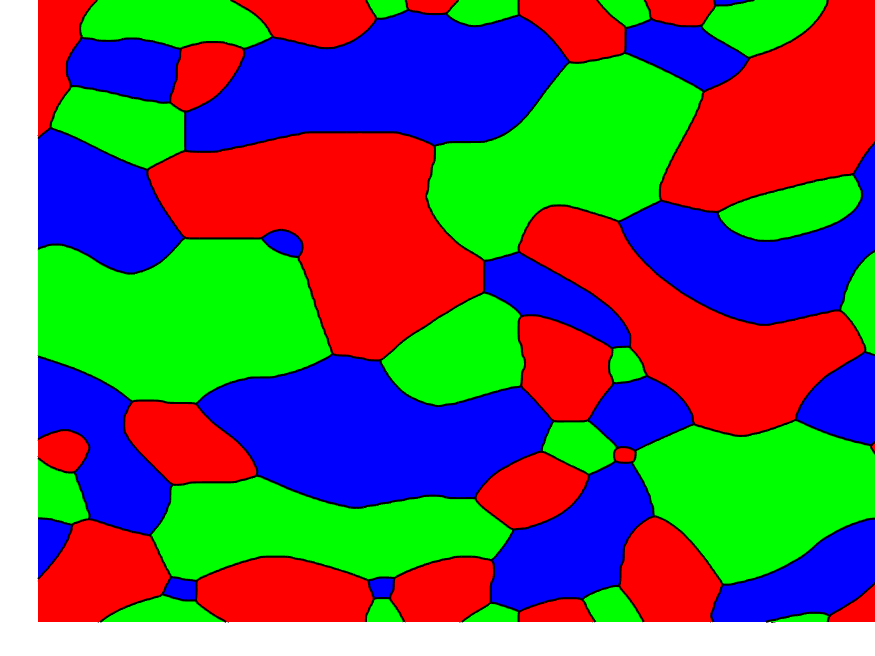}
\includegraphics[width=0.245\textwidth]{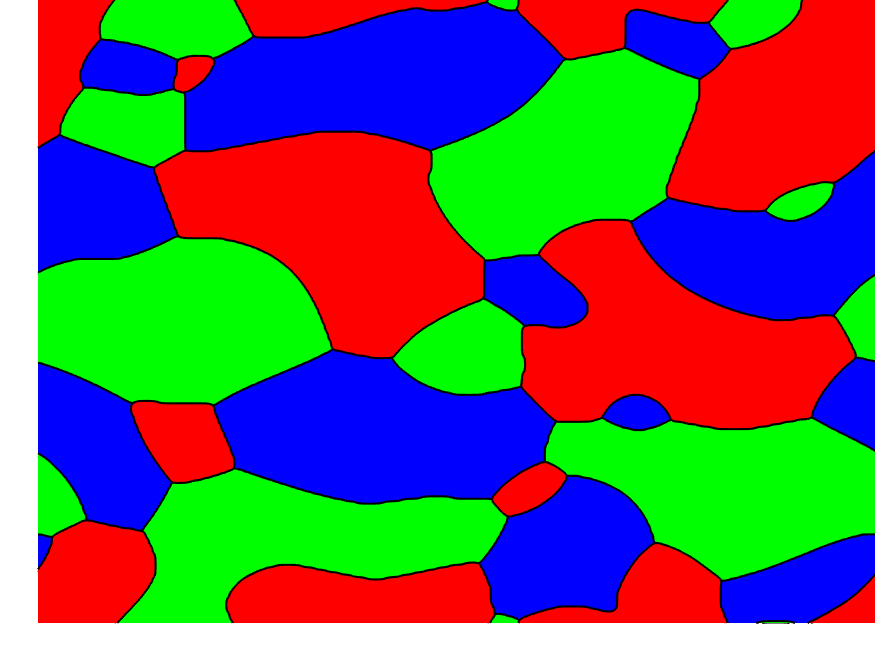}
\includegraphics[width=0.245\textwidth]{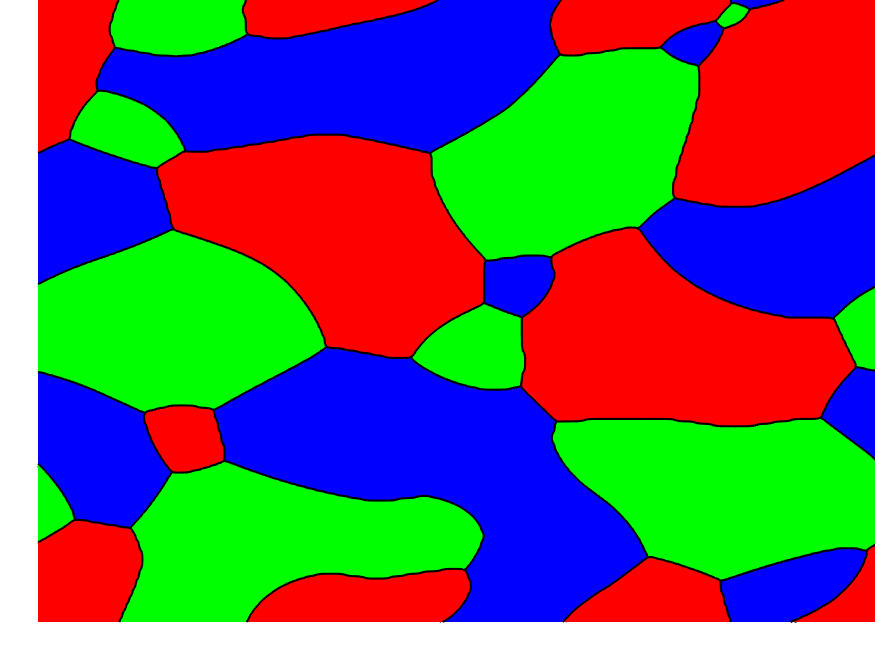}
\caption{The temporal evolution of morphologies during phase separation of 3-component NS-CAC model with $t=0.185, 0.36, 0.655, 1$ (from left to right) when $\theta = 0.6$.}
\label{figure:temporal_evolution_theta06_3_NS_CAC_Phase}
\end{figure}

\begin{figure}[!htp]
\centering
\includegraphics[width=0.6\textwidth]{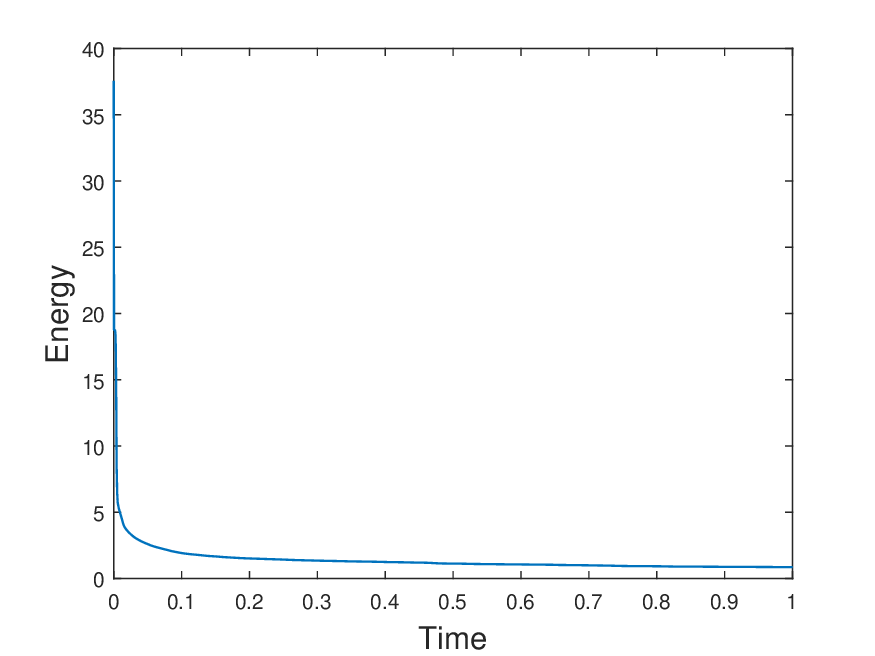}
\caption{The temporal evolution of the total modified energy of 3-component NS-CAC model when $\theta = 0.6$. }
\label{figure:temporal_evolution_theta06_3_NS_CAC_Energy}
\end{figure}

\begin{figure}[!htp]
\centering
\includegraphics[width=0.245\textwidth]{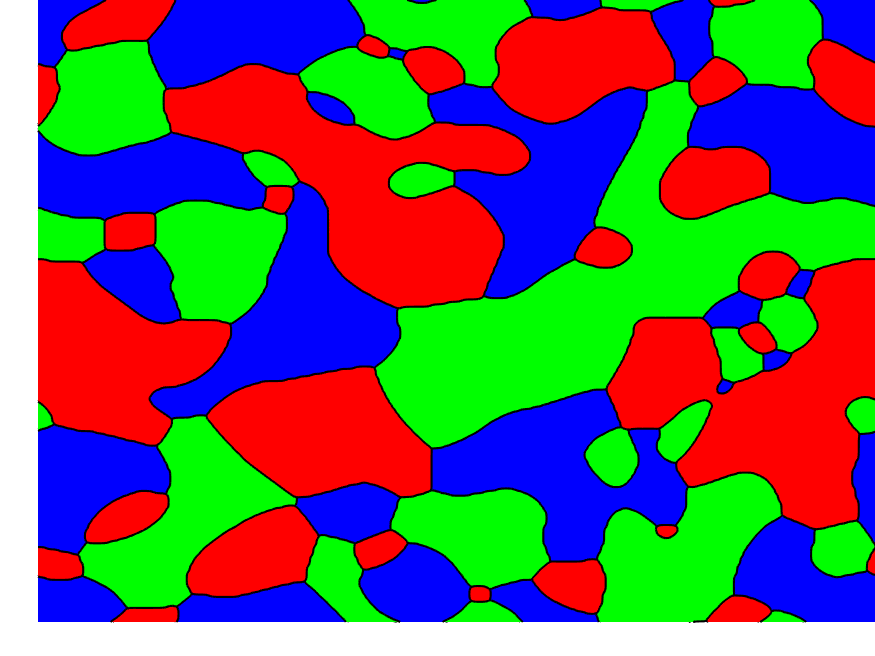}
\includegraphics[width=0.245\textwidth]{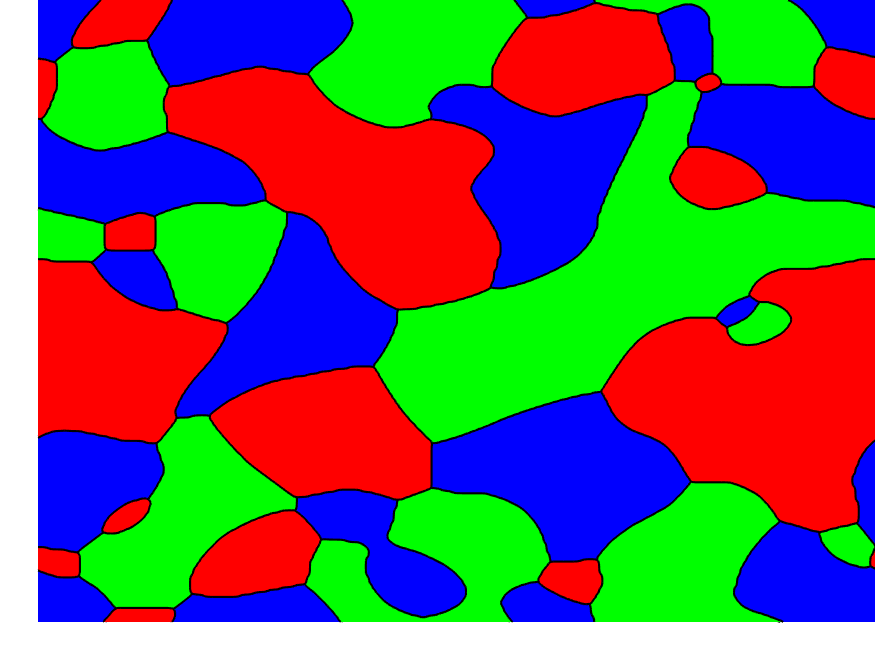}
\includegraphics[width=0.245\textwidth]{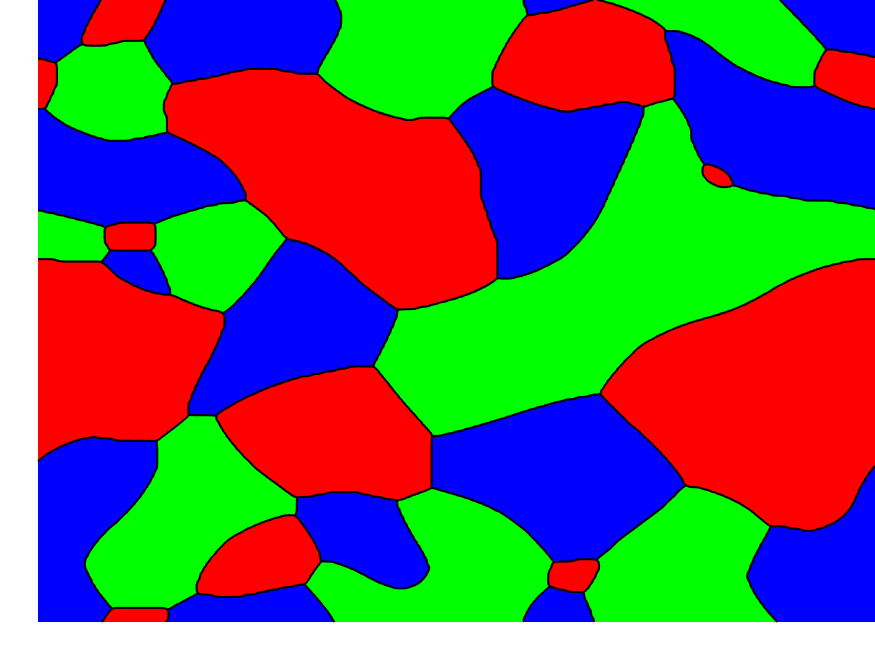}
\includegraphics[width=0.245\textwidth]{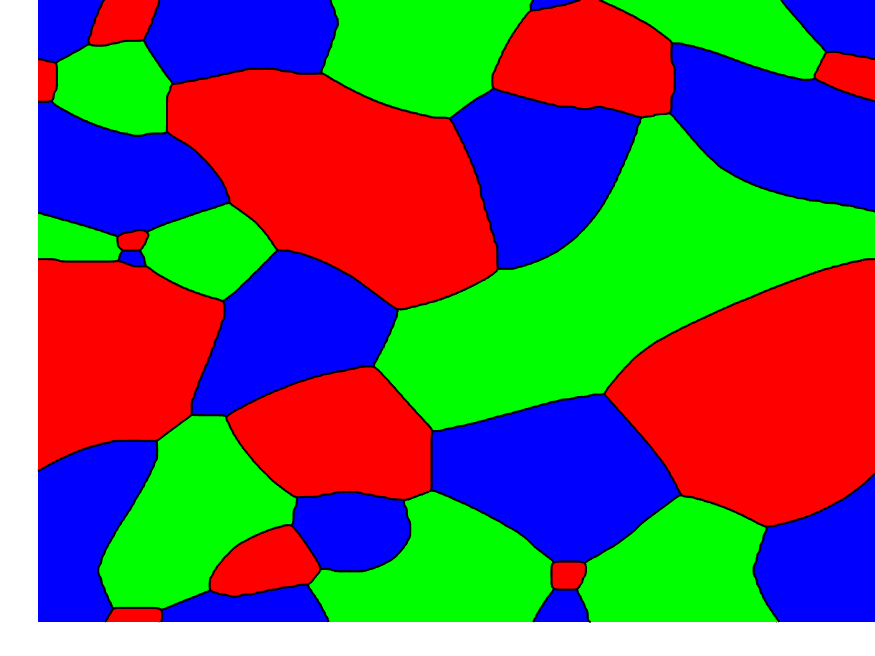}
\caption{The temporal evolution of morphologies during phase separation of 3-component D-CAC model with $t=0.185, 0.36, 0.655, 1$ (from left to right) when $\theta = 0.6$.}
\label{figure:temporal_evolution_theta06_3_D_CAC_Phase}
\end{figure}

\begin{figure}[!htp]
\centering
\includegraphics[width=0.6\textwidth]{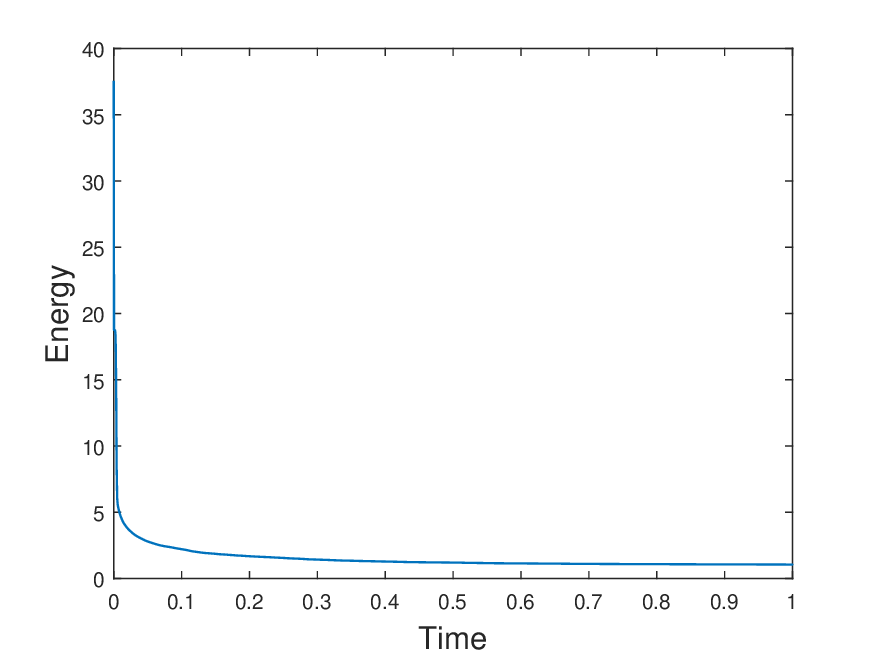}
\caption{The temporal evolution of the total modified energy of 3-component D-CAC model when $\theta = 0.6$. }
\label{figure:temporal_evolution_theta06_3_D_CAC_Energy}
\end{figure}

\section{Conclusions} \label{sec-5}

In this work, we
develop structre-preserving weighted IMEX methods to solve N-component CAC equation copuled with two fluid flow models, i.e., incompressible Navier-Stokes equation and Darcy equation.
The schemes are designed based on
the SAV method, projection method as well as a second-order decoupled IMEX method with arbitrary weight parameter $\theta \in [1/2,1]$.
We rigorously prove the mass conservation and energy stability of the schemes, with the latter employing the idea of G-stability.
The proposed schemes are fully decoupled and explicitly implementable. Moreover, the schemes establish a connection between the CN and BDF2 methods, thereby enhancing the flexibility of the approach.
Finally, through the simulation experiments of two and three components in 2D space, we demonstrate the correctness and effectiveness of the proposed schemes of the theoretical results.

\section*{Acknowledgements}
The work is supported by the
National Natural Science Foundation of China (Nos.11801527,12001499,11971416),
China Postdoctoral Science Foundation (No.2023T160589), Natural Science Foundation of Henan Province (No.222300420256), Training Plan of Young Backbone Teachers in Colleges of Henan Province (No. 2020GGJS230),  Henan University Science and Technology Innovation Talent support program (No.19HASTIT025).

\appendix
\section{Implementation of the schemes}
\subsection{The scheme \eqref{eqn:rrrmodel1_a}-\eqref{eqn:rrrmodel1_g}}\label{A.1}
In the implementation described below, we use the auxiliary variables to split other variables. Specific steps are as follows:
\par
$\bf Step\ \ 1$ \label{Step1}
Using $q^{n+\theta}$, we split $\phi_k^{n+1}, \mu_k^{n+1}, r^{n+1}, {\bf \tilde u}^{n+1}$ into a linear combination form as
\begin{equation}
\label{system1_0}
    \left \{
    \begin{aligned}
        &\phi_k^{n+1}=\phi_{k1}^{n+1}+q^{n+\theta}\phi_{k2}^{n+1},
       &\mu_k^{n+1}=\mu_{k1}^{n+1}+q^{n+\theta}\mu_{k2}^{n+1},  \\
        &r^{n+1}=r_1^{n+1}+q^{n+\theta}r_2^{n+1},
        &{\bf \tilde u}^{n+1}={\bf \tilde u}_1^{n+1}+q^{n+\theta}{\bf \tilde u}_2^{n+1}.
    \end{aligned}
    \right.
\end{equation}
Substituting the above  into \eqref{eqn:rrrmodel1_a}-\eqref{eqn:rrrmodel1_d} and grouping the terms with $q^{n+\theta}$ and without $q^{n+\theta}$, we can obtain four subsystems as follows:
\begin{align}
    &\left \{
    \begin{aligned}
    \label{system1_sub1}
       &\frac{\frac{2\theta+1}{2}\phi_{k1}^{n+1}-2\theta\phi_k^n+\frac{2\theta-1}{2}\phi_k^{n-1}}{\Delta t}
       +M(\theta\mu_{k1}^{n+1}+(1-\theta)\mu_k^n)=0,\\
       &\theta\mu_{k1}^{n+1}+(1-\theta)\mu_k^n
       =\lambda \left(-\Delta(\theta\phi_{k1}^{n+1}+(1-\theta)\phi_k^n)+(\bar{H}_k^{*}+\gamma^{*})(\theta r_1^{n+1}+(1-\theta)r^n)\right),\\
       &\frac{\frac{2\theta+1}{2}r_{1}^{n+1}-2\theta r^n+\frac{2\theta-1}{2}r^{n-1}}{\Delta t}
       =\frac{1}{2}\sum_{k=1}^{N}\int_{\Omega}\bar{H}_k^{*}\frac{\frac{2\theta+1}{2}\phi_{k1}^{n+1}-2\theta\phi_k^n+\frac{2\theta-1}{2}\phi_k^{n-1}}{\Delta t}d{\bf x},
    \end{aligned}
    \right.\\
    &\left \{
    \begin{aligned}
    \label{system1_sub2}
       &\frac{\frac{2\theta+1}{2}\phi_{k2}^{n+1}}{\Delta t}+\nabla \cdot ({\bf u}^{*}\phi_k^{*})
       +M\theta\mu_{k2}^{n+1}=0,\\
       &\theta\mu_{k2}^{n+1}
       =\lambda \left(-\Delta(\theta\phi_{k2}^{n+1})+(\bar{H}_k^{*}+\gamma^{*})\theta r_2^{n+1}\right),\\
       &\frac{\frac{2\theta+1}{2}r_{2}^{n+1}}{\Delta t}
       =\frac{1}{2}\sum_{k=1}^{N}\int_{\Omega}\bar{H}_k^{*}\frac{\frac{2\theta+1}{2}\phi_{k2}^{n+1}}{\Delta t}d{\bf x},
    \end{aligned}
    \right.\\
    &\frac{\frac{2\theta+1}{2}{\bf \tilde u}_{1}^{n+1}-2\theta{\bf u}^n+\frac{2\theta-1}{2}{\bf u}^{n-1}}{\Delta t}-\nu \Delta (\theta {\bf \tilde u}_{1}^{n+1}+(1-\theta){\bf u}^n)+\nabla p^n=0,
    \label{system1_sub3}\\
    &\frac{\frac{2\theta+1}{2}{\bf \tilde u}_{2}^{n+1}}{\Delta t}+{\bf u}^{*}\cdot \nabla {\bf u}^{*}-\nu \Delta (\theta {\bf \tilde u}_{2}^{n+1})
    +\sum_{k=1}^{N}\phi_k^{*}\nabla \mu_k^{*}=0,\label{system1_sub4}
\end{align}
where the boundary conditions are required be periodic or satisfy the following conditions
\begin{equation}
    \left. \nabla\phi_{k1}^{n+1}\cdot{\bf n}\right|_{\partial\Omega}
    =\left.\nabla\phi_{k2}^{n+1}\cdot{\bf n}\right|_{\partial\Omega}=0,
    \left.
    {\bf \tilde u}^{n+1}_1\right|_{\partial\Omega}=\left.
    {\bf \tilde u}^{n+1}_2\right|_{\partial\Omega}={\bf 0}.
\end{equation}
It is very easy to solve ${\bf \tilde u}_1^{n+1}$ and ${\bf \tilde u}_2^{n+1}$, since the subsystems \eqref{system1_sub3} and \eqref{system1_sub4} are both linear elliptic.

In addition, applying the integration by parts and the definitions of $\bar{H}_k^{*}$ and $\gamma^{*}$, we get
\begin{equation}
    (\bar{H}_k^{*},1)=0,\ \ (\gamma^{*},1)=0,\ \
    \left(\nabla\cdot({\bf u}^{*}\phi_k^{*}),1\right)=0,\ \
    (-\Delta\phi_{ki}^{n+1},1)=0,\ \ i=1,2.
\end{equation}
For the first two equations of \eqref{system1_sub2}, we both take the $L^2$-inner products with 1 and then have $\int_{\Omega}\phi_{k2}^{n+1}d{\bf x}=0$, which is obtained to prove that $q^{n+1}$ or $q^{n+\theta}$ is solvable in $\bf Step\ \ 3$.
\par
$\bf Step\ \ 2$ \label{Step2}
From the subsystem \eqref{system1_sub1}, using $r_1^{n+1}$, we divide $\phi_{k1}^{n+1}$, $\mu_{k1}^{n+1}$ into two parts: $\phi_{k1}^{n+1}=\phi_{k11}^{n+1}+r_1^{n+1}\phi_{k12}^{n+1}$, $\mu_{k1}^{n+1}=\mu_{k11}^{n+1}+r_1^{n+1}\mu_{k12}^{n+1}$ and obtain
\begin{align}
    &\left \{
    \begin{aligned}
    \label{system1_sub1.1}
       &\frac{\frac{2\theta+1}{2}\phi_{k11}^{n+1}-2\theta\phi_k^n+\frac{2\theta-1}{2}\phi_k^{n-1}}{\Delta t}
       +M(\theta\mu_{k11}^{n+1}+(1-\theta)\mu_k^n)=0,\\
       &\theta\mu_{k11}^{n+1}+(1-\theta)\mu_k^n
       =\lambda \left(-\Delta(\theta\phi_{k11}^{n+1}+(1-\theta)\phi_k^n)+(\bar{H}_k^{*}+\gamma^{*})(1-\theta)r^n\right),
    \end{aligned}
    \right.\\
    &\left \{
    \begin{aligned}
    \label{system1_sub1.2}
       &\frac{\frac{2\theta+1}{2}\phi_{k12}^{n+1}}{\Delta t}
       +M\theta\mu_{k12}^{n+1}=0,\\
       &\theta\mu_{k12}^{n+1}
       =\lambda \left(-\Delta(\theta\phi_{k12}^{n+1})+(\bar{H}_k^{*}+\gamma^{*})\theta\right),
    \end{aligned}
    \right.\\
    &\left(\frac{2\theta+1}{2}-\frac{1}{2}\sum_{k=1}^{N}\int_{\Omega}\bar{H}_k^{*}\frac{2\theta+1}{2}\phi_{k12}^{n+1}d{\bf x}\right)r_{1}^{n+1}
    \nonumber\\
    &=(2\theta r^n-\frac{2\theta-1}{2}r^{n-1})
    +\frac{1}{2}\sum_{k=1}^{N}
    \int_{\Omega}\bar{H}_k^{*}\left(\frac{2\theta+1}{2}\phi_{k11}^{n+1}-2\theta\phi_k^n+\frac{2\theta-1}{2}\phi_k^{n-1}\right)d{\bf x},
    \label{system1_sub1.3}
\end{align}
where the boundary conditions are periodic or $\left. \nabla\phi_{k1j}^{n+1}\cdot{\bf n}\right|_{\partial\Omega}=0\ \ (j=1,2)$. Then $\phi_{k11}^{n+1}$, $\mu_{k11}^{n+1}$, $\phi_{k12}^{n+1}$ and $\mu_{k12}^{n+1}$ can be updated by solving the linear elliptic equations \eqref{system1_sub1.1} and \eqref{system1_sub1.2}.
\par
Moreover, we can find that \eqref{system1_sub1.3} is solvable. From \eqref{system1_sub1.2}, we can rewrite it in the following form by eliminating $\mu_{k12}$
\begin{equation}
\label{eqn:A.9_1}
    \frac{2\theta+1}{2\Delta t}\phi_{k12}^{n+1}
    +M\lambda\theta(-\Delta\phi_{k12}^{n+1})
    +M\lambda\theta\bar{H}_k^{*}
    +M\lambda\theta\gamma^{*}=0,
\end{equation}
which can easily obtain $\int_{\Omega}\phi_{k12}^{n+1}d{\bf x}=0$ by taking the $L^2$-inner product with 1 and using the integration by parts. Then, computing the $L^2$-inner product of \eqref{eqn:A.9_1} with $\phi_{k12}^{n+1}$ and applying $\int_{\Omega}\phi_{k12}^{n+1}d{\bf x}=0$, we have
\begin{equation}
    -\int_{\Omega}\bar{H}_k^{*}\phi_{k12}^{n+1}d{\bf x}
    = \frac{2\theta+1}{2M\lambda\theta\Delta t}\left\|\phi_{k12}^{n+1}\right\|^2
    +\left\|\nabla\phi_{k12}^{n+1}\right\|^2\geq 0,
\end{equation}
which means that $\frac{2\theta+1}{2}-\frac{1}{2}\sum_{k=1}^{N}\int_{\Omega}\bar{H}_k^{*}\frac{2\theta+1}{2}\phi_{k12}^{n+1}d{\bf x}>0$ always holds. Therefore, we can update $r_1^{n+1}$ by using \eqref{system1_sub1.3} in an explicit way, and then obtain $\phi_{k1}^{n+1}$ and $\mu_{k1}^{n+1}$.

Similarly, using $r^{n+1}$, \eqref{system1_sub2} can be reformulated as
\begin{align}
    &\left \{
    \begin{aligned}
    \label{system1_sub2.1}
       &\frac{\frac{2\theta+1}{2}\phi_{k21}^{n+1}}{\Delta t}+\nabla \cdot ({\bf u}^{*}\phi_k^{*})
       +M\theta\mu_{k21}^{n+1}=0,\\
       &\theta\mu_{k21}^{n+1}
       =\lambda (-\Delta(\theta\phi_{k21}^{n+1})),
    \end{aligned}
    \right.\\
    &\left \{
    \begin{aligned}
    \label{system1_sub2.2}
       &\frac{\frac{2\theta+1}{2}\phi_{k22}^{n+1}}{\Delta t}+M\theta\mu_{k22}^{n+1}=0,\\
       &\theta\mu_{k22}^{n+1}
       =\lambda \left(-\Delta(\theta\phi_{k22}^{n+1})+(\bar{H}_k^{*}+\gamma^{*})\theta\right),
    \end{aligned}
    \right.\\
     &\left(\frac{2\theta+1}{2}-\frac{1}{2}\sum_{k=1}^{N}\int_{\Omega}\bar{H}_k^{*}\frac{2\theta+1}{2}\phi_{k22}^{n+1}\right)r_{2}^{n+1}
     =\frac{1}{2}\sum_{k=1}^{N}\int_{\Omega}\bar{H}_k^{*}\frac{2\theta +1}{2}\phi_{k21}^{n+1}d{\bf x},
     \label{system1_sub2.3}
\end{align}
where the boundary conditions are periodic or $\left. \nabla\phi_{k2j}^{n+1}\cdot{\bf n}\right|_{\partial\Omega}=0\ \ (j=1,2)$ and we can easily update $\phi_{k2}^{n+1}$ and $\mu_{k2}^{n+1}$. Noticing that \eqref{system1_sub2.2} is the same as \eqref{system1_sub1.2}, we find \eqref{system1_sub2.3} is also solvable due to $\phi_{k22}=\phi_{k12}$, resulting in $\phi_{k2}^{n+1}$ and $\mu_{k2}^{n+1}$.
\par
$\bf Step\ \ 3$ \label{Step3}
We solve $q^{n+1}$ in \eqref{eqn:rrrmodel1_g}. From \eqref{eqn:rrrmodel1_h}, we get
\begin{equation}
\label{q_n+1_theta}
    q^{n+\theta}=\theta q^{n+1}+(1-\theta)q^n \Rightarrow q^{n+1}=\frac{q^{n+\theta}-(1-\theta)q^n}{\theta}.
\end{equation}
Substituting $q^{n+1}$ into \eqref{eqn:rrrmodel1_g} and using the linear combination forms for the variables $\mu_k^{n+1}$ and ${\bf \tilde u}^{n+1}$
in \eqref{system1_0},  \eqref{eqn:rrrmodel1_g} is formulated into the following form:
\begin{equation}
\label{system1_sub5}
    \left(\frac{2\theta+1}{2\theta\Delta t}-\eta_1\right)q^{n+\theta}=\frac{2\theta^2+\theta+1}{2\theta\Delta t}q^{n}-\frac{2\theta-1}{2\Delta t}q^{n-1}+\eta_2,
\end{equation}
where
\begin{align*}
    \eta_1=&\sum_{k=1}^{N}
    \int_{\Omega}\nabla \cdot ({\bf u}^{*}\phi_k^{*})(\theta\mu_{k2}^{n+1})d{\bf x}
    +\sum_{k=1}^{N}
    \int_{\Omega}\phi_k^{*}\nabla \mu_k^{*}\cdot (\theta{\bf \tilde u}_2^{n+1})d{\bf x}
    +\int_{\Omega}{\bf u}^{*}\cdot  \nabla{\bf u}^{*}\cdot (\theta{\bf \tilde u}_2^{n+1})d{\bf x},\\
    \eta_2
    =&\sum_{k=1}^{N}
    \int_{\Omega}\nabla \cdot ({\bf u}^{*}\phi_k^{*})(\theta\mu_{k1}^{n+1}+(1-\theta)\mu_k^n)d{\bf x}
    +\sum_{k=1}^{N}
    \int_{\Omega}\phi_k^{*}\nabla \mu_k^{*}\cdot (\theta{\bf \tilde u}_1^{n+1}+(1-\theta){\bf u}^n)d{\bf x}\\
    &+\int_{\Omega}{\bf u}^{*} \cdot \nabla{\bf u}^{*} \cdot (\theta{\bf \tilde u}_1^{n+1}+(1-\theta){\bf u}^n)d{\bf x}.\nonumber
\end{align*}
Obviously, all terms contained in $\eta_1$ and $\eta_2$ have been obtained in $\bf Step\ \ 1$ and $\bf Step\ \ 2$. Thus, once \eqref{system1_sub5} is solvable, we can update $q^{n+\theta}$ in an explicit way, which further allows us to get $q^{n+1}$, $\phi_k^{n+1}$, $\mu_k^{n+1}$, $r^{n+1}$ and ${\bf \tilde u}^{n+1}$.\par
Next, we prove the solvability of the equation \eqref{system1_sub5}. From \eqref{system1_sub2}, by taking the $L^2$-inner products of the first, second, and third equation with $\theta\mu_{k2}^{n+1}$, $\frac{2\theta+1}{2\Delta t}\phi_{k2}^{n+1}$ and $2\theta r_2^{n+1}$, respectively, we obtain
\begin{align}
    &\left(\frac{\frac{2\theta+1}{2}\phi_{k2}^{n+1}}{\Delta t},\theta\mu_{k2}^{n+1}\right)
    +(\nabla \cdot ({\bf u}^{*}\phi_k^{*}),\theta\mu_{k2}^{n+1})
    +M\theta^2\left\|\mu_{k2}^{n+1}\right\|^2=0,
    \label{system1_sub2_1}\\
    &\left(\theta\mu_{k2}^{n+1},\frac{2\theta+1}{2\Delta t}\phi_{k2}^{n+1}\right)
    =\lambda\theta \frac{2\theta+1}{2\Delta t}
    (-\Delta\phi_{k2}^{n+1},
    \phi_{k2}^{n+1})
    +\lambda\theta\left(\bar{H}_k^{*}\theta r_2^{n+1},\frac{2\theta+1}{2\Delta t}\phi_{k2}^{n+1}\right),
    \label{system1_sub2_2}\\
    &\frac{(2\theta+1)\theta}{\Delta t}\left |r_{2}^{n+1}\right |^2
    =\sum_{k=1}^{N}\int_{\Omega}\bar{H}_k^{*}\frac{\frac{2\theta+1}{2}\phi_{k2}^{n+1}}{\Delta t}\theta r_2^{n+1}d{\bf x},
    \label{system1_sub2_3}
\end{align}
where the term $\lambda\theta(\gamma^{*}\theta r_2^{n+1},\frac{2\theta+1}{2\Delta t}\phi_{k2}^{n+1})$ is omited since $\int_{\Omega}\phi_{k2}^{n+1}d{\bf x}=0$ has already been demonstrated in $\bf Step\ \ 1$. Combining the above equations, we have
\begin{align}
    -\sum_{k=1}^N
    (\nabla \cdot ({\bf u}^{*}\phi_k^{*}),\theta\mu_{k2}^{n+1})
    =&M\theta^2\left\|\mu_{k2}^{n+1}\right\|^2
    +\frac{\lambda\theta(2\theta+1)}{2\Delta t}
    \left\|\nabla\phi_{k2}^{n+1}\right\|^2
    +\frac{\lambda(2\theta+1)\theta^2}{\Delta t}\left |r_{2}^{n+1}\right |^2\geq 0.
    \label{system1_sub2_12}
\end{align}
Taking the $L^2$-inner product of \eqref{system1_sub4} with $\theta {\bf \tilde u}_{2}^{n+1}$ gets
\begin{equation}
    -({\bf u}^{*}\cdot \nabla {\bf u}^{*},\theta {\bf \tilde u}_{2}^{n+1})
    -\sum_{k=1}^{N}(\phi_k^{*}\nabla \mu_k^{*},\theta {\bf \tilde u}_{2}^{n+1})
    =\frac{(2\theta+1)\theta}{2\Delta t}\left\|{\bf \tilde u}_{2}^{n+1}\right\|^2
    +\nu\theta^2 \left\|\nabla {\bf \tilde u}_{2}^{n+1}\right\|^2\geq 0.
\end{equation}
According to the above two equations, we obtain $-\eta_1\geq 0$, and then $\frac{2\theta+1}{2\theta\Delta t}-\eta_1>0$ always holds, which indicates that \eqref{system1_sub5} is solvable.
\par
$\bf Step\ \ 4$ \label{Step4}
From Remark \ref{remark_Poisson}, we can update $p^{n+1}$ and ${\bf u}$ since ${\bf \tilde u}^{n+1}$ has been already computed in $\bf Step\ \ 3$.
\par
Thus, the scheme \eqref{eqn:rrrmodel1_a}-\eqref{eqn:rrrmodel1_g} is solvable.

\subsection{The scheme \eqref{eqn:rrrmodel2_a}-\eqref{eqn:rrrmodel2_g}}
\label{A.2}
Since the equations \eqref{eqn:rrrmodel2_a}-\eqref{eqn:rrrmodel2_c} are similar to \eqref{eqn:rrrmodel1_a}-\eqref{eqn:rrrmodel1_c}, we omit the calculation of them and only focus on \eqref{eqn:rrrmodel2_d}-\eqref{eqn:rrrmodel2_f}. The specific processes are as follows:
\par
$\bf Step\ \ 1$ \
We decompose ${\bf \tilde u}^{n+1}$ into the following form by using $q^{n+\theta}$
\begin{equation}
    {\bf \tilde u}^{n+1}={\bf \tilde u}_1^{n+1}+q^{n+\theta}{\bf \tilde u}_2^{n+1}.
\end{equation}
Substituting the above equation into \eqref{eqn:rrrmodel2_d} and grouping the terms with $q^{n+\theta}$ and without $q^{n+\theta}$, we can obtain
\begin{align*}
    &\tau\frac{\frac{2\theta+1}{2}{\bf \tilde u}_{1}^{n+1}-2\theta{\bf u}^n+\frac{2\theta-1}{2}{\bf u}^{n-1}}{\Delta t}
    +\alpha \nu (\theta {\bf \tilde u}_{1}^{n+1}+(1-\theta){\bf u}^n)+\nabla p^n=0,\\
    &\tau\frac{\frac{2\theta+1}{2}{\bf \tilde u}_{2}^{n+1}}{\Delta t}
    +\alpha \nu (\theta {\bf \tilde u}_{2}^{n+1})
    +\sum_{k=1}^{N}\phi_k^{*}\nabla \mu_k^{*}=0,
\end{align*}
that is,
\begin{align}
    &\left(\tau\frac{2\theta+1}{2\Delta t}
    +\alpha\nu\theta\right)
    {\bf \tilde u}_{1}^{n+1}
    =\left(\tau\frac{2\theta}{\Delta t}
    -\alpha\nu(1-\theta)\right)
    {\bf u}^{n}
    -\tau\frac{2\theta-1}{2\Delta t}{\bf u}^{n-1}
    -\nabla p^n,
    \label{system2_sub1}\\
    &\left(\tau\frac{2\theta+1}{2\Delta t}
    +\alpha\nu\theta\right)
    {\bf \tilde u}_{2}^{n+1}
    =-\sum_{k=1}^{N}\phi_k^{*}
    \nabla \mu_k^{*}.
    \label{system2_sub2}
\end{align}
Clearly, \eqref{system2_sub1} and \eqref{system2_sub2} are both solvable since $\tau\frac{2\theta+1}{2\Delta t}
+\alpha\nu\theta>0$ always holds. Then ${\bf \tilde u}_{1}$ and ${\bf \tilde u}_{2}$ can be easily obtained in an explicit way.
\par
$\bf Step\ \ 2$ \
From \eqref{q_n+1_theta}, by replacing $q^{n+1}$ with $\frac{q^{n+\theta}-(1-\theta)q^n}{\theta}$, we have
\begin{equation}
\label{system2_sub3}
    \left(\frac{2\theta+1}{2\theta\Delta t}-\eta_1\right)q^{n+\theta}=\frac{2\theta^2+\theta+1}{2\theta\Delta t}q^{n}-\frac{2\theta-1}{2\Delta t}q^{n-1}+\eta_2.
\end{equation}
Here, the expressions of $\eta_1$ and $\eta_2$ are as follows:
\begin{align}
    &\eta_1=\sum_{k=1}^{N}
    \int_{\Omega}\nabla \cdot ({\bf u}^{*}\phi_k^{*})(\theta\mu_{k2}^{n+1})d{\bf x}
    +\sum_{k=1}^{N}
    \int_{\Omega}\phi_k^{*}\nabla \mu_k^{*}\cdot (\theta{\bf \tilde u}_2^{n+1})d{\bf x},
    \nonumber\\
    &\begin{aligned}
        \eta_2
        =&\sum_{k=1}^{N}
        \int_{\Omega}\nabla \cdot ({\bf u}^{*}\phi_k^{*})(\theta\mu_{k1}^{n+1}+(1-\theta)\mu_k^n)d{\bf x}
        +\sum_{k=1}^{N}
        \int_{\Omega}\phi_k^{*}\nabla \mu_k^{*}\cdot (\theta{\bf \tilde u}_1^{n+1}+(1-\theta){\bf u}^n)d{\bf x}.
    \end{aligned}\nonumber
\end{align}
In order to illustrate the equation is solvable, by taking the $L^2$-inner product of \eqref{system2_sub2} with $\theta {\bf \tilde u}_{2}^{n+1}$, we obtain
\begin{equation}
    -\sum_{k=1}^{N}(\phi_k^{*}
    \nabla \mu_k^{*},\theta {\bf \tilde u}_{2}^{n+1})
    =\left(\tau\frac{2\theta+1}{2\Delta t}
    +\alpha\nu\theta\right)\theta
    \left\|{\bf \tilde u}_{2}^{n+1}\right\|^2\geq 0.
\end{equation}
From \eqref{system1_sub2_1}-\eqref{system1_sub2_12}, we still have
$-\sum_{k=1}^N(\nabla \cdot ({\bf u}^{*}\phi_k^{*}),\theta\mu_{k2}^{n+1})\geq 0$, which further leads to $\frac{2\theta+1}{2\theta\Delta t}-\eta_1>0$. Therefore, we can obtain $q^{n+\theta}$ in an explicit way and then update ${\bf \tilde u}^{n+1}$.
\par
$\bf Step\ \ 3$ \
Applying the divergence operator to \eqref{eqn:rrrmodel2_e} and using \eqref{eqn:rrrmodel2_f}, we can still get the following Poisson equation which is similar to \eqref{eqn:Poisson}
\begin{equation}
\label{eqn:Poisson2}
    \Delta p^{n+1}=\Delta p^n+\tau\frac{2\theta+1}{2\theta\Delta t}\nabla \cdot {\bf \tilde u}^{n+1}.
\end{equation}
On the basis of Remark \ref{remark_Poisson},  we gain the unique solution of \eqref{eqn:Poisson2}. Since all variables are given or computed, we can easily obtain $\bf u^{n+1}$ from \eqref{eqn:rrrmodel2_e}, i.e.,
\begin{equation}
    {\bf u}^{n+1}=-\frac{2\theta\Delta t}{\tau(2\theta+1)}(\nabla p^{n+1}-\nabla p^{n})+{\bf \tilde u}^{n+1}.
\end{equation}
\par
Thus, the scheme \eqref{eqn:rrrmodel2_a}-\eqref{eqn:rrrmodel2_g} is also solvable.

\bibliographystyle{elsarticle-num}
\bibliography{thebib}

\end{document}